\documentclass[10pt]{amsart}
\usepackage{ amssymb} 
\usepackage{ mathrsfs} 
\usepackage{ amscd} 
\usepackage{ enumitem}

\allowdisplaybreaks[4]
\usepackage{ascmac}

\theoremstyle{definition}
\numberwithin{equation}{section}
\newtheorem{thm}{Theorem}[section]

\newtheorem{prop}[thm]{Proposition}
\newtheorem{cor}[thm]{Corollary}
\newtheorem{lem}[thm]{Lemma}
\newtheorem{conj}[thm]{Conjecture}
\newtheorem{rem}[thm]{Remark}

\title{A cohomological interpretation of archimedean zeta integrals  for ${\rm GL}_3\times {\rm GL}_2$}
\author[T. Hara and K. Namikawa]{Takashi Hara and Kenichi Namikawa}

\address{ Department of Mathematics, College of Liberal Arts,  Tsuda University, 2-1-1 Tsuda-machi, Kodaira City, Tokyo 187-8577, Japan }
\email{t-hara@tsuda.ac.jp}

\address{ Faculty of Mathematics, Kyushu University, 744 Motooka, Nishi-Ku, Fukuoka, 819-0395, Japan }
\email{namikawa@math.kyushu-u.ac.jp}

\subjclass[2010]{Primary 11F67,  Secondary 11F75, 11F70}

\begin{document}

\begin{abstract}
By studying an explicit form of the Eichler--Shimura map for ${\rm GL}_3$,   
    we describe a precise relation between critical values 
    of the complete $L$-function for the Rankin--Selberg convolution ${\rm GL}_3 \times {\rm GL}_2$ 
    and the cohomological cup product 
    of certain rational cohomology classes which are uniquely determined up to rational scalar multiples 
    from the cuspidal automorphic representations 
    under consideration. This refines rationality results on critical values due to Raghuram et al. 
\end{abstract}

\maketitle

\tableofcontents

\section{Introduction}

The relation between certain cohomology classes and special values of $L$-functions 
   is a fascinating subject in number theory.  
Some of special values of $L$-functions can be interpreted in terms of certain cohomology classes, 
    and rationality of the special values is often deduced from rationality of those classes.     
In \cite{del79}, Deligne introduces the notion of {\em periods} associated to pure motives by using the comparison isomorphisms between their Betti and de Rham realizations, 
    and he makes a conjecture which yields that critical values of the $L$-functions associated to them coincide with his periods up to multiplication of algebraic numbers. Meanwhile, Clozel conjectures the existence of motives associated to   cohomological cuspical automorphic representations of ${\rm GL}_n$ in \cite[Conjecture 4.5]{clo90}, which provides us with motivic interpretation of automorphic $L$-functions concerning the general linear groups.
In the present article, we focus on the critical values of the Rankin--Selberg $L$-functions associated to cohomological cuspidal automorphic representations of ${\rm GL}_3 \times {\rm GL}_2$, and provide their precise cohomological interpretation.

Rationality of critical values of the Rankin--Selberg $L$-functions is firstly studied by Manin \cite{man72} and Shimura \cite[Theorems 1 and 4]{shi76} for ${\rm GL}_{2}\times {\rm GL}_1$. 
Generalizing these classical works, 
Mahnkopf  considers in \cite{mah05} rationality of critical values of the Rankin--Selberg $L$-functions of ${\rm GL}_{n+1}\times {\rm GL}_n$ for general $n$, 
    based on the {\em generalized modular symbol method} proposed by Kazhdan, Mazur and Schmidt in \cite{kms00}.    
Meanwhile, Raghuram and Shahidi introduce the notion of {\em Whittaker periods} in \cite[Definition/Proposition 3.3]{rs08} for irreducible cuspidal automorphic representations $\pi$ of ${\rm GL}_n(F_{\mathbf A})$, where $F$ is a number field.    
They define their periods by using the comparison isomorphisms between 
       the Betti and de Rham cohomology groups of the symmetric spaces attached to ${\rm GL}_n$, which are regarded as substitutes of Clozel's conjectural motives associated to $\pi$.
Furthermore, Raghuram shows that the critical values of the Rankin--Selberg $L$-functions for ${\rm GL}_{n+1}\times {\rm GL}_n$ 
              coincide with, up to multiplication of algebraic numbers, the product of Raghuram--Shahidi's Whittaker periods and the inverse of certain archimedean zeta integrals  (see \cite[Theorem 2.50]{rag16} for details). 
This implies that one should precisely study the archimedean zeta integrals appearing as the equality \cite[(2.47)]{rag16} in order to discuss rationality of critical values of the Rankin--Selberg $L$-functions with respect to Raghuram--Shahidi's periods.  Recently Sun has proved in \cite{sun17} that these archimedean zeta integrals do not vanish, and Januszewski has studied in \cite{jan19} their behaviors under Tate twists. However, as far as the authors know, what is known on these archimedean zeta integrals is fairly little except for Sun and Januszewski's results. 
The explicit study of the archimedean zeta integrals is essentially important even for the study of the $p$-adic Rankin--Selberg $L$-functions for ${\rm GL}_{n+1} \times {\rm GL}_n$  made in \cite{jan},  
since it is necessary to establish certain Kummer-type congruences (which are introduced as the {\em Manin congruences} in \cite{jan}). 
In this article, based on the recent precise study of archimedean zeta integrals for $\mathrm{GL}_3\times \mathrm{GL}_2$ made by Hirano, Ishii and Miyazaki in \cite{him}, we show that the archimedean zeta integrals under consideration coincide with the multiples of the desired $\Gamma$-factors with certain subtle but explicit constants when $n$ equals $2$. Our explicit study of the archimedean zeta integrals should be considered as a refinement of Raghuram, Sun and Januszewski's works, and we expect that it will be useful for the study of the $p$-adic Rankin--Selberg $L$-functions for ${\rm GL}_3 \times {\rm GL}_2$.

In order to write down the precise statement of the main theorem in this article, we first review the critical values of the Rankin--Selberg $L$-functions for $\mathrm{GL}_3\times \mathrm{GL}_2$.  Put $\Sigma_{\mathbf Q}$ as the set of all places of ${\mathbf Q}$, and denote by $\infty$ its unique infinite place. 
For $n=2$ or $3$, let $\pi^{(n)} = \bigotimes^\prime_{v\in \Sigma_{\mathbf Q}} \pi^{(n)}_v$ denote a cohomological irreducible cuspidal automorphic representation of ${\rm GL}_{n}({\mathbf Q}_{\mathbf A})$.
Since $\pi^{(2)}$ is cohomological, the archimedean part
  $\pi^{(2)}_{\infty}$ of $\pi^{(2)}$ is a discrete series representation $D_{\nu_2, l_2 }$ of ${\rm GL}_2({\mathbf R})$, and (the archimedean part of) the central character $\omega_{\pi^{(2)}}$ of $\pi^{(2)}$
     satisfies $\omega_{\pi^{(2)}, \infty}(t) = t^{2\nu_2}$ for every $t\in {\mathbf R}^\times$ with $t>0$.
We further find that, since $\pi^{(3)}$ is also cohomological, the archimedean part $\pi^{(3)}_{\infty}$ of $\pi^{(3)}$ is a generalized principal series representation of ${\rm GL}_3({\mathbf R})$ parabolically induced from the direct product of a discrete series representation $D_{\nu_3, l_3}$ of ${\rm GL}_2({\mathbf R})$ and a character $\chi_{\nu_3}^{\delta}$ of ${\mathbf R}^\times$.
See Section \ref{sec:auto} for details on the parameters of these representations.  
Since our method used in this article is deeply based upon the modular symbol method adopted in \cite{mah05},  
  we always have to assume that the inequality $(1\leq)\, l_2 < l_3$ holds.

We prefer to use the language of the motive attached to $\pi^{(n)}$, 
   since it makes expositions concise. 
We let ${\mathcal M}[\pi^{(n)}]$ denote the (conjectural) motive over ${\mathbf Q}$ attached to $\pi^{(n)}$  
   satisfying 
\begin{align*}
  L(s, {\mathcal M}[\pi^{(n)}])
     =    L\left(s-\frac{n-1}{2}, \pi^{(n)}\right),   
\end{align*}   
although it is not yet constructed except for the case $n=2$.   
The expected properties of  ${\mathcal M}[\pi^{(n)}]$ is precisely written down in \cite{clo90}, 
     and we briefly summarize them in Section \ref{sec:auto,mot}. 
Note that our main result (Theorem~\ref{thm:mainIntro}) does not depend on conjectural hypotheses such as the existence of $\mathcal{M}[\pi^{(3)}]$.  
Define ${\mathcal M}(\pi^{(3)} \times \pi^{(2)}    )$ to be the tensor product
${\mathcal M}[\pi^{(3)}] \otimes {\mathcal M}[\pi^{(2)}]$ of motives. Then the complete Rankin--Selberg $L$-function $L(s,\pi^{(3)}\times \pi^{(2)})$ is interpreted as the (complete) $L$-function associated to the motive $\mathcal{M}(\pi^{(3)}\times \pi^{(2)})$; namely we obtain the identity of the $L$-functions  
\begin{align*}
   L( s,  {\mathcal M}(\pi^{(3)} \times \pi^{(2)})    )
   =  L\left(s-\frac{3}{2}, \pi^{(3)} \times  \pi^{(2)} \right).  
\end{align*}
Let us normalize the parameter $\nu_n$ as $\nu_n = -\frac{l_n}{2}  + \frac{n-1}{2}$, 
so that ${\mathcal M}(\pi^{(3)} \times \pi^{(2)})$ is pure of weight $l_2 + l_3$. 
Then, by studying the Hodge decomposition of the motive $\mathcal{M}(\pi^{(3)}\times \pi^{(2)})$, we find that $m\in \mathbf{Z}$ is a critical point of $L(s, {\mathcal M}(\pi^{(3)} \times \pi^{(2)}) )$ if and only if it satisfies the following inequalities
\begin{align*}
         \begin{cases}
             \dfrac{l_3}{2}   + 1 \leq m \leq \dfrac{l_3}{2} + l_2   &  \text{when }l_2 \leq  \dfrac{l_3}{2} \text{ holds},  \\ 
             l_2   + 1 \leq m \leq l_3, &    \text{when } \dfrac{l_3}{2} < l_2 < l_3\text{ holds}. 
         \end{cases}  
\end{align*}   
See also the argument at the end of Section~\ref{sec:motives}. In the rest of Introduction, we always suppose that $m$ satisfies the above condition.

Now let us briefly recall the strategy to study critical values $L(m, {\mathcal M}(\pi^{(3)} \times \pi^{(2)}) )$ via the (generalized) modular symbol method utilized in \cite{kms00,mah05,rag16} and \cite{jan}.   
Let ${\mathcal K}_n$ be an open compact subgroup of ${\rm GL}_n(\widehat{\mathbf Z})$ 
   and $Y^{(n)}_{\mathcal K_n}$ the corresponding symmetric space   
   ${\rm GL}_n({\mathbf Q})  \backslash  {\rm GL}_n(  {\mathbf Q}_{\mathbf A}  ) / {\mathbf R}^\times_{>0}{\rm SO}_n({\mathbf R}) {\mathcal K}_n$.    
Suppose that ${\mathcal K}_2 = {\rm GL}_2(\widehat{\mathbf Z})$  
   and that ${\mathcal K}_3$ is the mirabolic subgroup of  ${\rm GL}_3(\widehat{\mathbf Z})$ consisting of matrices 
   whose bottom rows are congruent to $(0, 0, 1)$ modulo an ideal ${\mathfrak N}$ of $\widehat{\mathbf Z}$ (see also Remark~\ref{rem:tameint} on the assumption $\mathcal{K}_2=\mathrm{GL}_2(\widehat{\mathbf{Z}})$).                   
We also assume that $\pi^{(2)}$ has a ${\mathcal{K}}_2$-fixed vector and that $\mathfrak{N}$ is minimal among ideals of $\widehat{\mathbf{Z}}$ such that $\pi^{(3)}$ has a ${\mathcal K}_3$-fixed vector. Since $\pi^{(n)}$ is cohomological, we can define cohomology classes $\eta_{\pi^{(3)}}$ and $\eta^\pm_{\pi^{(2)}}$ 
    in certain cohomology groups of rational local systems on $Y^{(3)}_{\mathcal K_3}$ and $Y^{(2)}_{\mathcal K_2}$ respectively (see Section~\ref{sec:per} for the definition of these classes).
Raghuram and Shahidi define their periods $\Omega_{\pi^{(3)}}$ and $\Omega^\pm_{ \pi^{(2)} }$ attached to $\pi^{(3)}$ and $\pi^{(2)}$ in \cite{rs08}  
   as the ratios of $\eta_{\pi^{(3)}}$ and $\eta^\pm_{\pi^{(2)}}$ to the images of appropriate cusp forms of $\pi^{(3)}$ and $\pi^{(2)}$ under the Eichler--Shimura maps $\delta^{(3)}$ and $\delta^{(2)}$ respectively. One of the notable features in this article is that we give an explicit construction of the Eichler--Shimura map $\delta^{(3)}$ for $\mathrm{GL}_3$ (see Section~\ref{sec:ES} and Appendix~\ref{sec:AppA} for details), which enables us to proceed computation of the archimedean local zeta integrals.  
Furthermore, let us consider the natural projection
\begin{align*}
 {\rm p}_n  & : {\mathcal Y}^{(n)}_{\mathcal K_n} := {\rm GL}_n({\mathbf Q})  \backslash  {\rm GL}_n(  {\mathbf Q}_{\mathbf A}  ) /   {\rm SO}_n({\mathbf R}) {\mathcal K}_n
                           \longrightarrow 
                        Y^{(n)}_{\mathcal K_n}.   
\end{align*}
By considering ${\rm GL}_2$ as a subgroup of ${\rm GL}_3$ via the embedding $g\mapsto 
\begin{pmatrix}
 g & \\ & 1
\end{pmatrix}$,   
the branching rule for irreducible algebraic representations of ${\rm GL}_3$ and ${\rm GL}_2$ induces  maps from local systems on ${\mathcal Y}^{(3)}_{\mathcal K_3}$ to those on ${\mathcal Y}^{(2)}_{\mathcal K_2}$; 
therefore, using the branching rule, we can define cohomology classes $\nabla^{\boldsymbol{n}_m} \iota^* {\rm p}^{\ast}_3 \eta_{\pi^{(3)}}$ 
   for $\boldsymbol{n}_m=(l_2-1,m-l_2-1)$.   
See Section \ref{sec:branch} for the definition of the map $\nabla^{\boldsymbol{n}}$ for $\boldsymbol{n}=(n_1,n_2)$ coming from the branching rule. 
Then the rationality of critical values $L(m, {\mathcal M}(\pi^{(3)} \times \pi^{(2)}) )$ is verified by a study of the cup product of two cohomology classes $\nabla^{\boldsymbol{n}_m} \iota^* {\rm p}^{\ast}_3 \eta_{\pi^{(3)}}$ and ${\rm p}^{\ast}_2 \eta_{\pi^{(2)}}$.

One of the difficulties concerning the modular symbol method 
   is to study the archimedean zeta integrals appearing in the unfolding 
     of the cup product  $I(m, \pi^{(3)}, \pi^{(2)}   )$ 
     of  $\nabla^{\boldsymbol{n}_m} \iota^* {\rm p}^{\ast}_3 \eta_{\pi^{(3)}}$  and ${\rm p}^{\ast}_2 \eta^\pm_{\pi^{(2)}}$; see the arguments in \cite[Section 2.5.3.6]{rag16}. 
In this article, by giving an effective form of the Eichler--Shimura maps for ${\rm GL}_3$      
     and using a detailed study of archimedean Whittaker functions of $\pi^{(n)}$ made in \cite{him}, 
  we refine the definition of Raghuram and Shahidi's periods $\Omega_{\pi^{(3)}}$ and  $\Omega^\pm_{ \pi^{(2)} }$, 
  and then give an explicit formula for the archimedean zeta integral appearing in the computation of the cup product pairing.   
The main result in this article is as follows:

\begin{thm}\label{thm:mainIntro}(=Theorem \ref{thm:main})
{\itshape
Suppose that $   (-1)^{m + \delta  + \frac{l_3}{2}}
                         = \pm 1$ holds.
Then we have 
\begin{align*}
       I(m, \pi^{(3)}, \pi^{(2)}   )
     =  (-1)^{\delta} 
        \sqrt{-1}^{ \frac{l_3}{2} -m +1}    
        \binom{ \frac{l_3}{2} - 1  }{ m-\frac{l_3}{2}-1 }  
                 \binom{  \frac{l_3}{2} - 1  }{  \frac{l_3}{2}+l_2-m }
          \frac{  L(m, {\mathcal M}(\pi^{(3)} \times \pi^{(2)})   ) }{ \Omega_{\pi^{(3)}} \Omega^\pm_{\pi^{(2)}} }
\end{align*}
where $\binom{a}{b}$ denotes the binomial coefficient $\frac{a!}{b!(a-b)!}$. 
If $  (-1)^{m + \delta  + \frac{l_3}{2} }
      \neq \pm 1$ holds, 
      we have  $I(m, \pi^{(3)}, \pi^{(2)}   )=0$.
}
\end{thm}

Theorem \ref{thm:mainIntro} is derived from explicit evaluations of archimedean zeta integrals in \cite[(2.49)]{rag16},     
  and hence this refines a formula proposed in \cite[Theorem 2.50]{rag16}.    
In particular, 
  Deligne's conjecture implies that $\Omega_{\pi^{(3)}} \Omega^\pm_{\pi^{(2)}}$ coincides with Deligne's period $c^+( {\mathcal M}(\pi^{(3)} \times \pi^{(2)}) (m)  )$ up to multiplication of an algebraic number, 
  and thus Theorem \ref{thm:mainIntro} clarifies a motivic background of Raghuram--Shahidi's  Whittaker period $\Omega_{\pi^{(3)}} \Omega^\pm_{\pi^{(2)}}$.           
See Remarks \ref{rem:relCP} and \ref{rem:jan19} for more details on the result in \cite{rag16} and the compatibility with Deligne's conjecture.  
See also Remark \ref{rem:RelDelPer} (and \cite{yos01}) on the motivic interpretation of each of $\Omega_{\pi^{(3)}}$ and $\Omega^\pm_{\pi^{(2)}}$.

We would also like to explain a motivation of Theorem \ref{thm:mainIntro} from a viewpoint of the study of the $p$-adic $L$-functions.
In a series of his papers, Januszewski constructs the $p$-adic Rankin--Selberg $L$-functions for ${\rm GL}_{n+1} \times {\rm GL}_n$ (see \cite{jan} and references therein for details).  
However, some of expected important properties proposed in Coates and Perrin-Riou's conjecture \cite{cp89}, \cite{coa89} 
    have not been  established yet.     
Among such kinds of properties, we take up problems on effects of choices of periods on the Manin conguruences. Indeed, the periods which  Januszewski introduces in 
\cite[Theorem A]{jan} are defined critical pointwisely, and they a priori depend on the choice of critical points $m$. Such dependence of periods on critical points also appears in \cite[Theorem~1.1]{rag16}  for the completely same reason.   
Meanwhile, 
it is important to determine  
   an explicit relation between cup product pairings and critical values 
      including the signature, powers of $\sqrt{-1}$ and so on,     
   to establish congruences between the critical values at different critical points. 
Theorem \ref{thm:mainIntro} exhibits this explicit relation and hence we expect that the formula in Theorem \ref{thm:mainIntro} will be useful for the further study of the $p$-adic Rankin--Selberg $L$-functions.   
We also mention that this point is also a motivation of Januszewski's study on {\em period relations} made in \cite{jan19},  
   and 
   Theorem \ref{thm:mainIntro} refines the main result in \cite{jan19} in the case of ${\rm GL}_3 \times {\rm GL}_2$.  
See Remark \ref{rem:jan19} for the relation between Januszewski's works and ours.

\

The organization of this article is as follows.
Section \ref{sec:not} is devoted to a preparation of necessary general notation. Section~\ref{sec:findim} is a summary on finite dimensional representations which we use throughout this article. We remark that here we propose not only general theory but also explicit models of representations under consideration. Moreover, we give a branching rule for the pair $({\rm GL}_3, {\rm GL}_2)$ in an explicit manner (see Section~\ref{sec:branch} for details), which is one of the keys in the present article.  In Section~\ref{sec:auto,mot}, we introduce our setting on automorphic representations for which we consider the Rankin--Selberg $L$-function, and then survey the theory of motives associated with $\mathrm{GL}_3\times \mathrm{GL}_2$ according to \cite{clo90}. 
The fundamental tool in this article is to give an appropriate form of the Eichler--Shimura map for ${\rm GL}_3$, 
    which is a map from the space of cusp forms in $\pi^{(3)}$ to a cohomology group of $Y^{(3)}_{\mathcal K_3}$ of degree 2 with certain local coefficients.   
In Section \ref{sec:ESmap}, we introduce notion of cusp forms in this article and the Eichler--Shimura maps for ${\rm GL}_2$ and ${\rm GL}_3$.
Although the Eichler--Shimura map for ${\rm GL}_3$ (resp.\ for ${\rm GL}_2$) to a cohomology group of degree $2$ (resp.\ of degree $1$) is sufficient for the purpose of this article,   
   we can also consider the Eichler--Shimura map for ${\rm GL}_3$ to a cohomology group of degree $3$.    
We shall sketch the construction of the Eichler--Shimura maps for $\mathrm{GL}_3$ into not only the cohomology groups of degree $2$ but also those of degree $3$ in Appendix \ref{sec:AppA}, because it has independent significance in itself beyond the scope of this article,  
We then summarize explicit formulas on differential forms and their connections to the global integrals in Section \ref{sec:zetaint}.
Some of the formulas are deduced from longsome but straightforward computations, which we postpone to Appendix \ref{sec:AppB} for the sake of readability of this article.
We finally propose the definition of periods $\Omega_{\pi^{(3)}}$ and $\Omega_{\pi^{(2)}}^\pm$, and complete the proof of Theorem \ref{thm:mainIntro} in Section \ref{sec:permain}. We also discuss a motivic interpretation of our periods and compare our results with the previous researches.

\

We end the introduction with a remark on the base field.   
The method in this article is given in the ad\`elic language, and 
    hence the main theorem (Theorem \ref{thm:mainIntro}) is generalized to the case where the base field is totally real in a straightforward way.   
We leave it to the reader to make the notation and the argument concise.

\section{Notation}\label{sec:not}

Denote by ${\mathbf Q}$ (resp.\ ${\mathbf Z}$, ${\mathbf R}$, ${\mathbf C}$) the rational number field 
   (resp.\ the ring of the rational integers, the real number field, the complex number field).     
Let ${\mathbf Q}_{\mathbf A}$ be the ring of ad\`eles of $\mathbf{Q}$
      and ${\mathbf Q}_{\mathbf A, {\rm fin}}$ the subring of finite ad\`eles.  
Write the set of all places of ${\mathbf Q}$ as $\Sigma_{\mathbf Q}$ and denote by $\infty$ its unique infinite place.   
For each $x=(x_v)_{v\in \Sigma_{\mathbf Q}} \in {\mathbf Q}_{\mathbf A}$, we put
$x_{\rm fin} = (x_v)_{v\in \Sigma_{\mathbf Q} \setminus \{\infty\} } \in {\mathbf Q}_{\mathbf A, {\rm fin}}$. 
Let $G$ be an algebraic group $G$ over ${\mathbf Q}$.    
Then, for each ad\`elic point $g\in G({\mathbf Q}_{\mathbf A})$, 
   we write the infinite (resp.\ finite) part of $g$ as $g_\infty$ (resp.\ $g_{\rm fin}$).      
We denote the $v$-th component of $g$ by $g_v$ for each $v\in \Sigma_{\mathbf Q}$. 
We always use similar notation for each ad\`elic object.
    
In this article, we use the terminology ``$L$-function'' to denote a complete one. For instance, for each motive ${\mathcal M}$ over ${\mathbf Q}$,   
the $L$-function $L(s, {\mathcal M})$ is defined to be the product of the $\Gamma$-factor $L_\infty(s, {\mathcal M})$ and the finite part $L_{\rm fin}(s, {\mathcal M} )$.

\section{Models of finite dimensional representations}\label{sec:findim}

In this section, we summarize classifications and models of finite dimensional representations appearing in this article. We also introduce isomorphisms inducing so-called ``the branching rules'' in explicit forms.
In particular,  
     we follow \cite[Section 3.1, 4.1]{him} on the descriptions of representation spaces and the branching rule for the orthogonal groups introduced in Sections~\ref{sec:repSO2}, \ref{sec:repSO3} and \ref{sec:branch_orth}.

\subsection{Irreducible representations of ${\rm GL}_2$}\label{sec:repGL2}

Let $\mathcal{A}$ be a (commutative) integral domain of characteristic $0$. For ${\boldsymbol n}=(n_1,n_2)$ with $n_1 \in {\mathbf Z}, n_1\geq 0$ and $n_2\in {\mathbf Z}$, let $L^{(2)}( {\boldsymbol n}; {\mathcal A}) = {\mathcal A}[X,Y]_{n_1}$ be the set of homogenous polynomials of degree $n_1$ in variables $X$ and $Y$ with coefficients in $\mathcal{A}$. 
Define an action $\varrho^{(2)}_{{\boldsymbol n}}$ of ${\rm GL}_2({\mathcal A})$ on $L^{(2)}( {\boldsymbol n}; {\mathcal A})$ as follows: 
\begin{align*}
   \varrho^{(2)}_{{\boldsymbol n}}(g) P( X, Y  )    &= (\det g)^{n_2} P( (X,Y) g ) & \text{for } P\in L^{(2)}( {\boldsymbol n}; {\mathcal A}) \text{ and } g\in {\rm GL}_2({\mathcal A}).    
\end{align*}

When $\mathcal{A}$ is a field, it is well known that $(\varrho^{(2)}_{{\boldsymbol n}}, L^{(2)}( {\boldsymbol n}; {\mathcal A}) )$ gives an irreducible algebraic representation of ${\rm GL}_2({\mathcal A})$ of highest weight $(n_1+n_2, n_2)$. We finally define the central character $\omega_{\boldsymbol{n}} \colon \mathcal{A}^\times \rightarrow \mathcal{A}^\times$ of $L^{(2)}(\boldsymbol{n};\mathcal{A})$ so that $\varrho^{(2)}_{\boldsymbol{n}}(x1_2)P(X,Y)=\omega_{\boldsymbol{n}}(x)P(X,Y)$ holds for each $x\in \mathcal{A}^\times$ and $P(X,Y)\in L^{(2)}(\boldsymbol{n};\mathcal{A})$.

\subsection{Irreducible representations of ${\rm GL}_3$}\label{sec:repGL3}

Following \cite[Section 12]{fh13}, we can explicitly construct an irreducible algebraic representation $L^{(3)}( \boldsymbol{w}; {\mathbf C})$
of ${\rm GL}_3({\mathbf R})$ for ${\boldsymbol w}=(w_1^+, w_1^-, w_2)$, with $w_1^+, w_1^-, w_2\in {\mathbf Z}$ and $w_1^+, w_1^- \geq 0$.   
The highest weight of $L^{(3)}( {\boldsymbol w}; {\mathbf C})$ will be given by $(w_1^+ +w_2, w_2, -w_1^-+w_2)$. Later we mainly use this representation for the case where $w_1^+=w_1^-=w_2$ holds. Since the construction is purely algebraic, we construct $L^{(3)}({\boldsymbol w};\mathcal{A})$ for a general integral domain $\mathcal{A}$ of characteristic $0$.

Let $\mathcal{A}$ be a (commutative) integral domain of characteristic $0$ and set ${\boldsymbol w}=(w_1^+,w_1^-,w_2)$ with $w_1^+,w_1^-,w_2\in \mathbf{Z}$ and $w_1^+,w_1^-\geq 0$. Let ${\rm Std}_3 = {\mathcal A}^3$ be the standard representarion of ${\rm GL}_3({\mathcal A})$, 
 ${\rm Symm}^{w_1^+}{\rm Std}_3$ the $w_1^+$-th symmetric tensor product of ${\rm Std}_3$ 
 and ${\rm Symm}^{w_1^-}{\rm Std}^\vee_3$ the $w_1^-$-th symmetric tensor product of the contragradient of $\mathrm{Std}_3$. We identify ${\rm Symm}^{w_1^+}{\rm Std}_3 \otimes_{\mathcal A} {\rm Symm}^{w_1^-} {\rm Std}^\vee_3$ 
with the set ${\mathcal A}[X, Y,Z; A, B,C]_{w_1^+,w_1^-}$ of homogenous polynomials of degree $w_1^+$ in variables $X, Y, Z$ 
and of degree $w_1^-$ in variables $A, B, C$. 
Define an action $\varrho^{(3)}_{{\boldsymbol w}}$ of ${\rm GL}_3({\mathcal A})$ on ${\mathcal A}[X, Y,Z; A, B,C]_{w_1^+, w_1^-}$ as follows: 
\begin{align*}
   \varrho^{(3)}_{{\boldsymbol w}}(g) P( X, Y,Z; A, B, C  )    
      &=    (\det g)^{w_2} P( (X,Y, Z) g; (A,B,C){}^{\rm t}\!g^{-1} )  \\ 
         &  \qquad \qquad \text{for }P\in {\mathcal A}[X, Y,Z; A, B,C]_{w_1^+, w_1^-} \text{ and } g\in {\rm GL}_3({\mathcal A}).    
\end{align*}
We also define a differential operator $\iota_{w_1^+, w_1^-}$ to be 
\begin{align*}
\iota_{w_1^+, w_1^-}  = \frac{\partial^2}{ \partial X \partial A} + \frac{\partial^2}{ \partial Y \partial B} + \frac{\partial^2}{ \partial Z \partial C}  :
      {\mathcal A}[X, Y, Z; A, B, C]_{w_1^+,w_1^-} \to  {\mathcal A}[X, Y, Z; A, B, C]_{w_1^+-1,w_1^--1}.    
\end{align*}
We find that $\iota_{w_1^+,w_1^-}$ is equivariant with respect to the actions $\varrho_{{\boldsymbol w}}^{(3)}$ and $\varrho^{(3)}_{\boldsymbol{w}-(1,1,0)}$ of ${\rm GL}_3({\mathcal A})$. 
Then define $L^{(3)}( {\boldsymbol w}; {\mathcal A})$ as the kernel of $\iota_{w_1^+, w_1^-}$. If $\mathcal{A}$ is a field, we may readily verify that $\iota_{w_1^+,w_1^-}$ is surjective, and thus the dimension of $L^{(3)}({\boldsymbol w}; \mathcal{A})$ is calculated as follows:  
\begin{align}\label{eq:Vabdim}
\begin{aligned}
  {\rm dim}_{\mathcal A} L^{(3)}({\boldsymbol w}; {\mathcal A}) 
  =& {\rm dim}_{\mathcal A}{\mathcal A}[X,Y,Z; A, B, C]_{w_1^+, w_1^-} 
     - {\rm dim}_{\mathcal A}{\mathcal A}[X,Y,Z; A, B, C]_{w_1^+-1, w_1^- -1}     \\
 =& \binom{w_1^+ +2}{2} \binom{w_1^- +2}{2}  - \binom{w_1^+ +1}{2} \binom{w_1^- +1}{2}  = \frac{1}{2} (w_1^+ +1) (w_1^- +1) (w_1^+ +w_1^- +2).   
\end{aligned}
\end{align}
It is known that $L^{(3)}({\boldsymbol w};\mathcal{A})$ is an irreducible algebraic representation of ${\rm GL}_3({\mathcal A})$ when $\mathcal{A}$ is a field (see \cite[Section~13.2]{fh13} for details). We finally define the central character $\omega_{\boldsymbol{w}} \colon \mathcal{A}^\times \rightarrow \mathcal{A}^\times$ of $L^{(3)}(\boldsymbol{w};\mathcal{A})$ so that 
\begin{align} \label{eq:central_L3}
 \varrho^{(3)}_{\boldsymbol{w}}(x1_3)P(X,Y,Z;A,B,C)=\omega_{\boldsymbol{w}}(x)P(X,Y,Z;A,B,C)
\end{align}
holds for each $x\in \mathcal{A}^\times$ and $P(X,Y,Z;A,B,C)\in L^{(3)}(\boldsymbol{w};\mathcal{A})$.

\subsection{An explicit branching rule for $(\mathrm{GL}_3,\mathrm{GL}_2)$} \label{sec:branch}

Let us regard ${\rm GL}_2$ as a subgroup of ${\rm GL}_3$ via the fixed embedding 
\begin{align*}
    g\mapsto \iota(g) := \begin{pmatrix}   g & \\  & 1 \end{pmatrix}.    
\end{align*}
Then, it is well known that, for $w_1^+, w_1^-\in \mathbf{Z}$ with $w_1^+, w_1^-\geq 0$, 
an irreducible representation $L^{(3)}(w_1^+,w_1^-,0\, ; {\mathbf{C}})$ of $\mathrm{GL}_3(\mathbf{R})$ 
of highest weight $(w_1^+,0,-w_1^-)$ is decomposed as a $\mathrm{GL}_2({\mathbf{R}})$-representation as follows (see, for example, \cite[Theorem~8.1.1]{gw99}): 
\begin{align*}
L^{(3)}(w_1^+,w_1^-,0\,;\mathbf{C})\vert_{\mathrm{GL}_2(\mathbf{R})}
      \cong 
\bigoplus_{  -w_1^- \leq -j \leq 0 \leq i \leq w_1^+ } L^{(2)}(i+j,-j;\mathbf{C}).
\end{align*}
Note that $L^{(2)}(i+j,-j;\mathbf{C})$ is an algebraic irreducible representation of $\mathrm{GL}_2(\mathbf{R})$ of highest weight $(i,-j)$. Such a decomposition is called the {\em branching rule} for the pair $(\mathrm{GL}_3,\mathrm{GL}_2)$. 
 In this subsection, we will construct an explicit  $\mathrm{GL}_2$-equivariant isomorphism as above in terms of the models introduced in Sections~\ref{sec:repGL2} and \ref{sec:repGL3}, which would induce a decomposition of local systems pulled back to the symmetric space ${\mathcal Y}^{(2)}_{\mathcal K_2}=\mathrm{GL}_2(\mathbf{Q})\backslash \mathrm{GL}_2(\mathbf{Q_A})/\mathrm{SO}_2(\mathbf{R})\mathcal{K}_2$; see Section~\ref{sec:zetaint} for details.

Continuing from the previous subsection, let ${\mathcal A}$ be a (commutative) integral domain of characteristic $0$, and assume that $w_1^+!\, w_1^-!$ is invertible in ${\mathcal A}$.
For each $0\leq k \leq w_1^+$,  $0\leq l \leq w_1^-$ and $P\in {\mathcal A}[X,Y,Z; A, B,C]_{w_1^+, w_1^-}$, set 
\begin{align}\label{eq:defnab}
   (\nabla_{k,l} P) (X,Y)  = \frac{1}{k! \, l!} \frac{\partial^{k+l} P }{ \partial Z^k \partial C^l} (X, Y, 0\, ; -Y, X, 0).   
\end{align}
Then one sees that $\nabla_{k,l} P$ is a polynomial in variables $X$ and $Y$, which is homogenous of degree $w_1^+ + w_1^- -k-l$. 
The following lemma shows that $\nabla_{k,l} $ is a ${\rm GL}_2({\mathcal A})$-equivariant map: 

\begin{lem}\label{lem:nabkl}
{\itshape 
For each $g \in {\rm GL}_2({\mathcal A})$, we have 
\begin{align*}
    \left(  \nabla_{k,l}  \left(( \varrho^{(3)}_{w_1^+,w_1^-,0}(\iota(g))  P\right)  \right)(X, Y)    
    =   (\det g)^{-(w_1^--l)} \left( \nabla_{k,l} P    \right) ((X,Y)g).   
\end{align*}
In particular, $\nabla_{k,l}$ induces a ${\rm GL}_2({\mathcal A})$-equivariant homomorphism 
\begin{align*}
    \nabla_{k,l} \colon L^{(3)}(w_1^+,w_1^-,0\,;\mathcal{A})|_{\rm GL_2({\mathcal A})}  \longrightarrow L^{(2)}(w_1^+ +w_1^- -k-l,-(w_1^- -l); \mathcal{A}).
\end{align*}
}
\end{lem}
\begin{proof}
For $g=\begin{pmatrix} a & b \\ c & d  \end{pmatrix}\in {\rm GL}_2({\mathcal A})$, we find that 
\begin{align*}
   {}^{\rm t} \iota(g)^{-1} = (\det g)^{-1}  \begin{pmatrix}  d & -c   & 0 \\ -b & a & 0 \\ 0 & 0 & \det g  \end{pmatrix}. 
\end{align*}
Hence, by definition, we obtain the following equality: 
\begin{align*}
  \left( \varrho^{(3)}_{w_1^+,w_1^-,0} ( \iota(g) ) P  \right) (X,Y,Z;  A, B, C) 
    &= P( (X,Y,Z) \iota(g) ; (A,B,C){}^{\rm t}\iota(g)^{-1} )  \\
    &= (\det g)^{-w_1^-} P(aX+cY, bX+dY, Z  ;   dA-bB, -cA+aB, (\det g)C).   
\end{align*}
Using the chain rule, we can calculate the derivatives of $\varrho^{(3)}_{w_1^+,w_1^-,0} ( \iota(g) ) P$ as 
\begin{multline*}
       \frac{1}{k!\, l!}  \frac{\partial^{k+l}   }{ \partial Z^k \partial C^l   }     \left(    \left( \varrho^{(3)}_{w_1^+,w_1^-,0} ( \iota(g) ) P  \right) (X,Y,Z;  A, B, C)     \right)  \\
  =       \frac{1}{k!\,l!}  (\det g)^{-w_1^- +l}  \frac{\partial^{k+l}  P }{ \partial Z^k \partial C^l   }  (aX+cY, bX+dY, Z  ;   dA-bB, -cA+aB, (\det g)C).  
\end{multline*}
Substituting $A=-Y, B=X$ and $C=Z=0$, we finally obtain the desired equality:  
\begin{align*}
       \left(  \nabla_{k,l}  \left(( \varrho^{(3)}_{w_1^+,w_1^-,0}(\iota(g))  P\right)  \right)(X, Y)    
   = &     \frac{1}{k!\, l!}  (\det g)^{-w_1^- +l}  \frac{\partial^{k+l}  P }{ \partial Z^k \partial C^l   }  (aX+cY, bX+dY, 0  ;   -(bX+dY), aX+cY, 0 )  \\
   = &  (\det g)^{- (w_1^- -l)} (\nabla_{k,l} P) (aX+cY, bX+dY) \\
   = &  (\det g)^{-(w_1^- -l)}  (\nabla_{k,l} P) ( (X,Y) g ).   \qedhere
\end{align*}
\end{proof}

By Lemma \ref{lem:nabkl}, we can define a ${\rm GL}_2({\mathcal A})$-equivariant homomorphism
\begin{align*}
   {\boldsymbol \nabla}  = (\nabla_{k,l})_{\substack{0\leq k\leq w_1^+ \\ 0\leq l\leq w_1^-} }:  
       L^{(3)}(w_1^+,w_1^-,0\,; \mathcal{A}) |_{{\rm GL}_2( {\mathcal A} )  }   \longrightarrow \bigoplus^{w_1^+}_{k=0}   \bigoplus^{w_1^-}_{l=0} L^{(2)}(w_1^+ +w_1^- -k-l, -(w_1^- -l); \mathcal{A}).
\end{align*}

\begin{prop}
{\itshape  
Suppose that ${\mathcal A}$ is a field.
Then the homomorphism ${\boldsymbol \nabla}$ is an isomorphism of $\mathrm{GL}_2(\mathcal{A})$-modules.
}
\end{prop}

\begin{proof}
 Since we have shown the ${\rm GL}_2({\mathcal A})$-equivariance of ${\boldsymbol \nabla}$ in Lemma \ref{lem:nabkl}, 
 it suffices to show that the ${\mathcal A}$-dimension of the source of ${\boldsymbol \nabla}$ equals the ${\mathcal A}$-dimension of the target of ${\boldsymbol \nabla}$. 
The ${\mathcal A}$-dimensiton of the source has been already calculated in (\ref{eq:Vabdim}).  
 It is also easy to calculate the ${\mathcal A}$-dimension of the target as 
 \begin{multline*}
        \sum^{w_1^+}_{k=0}  \sum^{w_1^-}_{l=0}  {\rm dim}_{\mathcal A} L^{(3)}(w_1^+ +w_1^- -k-l, -(w_1^- -l); {\mathcal A})    \\[-.5em]
\begin{aligned}
   &=   \sum^{w_1^+}_{k=0}  \sum^{w_1^-}_{l=0}(w_1^+ +w_1^- -k-l+1)  \\
   &=   (w_1^+ +w_1^- +1)(w_1^+ +1)(w_1^- +1)  - \frac{1}{2} w_1^+ (w_1^+ +1)(w_1^- +1)   - (w_1^+ +1) \cdot  \frac{1}{2} w_1^- (w_1^-+1)  \\
   &= \frac{1}{2} (w_1^+ +1) (w_1^- +1) (w_1^+ +w_1^- +2).   
\end{aligned}
 \end{multline*}
 This completes the proof.   
\end{proof}

Taking the ($w_2$-th) Tate twists for $w_2\in \mathbf{Z}$ into accounts, we finally obtain the explicit decomposition 
\begin{align*}
   {\boldsymbol \nabla}  = (\nabla_{k,l})_{\substack{0\leq k\leq w_1^+ \\ 0\leq l\leq w_1^-} }:  
       L^{(3)}(w_1^+,w_1^-,w_2; \mathcal{A}) |_{{\rm GL}_2( {\mathcal A} )  }   \longrightarrow \bigoplus^{w_1^+}_{k=0}   \bigoplus^{w_1^-}_{l=0} L^{(2)}(w_1^+ +w_1^- -k-l, -(w_1^- -l)+w_2; \mathcal{A}).
\end{align*}

To lighten the notation, let us set ${\boldsymbol n}=(n_1,n_2)=(w_1^+ +w_1^- -k-l, -(w_1^- -l)+w_2)$ and 
\begin{align} \label{eq:nabkl}
 \nabla^{\boldsymbol{n}}=\nabla^{n_1,n_2}=\nabla_{w_1^+ +w_2 -n_1-n_2, w_1^- -w_2+n_2}.
\end{align}
Then the above decomposition is described as 
\begin{align} \label{eq:branch}
  {\boldsymbol \nabla}  = (\nabla^{{\boldsymbol n}})_{\boldsymbol{n}\in \Xi_2(\boldsymbol{w})}:  
       L^{(3)}(\boldsymbol{w}; \mathcal{A}) |_{{\rm GL}_2( {\mathcal A} )  }   \longrightarrow \bigoplus_{\boldsymbol{n}\in \Xi_2(\boldsymbol{w})} L^{(2)}(\boldsymbol{n}; \mathcal{A}),
\end{align}
where  
\begin{align*}
   \Xi_2(\boldsymbol{w}) 
  =& \left\{ (n_1, n_2) \in {\mathbf Z}  \mid    (n_1,n_2)
                  =   (w_1^++w_1^- -k-l,-(w_1^- -l)+w_2)  \text{\; for } 0\leq {}^\exists k\leq w_1^+,\,  0\leq {}^\exists l \leq w_1^-   
            \right\} \\
  =& \left\{  (n_1, n_2) \in {\mathbf Z} \mid     w_2 \leq n_1+n_2 \leq w_1^+ +w_2,\, w_2-w_1^- \leq n_2\leq w_2  \right\}.    
\end{align*}

\subsection{Irreducible representation of ${\rm O}_2$}\label{sec:repSO2}

Set $\Lambda_2 =\{  (0, 1) \} \cup \{ (\lambda, 0) \mid  \lambda \in {\mathbf Z}, \lambda \geq 0  \}$
and let $\boldsymbol{\lambda} = (\lambda, \delta)$ be an element of $\Lambda_2$. 
We define an action $\tau^{(2)}_{\boldsymbol{\lambda}}$ of $\mathrm{O}_2(\mathbf{R})$ on $\mathbf{C}[X,Y]_{\lambda}$ by $\tau^{(2)}_{\boldsymbol{\lambda}}=\varrho^{(2)}_{\boldsymbol{\lambda}} \vert_{\mathrm{O}_2(\mathbf{R})}$, and define 
$v^{(2),\boldsymbol{\lambda}}_{\pm \lambda}\in \mathbf{C}[X,Y]_{\lambda}$ as\footnote{Note that both $v^{(2),(0,0)}_0$ and $v^{(2),(0,1)}_0$ equal $1$ as an element of $ V^{(2)}_{(0,0)}=V^{(2)}_{(0,1)}=\mathbf{C}$, but ${\mathrm O}_2(\mathbf{R})$ acts on them in different manners.} 
\begin{align*}
   v^{(2), \boldsymbol{\lambda}}_{\pm \lambda}
&= (\pm X+\sqrt{-1} Y)^{\lambda}  &  \text{(double sign in the same order).}
\end{align*}
Then the action of $\mathrm{O}_2(\mathbf{R})$ on $v^{(2),\boldsymbol{\lambda}}_{\pm \lambda}$ is described as follows: 
\begin{align*}
        \tau^{(2)}_{\boldsymbol{\lambda}} \left( \begin{pmatrix} \cos\theta & \sin \theta \\ - \sin \theta & \cos \theta \end{pmatrix} \right)  
           v^{(2), \boldsymbol{\lambda}}_{\pm \lambda}  
    &=  e^{\pm\sqrt{-1} \lambda \theta}   
           v^{(2), \boldsymbol{\lambda}}_{\pm  \lambda},    &  
        \tau^{(2)}_{\boldsymbol{\lambda}} \left( \begin{pmatrix} -1 & 0 \\ 0 & 1 \end{pmatrix} \right) 
           v^{(2),\boldsymbol{\lambda}}_{\pm\lambda} 
    &=  (-1)^{ \delta}  v^{(2), \boldsymbol{\lambda}}_{\mp\lambda}. 
\end{align*}
Hence the subspace $V^{(2)}_{\boldsymbol{\lambda}} := \langle v^{(2),\boldsymbol{\lambda}}_{\lambda}, v^{(2),\boldsymbol{\lambda}}_{-\lambda}  \rangle_{\mathbf{C}}$ of $\mathbf{C}[X,Y]_{\lambda}$  
      gives an irreducible representation of ${\rm O}_2({\mathbf R})$, 
and hereafter we write $\tau^{(2)}_{\boldsymbol{\lambda}}$ to denote the action of ${\rm O}_2({\mathbf R})$ on this space. 
Then each irreducible representation of ${\rm O}_2({\mathbf R})$ is classified as 
  $\tau^{(2)}_{\boldsymbol{\lambda}}$ for some $\boldsymbol{\lambda} \in \Lambda_2$.

\subsection{Irreducible representation of ${\rm O}_3$}\label{sec:repSO3}

Set $\Lambda_3 = \{ (\lambda, \delta)   \mid   \lambda\in {\mathbf Z},\,  \lambda \geq 0,\, \delta \in \mathbf{Z}/2\mathbf{Z} \}$
 and let $\boldsymbol{\lambda}=(\lambda, \delta)$ be an element of $\Lambda_3$. 
We define an action $\tau^{(3)}_{\boldsymbol{\lambda}}$ of ${\rm O}_3({\mathbf R})$ on ${\mathbf C}[z_1,z_2,z_3]_{\lambda}$ by 
\begin{align*}
  \tau^{(3)}_{\boldsymbol{\lambda}} (u) P(z_1, z_2, z_3) =  (\det u)^{\delta} P( (z_1, z_2, z_3) u )
\end{align*}
for each  $u \in {\rm O}_3({\mathbf R})$ and $P\in {\mathbf C}[z_1, z_2, z_3]_{\lambda}$.
Then the subspace ${\mathcal V}_{\boldsymbol{\lambda}}$ defined as ${\mathcal V}_{\boldsymbol{\lambda}} = (z^2_1 + z^2_2 + z^2_3){\mathbf C}[z_1, z_2, z_3]_{\lambda -2}$ when $\lambda \geq 2$ and ${\mathcal V}_{\boldsymbol{\lambda}} = 0$ otherwise is stable under the action $\tau^{(3)}_{\boldsymbol{\lambda}}$ of ${\rm O}_3({\mathbf R})$. One may verifiy that 
         the quotient space $V^{(3)}_{\boldsymbol{\lambda}} = {\mathbf C}[z_1, z_2, z_3]_{\lambda} / {\mathcal V}_{\boldsymbol{\lambda}}$ gives rise to
         an irreducible representation of ${\rm O}_3({\mathbf R})$, which we also denote by $\tau^{(3)}_{\boldsymbol{\lambda}}$ (see also \cite[Section~4.1]{him}).   
Then each irreducible representation of ${\rm O}_3({\mathbf R})$ is classified as $\tau^{(3)}_{\boldsymbol{\lambda}}$ for some $\boldsymbol{\lambda}\in \Lambda_3$. It is easy to see that the dimension of $V^{(3)}_{\boldsymbol{\lambda}}$ is equal to $2\lambda+1$. 

Next we construct an explicit basis of the representation space $V^{(3)}_{\boldsymbol{\lambda}}$. For every  $\mu \in{\mathbf Z}$ with $ 0 \leq \mu \leq \lambda$, let $v^{(3),\boldsymbol{\lambda}}_{\pm \mu}$ be elements of $V^{(3)}_{\boldsymbol{\lambda}}$ defined as 
\begin{align*}
    v^{(3), \boldsymbol{\lambda}}_{\pm \mu}
    &=  ( \pm z_1+\sqrt{-1} z_2 )^{\mu}  z^{\lambda - \mu}_3
            \mod {\mathcal V}_{\boldsymbol{\lambda}} & \text{(double sign in the same order)}. 
\end{align*}
We abbreviate $  v^{(3), \boldsymbol{\lambda}}_{\pm \mu}$ to   $v^{ \boldsymbol{\lambda}}_{\pm \mu}$ if no confusion likely occurs. 
Then $\{v^{(3),\boldsymbol{\lambda}}_{\pm \mu} \mid 0\leq \mu\leq \lambda \}$ forms a basis of $V^{(3)}_{\boldsymbol{\lambda}}$. The multiplicativity among these elements
\begin{align} \label{eq:vmult}
 v^{(3),\boldsymbol{\lambda}}_{\pm \mu}\times  v^{(3),\boldsymbol{\lambda}^\prime}_{\pm \mu^\prime} &=  v^{(3),\boldsymbol{\lambda}+\boldsymbol{\lambda}^\prime}_{\pm\mu\pm\mu^\prime} & \text{for }\boldsymbol{\lambda}=(\lambda,\delta),\, \boldsymbol{\lambda}'=(\lambda',\delta')\in \Lambda_3 \text{ and } 0\leq \mu\leq \lambda,\,  0\leq \mu'\leq \lambda' 
\end{align}
is immediately deduced from their definition. 

For later use, we here introduce the matrix representation $M_{\boldsymbol{\lambda}}(u) \in {\rm GL}_{2\lambda+1}({\mathbf C})$ of the action of $u\in {\rm O}_3({\mathbf R})$ on $V^{(3)}_{\boldsymbol{\lambda}}$ with respect to the above specific basis; namely we set  
\begin{align}\label{eq:matM}
  \begin{pmatrix}
   \tau^{(3)}_{\boldsymbol{\lambda}} (u) v^{(3),\boldsymbol{\lambda}}_{\lambda} & \tau^{(3)}_{\boldsymbol{\lambda}}  (u)v^{(3),\boldsymbol{\lambda}}_{\lambda-1} & 
     \dotsc &
     \tau^{(3)}_{\boldsymbol{\lambda}}  (u)v^{(3),\boldsymbol{\lambda}}_{-\lambda}  
 \end{pmatrix}  
   = 
\begin{pmatrix}
 v^{(3),\boldsymbol{\lambda}}_{\lambda}& v^{(3),\boldsymbol{\lambda}}_{\lambda-1} & \dotsc & v^{(3),\boldsymbol{\lambda}}_{-\lambda}
\end{pmatrix}  M_{\lambda}(u).
\end{align}

\subsection{An explicit branching rule for $({\mathrm{O}_3,\mathrm{O}_2})$} \label{sec:branch_orth}

We finally emphasize that, by using the specific bases introduced in the previous subsections, we can explicitly describe the branching rule for $({\mathrm O}_3(\mathbf{R}), {\mathrm O}_2(\mathbf{R}))$ with respect to the fixed inclusion ${\mathrm O}_2(\mathbf{R})\hookrightarrow {\mathrm O}_3(\mathbf{R}); u \mapsto 
\begin{pmatrix}
 u & \\ & 1
\end{pmatrix}$ as follows. For each $\boldsymbol{\lambda}=(\lambda,\delta)\in \Lambda_3$, define a subset $\Sigma(\boldsymbol{\lambda})$ of $\Lambda_2$ as 
\begin{align*}
  \Sigma(\boldsymbol{\lambda}) = \left\{  (0, \delta)   \right\}  
                                  \cup \left\{ (\mu, 0)  
                                                        \mid \mu \in {\mathbf Z}, 1\leq \mu \leq \lambda   \right\} \quad \subseteq \Lambda_2.  
\end{align*}

For each $\boldsymbol{\mu}=(\mu,\delta) \in \Sigma(\boldsymbol{\lambda})$, the correspondence $v^{(2), \boldsymbol{\mu}}_{\pm \mu} \mapsto  (\pm 1)^{\delta} v^{(3), \boldsymbol{\lambda}}_{\pm \mu}$ defines an ${\mathrm O}_2(\mathbf{R})$-equivariant homomorphism $\iota^{\boldsymbol{\lambda}}_{\boldsymbol{\mu}} : V^{(2)}_{\boldsymbol{\mu}} \to V^{(3)}_{\boldsymbol{\lambda}}$  (see also \cite[Lemma 9.3]{him}), which induces an isomorphism of ${\mathrm O}_2(\mathbf{R})$-representations    
\begin{align*}
 \bigoplus_{ \boldsymbol{\mu} \in \Sigma(\boldsymbol{\lambda}) } V^{(2)}_{\boldsymbol{\mu}}  \stackrel{\sim}{\longrightarrow}   V^{(3)}_{\boldsymbol{\lambda}} \vert_{{\mathrm O}_2(\mathbf{R})}.  
\end{align*}
Obviously there also exists an ${\rm O}_2(\mathbf{R})$-homomorphism 
  ${\mathrm{P}}^{\boldsymbol{\lambda}}_{\boldsymbol{\mu}} \colon V^{(3)}_{\boldsymbol{\lambda}} \vert_{{\mathrm O}_2(\mathbf{R})} \to V^{(2)}_{\boldsymbol{\mu}}$ satisfying ${\rm P}^{\boldsymbol{\lambda}}_{\boldsymbol{\mu}} \circ \iota^{\boldsymbol{\lambda}}_{\boldsymbol{\mu}} =  \mathrm{id}_{V_{\boldsymbol{\mu}}^{(2)}}$.

\section{Motives associated to the Rankin--Selberg convolution ${\rm GL}_3 \times {\rm GL}_2$}\label{sec:auto,mot}

After briefly reviewing  the notion of cohomological automorphic representations of $\mathrm{GL}_n(\mathbf{Q}_{\mathbf{A}})$, we summarize in this section the  settings on the automorphic representations of ${\rm GL}_3({\mathbf Q}_{\mathbf A})$, ${\rm GL}_2({\mathbf Q}_{\mathbf A})$ and their Rankin--Selberg convolution, which we consider throughout this article. We also introduce the notion of their archimedean $L$-fuctors and critical points,  and discuss their relations to the (conjectural) associated motives, according to \cite{del79} and \cite{clo90}.
Some of materials are also dealt with in \cite{mah05} and \cite{rag16}. 
Beyond these generalities, we would emphasize here explicit models and parameters of the representations under consideration, which are fundamental tools in this article. We always consider that all the automorphic $L$-functions are normalized so that they satisfy the functional equations between $L(s,-)$ and $L(1-s,-^\vee)$ throughout this article (where $^\vee$ denotes the contragredient).

\subsection{Generalities on automorphic representations of ${\rm GL}_n$}\label{sec:auto}

In this subsection, we briefly review general theory on cohomological automorphic representations of ${\rm GL}_n$. 
Let $\pi^{(n)}$ be an irreducible cuspidal automorphic representation of ${\rm GL}_n({\mathbf Q}_{\mathbf A})$.    
We first recall the notion of the Langlands parameter of $\pi^{(n)}$ according to \cite[Section~2]{kna94} and \cite{clo90}. 
For this purpose, let us introduce several notation on representations of the Weil group $W_{\mathbf R}$ of the real number field ${\mathbf R}$. 

Here we realize the Weil group $W_{\mathbf R}$ as the disjoint union ${\mathbf C}^\times \sqcup ({\mathbf C}^\times j)$ where $j$ satisfies $j^{2}=-1$ and $jzj = -z^c$ for arbitrary $z \in {\mathbf C}^\times$, where $z^c$ denotes the complex conjugate of $z$.
Then, for every $\nu \in {\mathbf C}$, $\delta\in \{0, 1\}$ and $l\in {\mathbf Z}$ with $l\geq 0$, 
define a $1$-dimensional representation $\phi^\delta_\nu: W_{\mathbf R} \to {\mathbf C}^\times$ 
and a $2$-dimensional representation $\phi_{\nu, l} : W_{\mathbf R} \to {\rm GL}_2({\mathbf C})$ of $W_{\mathbf{R}}$ as follows:
\begin{align*}
     \phi^\delta_\nu (z) &= (zz^c)^{\nu} \hphantom{\begin{pmatrix}  (z^c/z)^{\frac{l}{2}}  &  0 \\  0 &  (z/z^c)^{\frac{l}{2}}  \end{pmatrix}} \;\quad   \text{for }z\in {\mathbf C}^\times, &
                                       \phi^\delta_\nu(j) &= (-1)^\delta,   \\   
     \phi_{\nu, l} (z)  &=  (zz^c)^\nu \begin{pmatrix}   (z^c/z)^{\frac{l}{2}}  &  0 \\  0 &  (z/z^c)^{\frac{l}{2}}  \end{pmatrix} \quad  \text{for }z \in {\mathbf C}^\times, & 
                                        \phi_{\nu, l} (j) &=  \begin{pmatrix}  0 & (-1)^l  \\ 1 & 0 \end{pmatrix}. 
\end{align*} 
Let $\chi^\delta_\nu : {\mathbf R}^\times \to {\mathbf C}^\times$  be a character of $\mathbf{R}^\times$ defined by $\chi^\delta_\nu(t) = {\rm sgn}(t)^\delta |t|^\nu_\infty$, and  $D_{\nu, l}$ an irredicible representation of $\mathrm{GL}_2(\mathbf{R})$ satisfying  
\begin{align*}
  D_{\nu, l}(t1_2) &= t^{2\nu} \quad \text{for } t \in {\mathbf R}^\times \text{ with } t>0, & 
  D_{\nu, l}|_{ {\rm SL}_2({\mathbf R}) }  &\cong D^+_l \oplus D^-_l,   
\end{align*}
where $D^+_l$ (resp.\ $D^-_l$) denotes the holomorphic (resp.\ anti-holomorphic) discrete series representation of ${\rm SL}_2({\mathbf R})$  of lowest weight $l+1$ (resp.\ of highest weight $-l-1$). 
Then $\phi^\delta_\nu$ (resp.\ $\phi_{\nu, l}$) corresponds to $\chi^\delta_\nu$ (resp.\ $D_{\nu, l}$) via the (archimedean) local Langlands correspondence; see \cite[Theorem~2]{kna94} for details.

It is well known that, when the automorphic representation $\pi^{(n)}$ under consideration is {\em cohomological} (or {\em regular algebraic} in some literature), strong constraints are imposed on the Langlands parameter of its archimedean part. We do not review here the precise definition of cohomological automorphic representations, and just quote the following classification results which are necessary later. For details on cohomological automorphic representations, see Clozel's exposition \cite[Section~3.5]{clo90} (refer also to \cite[Section~3.1.1]{mah05} and \cite[Section~2.3]{gr14} for brief summaries). 

\begin{prop}\label{prop:CloCoh}(\cite[
Lemme de puret\'{e} 4.9]{clo90})
{\itshape 
Let $\pi^{(n)}$ be a cohomological irreducible cuspidal automorphic representation of $\mathrm{GL}_n(\mathbf{Q}_{\mathbf{A}})$. 
Then the Langlands parameter of the archimedean part $\pi^{(n)}_\infty$ of $\pi^{(n)}$ is described as follows$:$
\begin{enumerate}[label={\rm (\roman*)}]
 \item  When $n$ is even, set $n=2m$.    
             Then the Langlands parameter of $\pi^{(n)}$ is given by 
             \begin{align*}
               \bigoplus^m_{i=1} \phi_{\nu_i,l_i} \quad \text{with } l_1> l_2>\cdots > l_m              
             \end{align*}
such that  $\nu_i  \pm \dfrac{l_i}{2}-\dfrac{n-1}{2}$ are rational integers for each $1\leq i\leq m$. Furthermore, the purity implies that 
                      \begin{align*}
            w:= 2\nu_i-n+1  \quad \left(= \left(  \nu_i-\frac{l_i}{2} -\frac{n-1}{2}  \right)  +  \left( \nu_i + \frac{l_i}{2} -\frac{n-1}{2} \right)\right)
         \end{align*}
         is independent of $i;$ namely, the values of $\nu_1,\nu_2,\dotsc,\nu_m$ coincide with one another.
 \item     When $n$ is odd, set $n=2m+1$.    
             Then the Langlands parameter of $\pi^{(n)}$ is given by 
             \begin{align*}
                \phi^\delta_\nu  \oplus \bigoplus^m_{i=1} \phi_{\nu_i,l_i}  \quad  \text{with } l_1> l_2>\cdots > l_m 
             \end{align*} 
such that both $\nu-\dfrac{n-1}{2}$ and $\nu_i \pm \dfrac{l_i}{2}-\dfrac{n-1}{2}$ for each $1\leq i\leq m$ are rational integers. 
             Furthermore, the purity implies that 
            \begin{align*}
                 w  &:= 2\nu-n+1 = 2\nu_i-n+1   \\
& \left(=2\left(\nu-\dfrac{n-1}{2}\right) = \left(  \nu_i-\frac{l_i}{2} -\frac{n-1}{2}  \right)  +  \left( \nu_i + \frac{l_i}{2} -\frac{n-1}{2} \right)\right)
            \end{align*}
         is independent of $i;$ namely, the values of $\nu, \nu_1,\nu_2,\dotsc,\nu_m$ coincide with one another.
\end{enumerate}
In both cases, we say that $\pi^{(n)}$ is {\em pure of weight $w$}.
}
\end{prop}


\begin{rem} \label{rem:purity}
When $n$ is odd, Proposition~\ref{prop:CloCoh} (ii) implies that $\nu$ is an integer, and we readily see from purity that every $\nu_i$ is also an integer. Therefore, due to the same statement, we observe that every $\dfrac{l_i}{2}$ is an integer, or in other words, every weight parameter $l_i$ should be {\em even} when $n$ is odd.    
\end{rem}

We end this subsection by introducing Clozel's conjecture for the existence of motives associated to cohomological irreducible cuspidal automorphic representations of $\mathrm{GL}_n(\mathbf{Q}_{\mathbf{A}})$ according to \cite[Section~4]{clo90}. Let $L(s,\pi^{(n)})$ denote the (complete) automorphic $L$-function associated to $\mathrm{GL}_n$, which is constructed by Godement and Jacquet in \cite{gj72}. 

\begin{conj}(\cite[Conjecture 4.5]{clo90})
{\itshape
Let $\pi^{(n)}$ be a cohomological irreducible cuspidal automorphic representation of $\mathrm{GL}_n(\mathbf{Q}_{\mathbf{A}})$ which is pure of weight $w$ in the sense of Proposition~$\ref{prop:CloCoh}$. 
Then there exists a pure motive  ${\mathcal M}[\pi^{(n)}]$ of rank $n$ and  of weight $-w$ which satisfies
\begin{align*}
   L(s, {\mathcal M}[\pi^{(n)} ])=L\left(s-\frac{n-1}{2}, \pi^{(n)}\right).  
\end{align*}
}
\end{conj}

\subsection{Critical values  and motives for $\mathrm{GL}_3\times \mathrm{GL}_2$}\label{sec:motives}

In this subsection, we focus on the cases where $n=2$ and $3$, and study the Rankin--Selberg $L$-function $L(s,\pi^{(3)}\times \pi^{(2)})$ constructed by Jacquet, Piatetski-Shapiro and Shalika in \cite{jpss83}. We especially discuss the critical points of $L(s,\pi^{(3)}\times \pi^{(2)})$ and their motivic interpretations.

We retain the notations in Section \ref{sec:auto}; in particular, let $\pi^{(3)}$ and $\pi^{(2)}$ be cohomological irreducible cuspidal automorphic representations of $\mathrm{GL}_3(\mathbf{Q}_{\mathbf{A}})$ and $\mathrm{GL}_2(\mathbf{Q}_{\mathbf{A}})$ respectively. Then Proposition~\ref{prop:CloCoh} and Remark~\ref{rem:purity} yield that the Langlands parameters of $\pi^{(2)}_\infty$ and $\pi^{(3)}_\infty$ are described as $\phi_{\nu_2,l_2}$ and $\phi_{\nu_3}^{\delta}\oplus \phi_{\nu_3,l_3}$ respectively, where  $\nu_2\in \frac{1}{2}\mathbf{Z}$, $\nu_3 \in {\mathbf Z}$, $\delta \in \mathbf{Z}/2\mathbf{Z
}$, $l_2 \in {\mathbf Z}$ with $l_2\geq 1$ and  $l_3\in 2\mathbf{Z}$ with $l_3\geq 2$. Thus there exist isomorphisms of $(\mathfrak{gl}_n(\mathbf {R}), {\mathrm O}_n(\mathbf{R}))$-modules
\begin{align*}
   \pi^{(2)}_\infty \cong  D_{ \nu_2, l_2 }, \quad    
   \pi^{(3)}_\infty \cong  {\rm Ind}^{\rm GL_3({\mathbf R})}_{{\rm P}_{2,1}({\mathbf R})  } (   D_{ \nu_3, l_3 }  \boxtimes  \chi_{\nu_3}^{\delta}  )
\end{align*}
due to the Langlands classification (refer to \cite[Section~2]{kna94}). Here ${\rm Ind}^{\rm GL_3({\mathbf R})}_{{\rm P}_{2,1}({\mathbf R})  } (   D_{ \nu_3, l_3 }  \boxtimes  \chi_{\nu_3}^{\delta}  )$ denotes the generalized principal series representation parabolically induced from $D_{\nu_3,l_3}$ and $\chi_{\nu_3}^{\delta}$. From this observation, we see that the minimal ${\rm O}_2({\mathbf R})$-type of $\pi^{(2)}_\infty$ is given by $\tau^{(2)}_{(l_2+1, 0)}$, and thus 
\begin{align*}
  \Lambda^{\rm coh}_2 := \left\{ (\lambda_2, 0) \mid    \lambda_2 \in {\mathbf Z},\,  \lambda_1 \geq 2   \right\}
\end{align*}
classifies the isomorphism classes of minimal ${\rm O}_2({\mathbf R})$-types of (the archimedean parts of) cohomological irreducible cuspidal automorphic representations of ${\rm GL}_2({\mathbf Q}_{\mathbf{A}})$.   
Similarly,  
the minimal ${\rm O}_3({\mathbf R})$-type of $\pi^{(3)}_\infty$ is given by 
$\tau^{(3)}_{(l_3+1, \delta)}$, and thus 
\begin{align*}
  \Lambda^{\rm coh}_3  := \left\{ (\lambda_3, \delta)   \mid   \lambda_3 \in 1+2{\mathbf Z},\,  \lambda_3 \geq 3,\,  \delta \in \mathbf{Z}/2\mathbf{Z}    \right\}
\end{align*}
classifies the isomorphism classes of minimal ${\rm O}_3({\mathbf R})$-types of (the archimedean parts of) cohomological irreducible cuspidal automorphic representations of ${\rm GL}_3({\mathbf Q}_{\mathbf{A}})$.

From now on we normalize  the  action of the center of $\mathrm{GL}_n$  on $\pi^{(n)}$ for $n=2$ and $3$ so that the parameters $\nu_2$ and $\nu_3$ satisfy the  following equalities:
\begin{align}\label{eq:NormNu}
    \nu_2  &=  -\frac{l_2}{2}  + \frac{1}{2}, &  \nu_3 &= - \frac{l_3}{2} + 1. 
\end{align}
These normalizations simplify the description of the Hodge filtrations of the motives $\mathcal{M}[\pi^{(2)}]$ and $\mathcal{M}[\pi^{(3)}]$; see (\ref{eq:Hodge_fil}) below. Furthermore we always assume that the inequality 
\begin{align} \label{eq:l2<l3}
 l_2 < l_3
\end{align}
holds, since we can study the special values of $L(s, \pi^{(3)} \times \pi^{(2)})$ only under this condition by using the modular symbol method. Now let us recall the definitions of the $L$-factor $L_\infty(s,\pi^{(3)}\times\pi^{(2)})$ and the $\epsilon$-factor $\epsilon_\infty(s,\pi^{(3)}\times \pi^{(2)},\psi_\infty)$ at the infinite place. Firstly, for $\pi^{(n)}$ with $n=2$ or $3$, the local $L$-factors and the local $\epsilon$-factors are defined as
\begin{align*}
  L_\infty(s, \pi^{(n)} )  = \begin{cases} \Gamma_{\mathbf C}\left( s + \nu_2 + \dfrac{l_2}{2} \right)   &  \text{for }n=2,  \\
                                                              \Gamma_{\mathbf C}\left( s + \nu_3 + \dfrac{l_3}{2}  \right) \Gamma_{\mathbf R}\left( s + \nu_3 + \delta \right)       &  \text{for }n=3,  \end{cases}    \quad 
  \epsilon_\infty(s, \pi^{(n)}, \psi_\infty )  = \begin{cases} \sqrt{-1}^{ l_2 + 1 }   &  \text{for }n=2,  \\
                                                                        \sqrt{-1}^{ l_3+1 + \delta }       &  \text{for }n=3,  \end{cases}  
\end{align*}
where $\Gamma_{\mathbf{R}}(s)$ and $\Gamma_{\mathbf{C}}(s)$ denote $\pi^{-s/2}\Gamma(s/2)$ and $2(2\pi)^{-s}\Gamma(s)$ respectively, and $\psi_\infty \colon \mathbf{R}\rightarrow \mathbf{C}^\times$ is the standard additive character defined as $\psi_\infty(x)=\exp(2\pi \sqrt{-1}x)$. Since the tensor product of the Langlands parameters associated to $\pi^{(3)}$ and $\pi^{(2)}$ is easily calculated as 
\begin{align*}
   \left(  \phi_{\nu_3, l_3} \oplus \phi^{\delta}_{\nu_3}  \right) \otimes \phi_{\nu_2, l_2}
   = \phi_{ \nu_3+\nu_2, l_3+l_2 } \oplus \phi_{\nu_3+\nu_2, | l_3-l_2 | } \oplus  \phi_{\nu_3+\nu_2, l_2},   
 \end{align*}
we also find that the local $L$-factor and the local $\epsilon$-factor of the Rankin--Selberg convolution $\pi^{(3)} \times \pi^{(2)}$ at the infinite place are given as 
\begin{align*}
   L_\infty(s, \pi^{(3)} \times  \pi^{(2)} )  
   &= \Gamma_{\mathbf C} \left(  s + \nu_3 + \nu_2 + \frac{  l_3 + l_2 }{2}  \right)
      \Gamma_{\mathbf C} \left(  s + \nu_3 + \nu_2 + \frac{   l_3 - l_2  }{2}  \right)
      \Gamma_{\mathbf C} \left(  s + \nu_3 + \nu_2 + \frac{   l_2 }{2}  \right),   \\
\epsilon_\infty(s,  \pi^{(3)} \times  \pi^{(2)} , \psi_\infty )
  &=  \sqrt{-1}^{2l_3+l_2 + 3  } 
\end{align*}
(here we take the assumption (\ref{eq:l2<l3}) into accounts). Mimicking Deligne's definition \cite[Definition~1.3]{del79}, we call an integer $m$ a {\em critical point} of $L(s,\pi^{(3)}\times \pi^{(2)})$ if neither $L_\infty(s,\pi^{(3)}\times \pi^{(2)})$ nor $L_\infty(1-s,(\pi^{(3)}\times \pi^{(2)})^\vee)$ has a pole at $s=m-\dfrac{3}{2}$. One can find an explanation for the shift $-\dfrac{3}{2}$ at (\ref{eq:Gfactor}) below.

Now let us determine the region of the critical points of $L(s, \pi^{(3)} \times \pi^{(2)})$.

\begin{prop}\label{prop:critregion}
{\itshape 
\begin{enumerate}[label={\rm (\roman*)}]
\item Suppose that $l_2 \leq  \dfrac{l_3}{2}$ holds. Then $m$ is a critical point of $L(s,\pi^{(3)}\times \pi^{(2)})$ if and only if the inequality $\dfrac{l_3}{2}   + 1 \leq m \leq \dfrac{l_3}{2} + l_2$ holds.
\item Suppose that $\dfrac{l_3}{2} < l_2 < l_3$ holds. 
Then $m$ is a critical point of $L(s,\pi^{(3)}\times \pi^{(2)})$ if and only if the inequality $l_2   + 1 \leq m \leq l_3$ holds.
\end{enumerate}
}
\end{prop}
\begin{proof}
The statements directly follow from the arguments in \cite[Section 2.4.5.3]{rag16}.  
However we briefly give a proof for the sake of the reader's convenience and the completeness. Due to the functoriality of the local Langlands correspondence, the Langlands parameters of the contragredients $\pi^{(2),\vee}_\infty$ and $\pi^{(3),\vee}_\infty$ are described as follows:
\begin{align*}
\phi_{\nu_2,l_2}^{\vee} &\cong \phi_{-\nu_2,l_2}, & \phi^{\delta,\vee}_\nu\oplus \phi^\vee_{\nu_3,l_3} &\cong \phi^\delta_{-\nu} \oplus \phi_{-\nu_3,l_3}.
\end{align*}
Therefore the Langlands parameter of $(\pi^{(3)}\times \pi^{(2)})^\vee$ is calculated as $\phi_{-\nu_3-\nu_2,l_3+l_2}\oplus \phi_{-\nu_3-\nu_2,\lvert l_3-l_2\rvert}\oplus \phi_{-\nu_3-\nu_2,l_2}$, and by using (\ref{eq:NormNu}) and (\ref{eq:l2<l3}),  we obtain 
\begin{align*}
 L_\infty\left(s-\dfrac{3}{2},\pi^{(3)}\times \pi^{(2)}\right)&=\Gamma_{\mathbf{C}}(s)\Gamma_{\mathbf{C}}(s-l_2) \Gamma_{\mathbf{C}}\left(s-\dfrac{l_3}{2}\right), \\
 L_\infty \left(1-\left(s-\dfrac{3}{2}\right),\left(\pi^{(3)}\times \pi^{(2)}\right)^\vee\right) 
&=\Gamma_{\mathbf{C}}(l_3+l_2+1-s)\Gamma_{\mathbf{C}}(l_3+1-s) \Gamma_{\mathbf{C}}\left(\dfrac{l_3}{2}+l_2+1-s\right).
\end{align*}
Therefore in order for $m\in \mathbf{Z}$ to be a critical point of $L(s,\pi^{(3)}\times \pi^{(2)})$, it must satisfy both of the inequalities
\begin{align*}
 m& \geq \max \left\{l_2+1,\dfrac{l_3}{2}+1\right\}& \text{and}& & m&\leq \min \left\{l_3, \dfrac{l_3}{2}+l_2\right\},
\end{align*}
which verifies the statement. 
\end{proof}

We finally discuss the motivic aspects of these critical points. For $n=2$ and $3$, the Betti realization  $H_{\rm B} ({\mathcal M}[ \pi^{(n)} ]) $ of ${\mathcal M}[\pi^{(n)}]$ has the following Hodge decomposition due to (\ref{eq:NormNu}):
\begin{align} \label{eq:Hodge_fil}
\begin{aligned}
    H_{\rm B} ( {\mathcal M}[ \pi^{(2)} ]) \otimes {\mathbf C}
    &= H^{l_2, 0}  ( {\mathcal M}[ \pi^{(2)} ] )
       \oplus H^{0,l_2}  ( {\mathcal M}[ \pi^{(2)} ] ), \\ 
    H_{\rm B} ({\mathcal M}[ \pi^{(3)} ]) \otimes {\mathbf C}
    &= H^{l_3, 0}  ( {\mathcal M}[ \pi^{(3)} ] )
       \oplus H^{\frac{l_3}{2} ,  \frac{l_3}{2}  }  ( {\mathcal M}[ \pi^{(3)} ] )
       \oplus H^{ 0, l_3 }  ( {\mathcal M}[ \pi^{(3)} ] ). 
\end{aligned}
\end{align}
Define ${\mathcal M}(\pi^{(3)}\times \pi^{(2)})$ to be the tensor product  ${\mathcal M}[ \pi^{(3)} ] \otimes {\mathcal M}[ \pi^{(2)} ]$ of the motives $\mathcal{M}[\pi^{(3)}]$ and $\mathcal{M}[\pi^{(2)}]$. Then (conjecturally) we obtain the identity of the $L$-functions
\begin{align}\label{eq:Gfactor}
   L(s, {\mathcal M}(\pi^{(3)}\times \pi^{(2)}) )  
= L\left(s-\frac{3}{2}, \pi^{(3)} \times  \pi^{(2)} \right).
\end{align}
The Hodge decomposition of $\mathcal{M}(\pi^{(3)}\times \pi^{(2)})(m)$ is calculated as follows (here we abbreviate $\mathcal{M}(\pi^{(3)}\times \pi^{(2)})(m)$ as $\mathcal{M}_{\pi,m}$):
\begin{multline*}
 H_B(\mathcal{M}_{\pi,m})\otimes_{\mathbf{Q}}\mathbf{C}\cong H^{l_3+l_2-m,-m}(\mathcal{M}_{\pi,m} ) \oplus H^{l_3-m,l_2-m}( \mathcal{M}_{\pi,m}) \oplus H^{\frac{l_3}{2}+l_2-m,\frac{l_3}{2}-m}( \mathcal{M}_{\pi,m}) \\
\oplus H^{\frac{l_3}{2}-m,\frac{l_3}{2}+l_2-m}(\mathcal{M}_{\pi,m} ) \oplus H^{l_2-m,l_3-m}( \mathcal{M}_{\pi,m}) \oplus H^{ -m, l_3+l_2-m }( \mathcal{M}_{\pi,m}).
\end{multline*}

As is explained in \cite[Section~1]{del79}, the motive $\mathcal{M}(\pi^{(3)}\times \pi^{(2)})(m)$ is critical if and only if the set
\begin{align*}
 \{(p,q) \in \mathbf{Z}^2 \mid H^{p,q}(\mathcal{M}(\pi^{(3)}\times \pi^{(2)})(m)) \neq 0 \}
\end{align*}
belongs to the union of  $\{ (p,q) \mid p\leq -1,\, q\geq 0\}$ and $\{ (p,q) \mid p\geq 0,\, q\leq -1\}$, since the motive $\mathcal{M}(\pi^{(3)}\times \pi^{(2)})(m)$ does not contain diagonal components. This observation also verifies the statement of Proposition~\ref{prop:critregion}, if one admits the existence of the motive $\mathcal{M}[\pi^{(3)}]$.



\section{Eichiler--Shimura maps for ${\rm GL}_3$}\label{sec:ESmap}

Due to Clozel \cite[Section~3.5]{clo90},  
 the cuspidal cohomology groups of the symmetric space associated to $\mathrm{GL}_n$ with appropriate local systems are well understood, and the finite parts of cohomological irreducible cuspidal automorphic representations of ${\rm GL}_n$ are realized in these cohomology groups. 
It is necessary for us to specify certain cohomology classes  
     to clarify their connection to global zeta integrals and their algebraic and $p$-adic properties.    
Such distinguished cohomology classes are constructed from appropriate cusp forms on ${\rm GL}_n$, and hence we obtain a map form the space of cusp forms on ${\rm GL}_n$ to the cuspidal cohomology groups with certain local coefficients, which we call the {\em Eichler--Shimura map} for ${\rm GL}_n$. 
The Eichler--Shimura map for ${\rm GL}_2$ is a classical object, and its explicit description can be found in standard references (often in terms of elliptic modular forms; see \cite{hid94,vat99} for instance). 
In this section, we give an explicit description of the Eichler--Shimura map for ${\rm GL}_3$ analogously to the ${\rm GL}_2$ case, 
    which will give explicit evaluations of the zeta integrals and clarify their connections to Whittaker periods.   

Let $n$ denotes $2$ or $3$.
In Section \ref{sec:cusp}, we propose descriptions of cusp forms on ${\rm GL}_n(\mathbf{Q}_{\mathbf{A}})$ in an appropriate way for our purpose. Then we introduce several local systems on the symmetric space $Y^{(n)}_{\mathcal{K}_n}$ in Section~\ref{sec:LS}.    
The Eichler--Shimura map for ${\rm GL}_n$ is introduced in Section \ref{sec:ES}.
In this article, we only need the Eichler--Shimura map taking values in the cohomology group of bottom degree.    
However, the explicit description of the Eichler--Shimura map taking values in the cohomology group of top degree will be also important 
   for further studies of zeta integrals, and hence we briefly summarize its construction in Appendix \ref{sec:AppA}.

\subsection{Cusp forms}\label{sec:cusp}

Here we summarize some notation on automorphic representations $\pi^{(3)}$ and $\pi^{(2)}$ in this article.  
For $n=2$ and $3$, 
denote by $\pi^{(n)}$ a cohomological irreducible cuspidal automorphic representation of ${\rm GL}_n({\mathbf Q}_{\mathbf A})$
     of minimal ${\rm O}_n({\mathbf R})$-type $\tau^{(n)}_{\boldsymbol{\lambda}^{(n)}}$ for $\boldsymbol{\lambda}^{(n)} = (\lambda_n, \delta_n) \in \Lambda_n^{\mathrm{coh}}$ (recall the definition of $\Lambda^{\mathrm{coh}}_n$ from Section~\ref{sec:motives}).  We often abbreviate $\tau^{(n)}_{\boldsymbol{\lambda}^{(n)}}$ as $\tau^{(n)}_{\boldsymbol{\lambda}}$ if there is no risk of confusion. Furthermore, since the second component of $\boldsymbol{\lambda}^{(2)}$ always equals $0$, we remove the subindex $3$ from the second component of $\boldsymbol{\lambda}^{(3)}$ and just write $\boldsymbol{\lambda}^{(3)}=(\lambda_3,\delta)$.
Let $\phi_{\nu_2, l_2}$ (resp.\ $\phi_{\nu_3, l_3}\oplus \phi_{\nu_3}^{\delta}$) be the Langlands parameter of $\pi^{(2)}_\infty$ (resp.\ $\pi^{(3)}_\infty$); then we have
\begin{align} \label{eq:Lambda}
 \lambda_3&=l_3+1, & \lambda_2 &=l_2+1.
\end{align}  
If we denote by $\omega_{\pi^{(n)}}$ the central character of $\pi^{(n)}$, the evaluation $\omega_{\pi^{(n)}, \infty}(-1)$ at $-1$ of its archimedean component  $\omega_{\pi^{(n)}, \infty}$  is described by these parameters as follows (note that $\lambda_3=l_3+1$ is odd since $l_3$ is even due to Remark~\ref{rem:purity}):
\begin{align*}
   \omega_{\pi^{(2)},\infty}(-1) &= (-1)^{\lambda_2}, 
 & \omega_{\pi^{(3)},\infty}(-1) &= (-1)^{\lambda_3+\delta} =  -(-1)^{\delta}.   
\end{align*}

For $n=2$ or $3$, let $\phi: {\rm GL}_n({\mathbf Q}) \backslash {\rm GL}_n({\mathbf Q}_{\mathbf A}) \to {\mathbf C}$ be a cuspidal automorphic form associated with $\pi^{(n)}$; 
namely, $\pi^{(n)}$ is isomorphic to the infinite dimensional irreducible subquotient of the right regular representation of ${\rm GL}_n({\mathbf Q}_{\mathbf A})$ generated by $\phi$.  
Suppose that $\pi^{(n)}$ is fixed by the action of an open compact subgroup ${\mathcal K}_n$ of ${\rm GL}_n({\mathbf Q}_{\mathbf A, {\rm fin}})$.  
Note that ${\rm Hom}_{{\mathrm O}_n({\mathbf R})}(\tau^{(n)}_{\boldsymbol{\lambda}}, \pi^{(n)}_\infty)$ is of dimenstion $1$ by Schur's lemma.    
Hence, for $n=2$,  
we find a pair $(h^+, h^-)$ of cuspidal automorphic forms $h^\pm : {\rm GL}_2({\mathbf Q}) \backslash {\rm GL}_2({\mathbf Q}_{\mathbf A}) \to {\mathbf C}$ 
associated with $\pi^{(2)}$ satisfying
\begin{align*}
\begin{pmatrix}
 h^+(gtu)&  h^-(gtu) 
\end{pmatrix}  
 &=  t^{2\nu_2 } 
         \begin{pmatrix}
	   h^+(g)  e^{ \sqrt{-1} \lambda_2 \theta} &  h^-(g)  e^{ - \sqrt{-1} \lambda_2 \theta}  
	 \end{pmatrix} &  
\begin{pmatrix}
  h^+(g \epsilon) &  h^-(g \epsilon)  
\end{pmatrix}
&= 
\begin{pmatrix}
 h^-(g) & h^+(g)  
\end{pmatrix}  \\
 &\stackrel{(\ref{eq:NormNu})}{=} t^{(-l_2+1)} 
\begin{pmatrix}
 h^+(g)  e^{ \sqrt{-1} (l_2+1) \theta}&  h^-(g)  e^{ - \sqrt{-1} (l_2+1)\theta} 
\end{pmatrix} , & & 
\end{align*}   
for each $g\in {\rm GL}_2({\mathbf Q}_{\mathbf A})$, $t \in {\mathbf R}^\times_{>0}$, 
              $u = \begin{pmatrix} \cos \theta & \sin \theta \\ - \sin \theta & \cos \theta \end{pmatrix}\in {\rm SO}_2({\mathbf R})$ and 
              $\epsilon = \begin{pmatrix} -1 & 0 \\ 0 & 1 \end{pmatrix}$.   
Since the archimedean part of $h^+$ (resp.\ $h^-$) belongs to $D^+_{l_2}$ (resp.\ $D^-_{l_2}$), 
            we call $h^+$ (resp.\ $h^-$) a holomorphic (resp.\ an anti-holomorphic) cuspidal automorphic form associated with $\pi^{(2)}$, which is uniquely determined up to scalar multiples.
For each $\boldsymbol{\lambda}^{(2)}\in \Lambda^{\mathrm{coh}}_2$, let ${\mathcal S}^{(2)}_{\boldsymbol{\lambda}^{(2)}}({\mathcal K}_2)$ (often abbreviated as $\mathcal{S}^{(2)}_{\boldsymbol{\lambda}}(\mathcal{K}_2)$) be the ${\mathbf C}$-vector space spanned 
    by both of holomorphic and anti-holomorphic cuspidal automorphic forms associated with certain cohomological irreducible  cuspidal automorphic representations of ${\rm GL}_2({\mathbf Q}_{\mathbf A})$, which are invariant under the right translation by ${\mathcal K}_2$ and 
    whose archimedean parts have  the minimal ${\mathrm O}_2({\mathbf R})$-type $\tau^{(2)}_{\boldsymbol{\lambda}}$. 

In a similar way, for  $n=3$, 
we find a tuple $\begin{pmatrix}
		  f^{\boldsymbol{\lambda}}_{\lambda_3}& f^{\boldsymbol{\lambda}}_{\lambda_3-1} &\dotsc & f^{\boldsymbol{\lambda}}_{-\lambda_3}
		 \end{pmatrix}$ consisting of  
cuspidal automorphic forms $f^{\boldsymbol{\lambda}}_\alpha \colon {\rm GL}_3({\mathbf Q}) \backslash {\rm GL}_3({\mathbf Q}_{\mathbf A}) \to {\mathbf C}$  for $-\lambda_3 \leq \alpha \leq \lambda_3$ 
associated with $\pi^{(3)}$ satisfying
\begin{align*}
\begin{pmatrix}
 f^{\boldsymbol{\lambda}}_{\lambda_3}(gtu) & f^{\boldsymbol{\lambda}}_{\lambda_3-1}(gtu) & \dotsc & f^{\boldsymbol{\lambda}}_{-\lambda_3}(gtu)  
\end{pmatrix}
 &=     \begin{pmatrix}
	  f^{\boldsymbol{\lambda}}_{\lambda_3}(g)& f^{\boldsymbol{\lambda}}_{\lambda_3-1}(g)  & \dotsc & f^{\boldsymbol{\lambda}}_{-\lambda_3}(g)  
	\end{pmatrix}  t^{ 3\nu_3}   M_{\boldsymbol{\lambda}}(u)  \\
&\stackrel{(\ref{eq:NormNu})}{=} 
\begin{pmatrix}
  f^{\boldsymbol{\lambda}}_{\lambda_3}(g) & f^{\boldsymbol{\lambda}}_{\lambda_3-1}(g) & \dotsc & f^{\boldsymbol{\lambda}}_{-\lambda_3}(g)  
\end{pmatrix} t^{ -\frac{3}{2}l_3+3}   M_{\boldsymbol{\lambda}}(u) 
\end{align*}   
for each $g\in {\rm GL}_3({\mathbf Q}_{\mathbf A})$, $t \in {\mathbf R}^\times_{>0}$ and $u\in {\rm O}_3({\mathbf R})$; recall the definition of the matrix $M_{\boldsymbol{\lambda}}(u)$ from (\ref{eq:matM}). Here we abbreviate $\boldsymbol{\lambda}^{(3)}$ as $\boldsymbol{\lambda}$ to lighten the notation.     
We regard this tuple as a vector and denote by ${\boldsymbol f}$, which we call a {\em cusp form associated with $\pi^{(3)}$}.
For each $\boldsymbol{\lambda}^{(3)}\in \Lambda^{\mathrm{coh}}_{3}$, define ${\mathcal S}^{(3)}_{\boldsymbol{\lambda}^{(3)}}({\mathcal K}_3)$
 (often abbreviated as $\mathcal{S}^{(3)}_{\boldsymbol{\lambda}}(\mathcal{K}_3)$) to be the  
   ${\mathbf C}$-vector space spanned 
    by cusp forms associated with certain cohomological irreducible cuspidal automorphic representations of ${\rm GL}_3({\mathbf Q}_{\mathbf A})$, which are invariant under the right translation by ${\mathcal K}_3$ and 
    whose archimedean parts have  the minimal ${\mathrm O}_3({\mathbf R})$-type $\tau^{(3)}_{\boldsymbol{\lambda}}$.

\subsection{Local systems}\label{sec:LS}

Let $n$ denote $2$ or $3$. 
For each open compact subgroup ${\mathcal K}_n\subset {\rm GL}_n(  {\mathbf Q}_{{\mathbf A}, {\rm fin}} )$, define the corresponding symmetric space  as $Y^{(n)}_{{\mathcal K}_n} = {\rm GL}_n({\mathbf Q}) \backslash  {\rm GL}_n({\mathbf Q}_{\mathbf A})   / {\mathbf R}^\times_{>0} {\rm SO}_n({\mathbf R})  {\mathcal K}_n$.
In this subsection, we introduce local systems $\mathcal{L}^{(n)}(*;A)$ on $Y^{(n)}_{{\mathcal K}_n}$ coming from finite dimensional irreducible representations $L^{(n)}(*; A)$ of ${\rm GL}_n$ for a field $A$ of characteristic $0$.
They will play crucial roles to give a cohomological interpretation of cusp forms on ${\rm GL}_n({\mathbf Q}_{\mathbf A})$ in the next subsection.

Let $\boldsymbol{n}=(n_1,n_2)\in \mathbf{Z}^2$ and $\boldsymbol{w}=(w_1^+,w_1^-,w_2)\in \mathbf{Z}^3$ be as in Sections~\ref{sec:repGL2} and \ref{sec:repGL3}, where $n_1, w_1^+,w_1^-\geq 0$. Then we define the local system ${\mathcal L}^{(n)}(*; A)$ to be the sheaf of locally constant sections of the projection $\mathrm{pr}_1$ defined as 
\begin{align*}
   {\rm GL}_n({\mathbf Q})  \backslash \left(  {\rm GL}_n({\mathbf Q}_{\mathbf A})   / {\mathbf R}^\times_{>0} {\rm SO}_n({\mathbf R})  {\mathcal K}_n  
               \times   L^{(n)}( *; A)   \right) 
               \stackrel{ {\rm pr}_1 }{ \longrightarrow }  Y^{(n)}_{{\mathcal K}_n},     
\end{align*}
where $*$ denotes $\boldsymbol{n}$ for $n=2$ and $\boldsymbol{w}$ for $n=3$. We are especially interested in the local system associated to a cohomological irreducible cuspidal automorphic representation $\pi^{(n)}$ of $\mathrm{GL}_n(\mathbf{Q}_{\mathbf{A}})$ whose archimedean component has the minimal $\mathrm{O}_3(\mathbf{R})$-type $\tau^{(n)}_{\boldsymbol{\lambda}}$. For $\boldsymbol{\lambda}^{(n)}=(\lambda_n,\delta_n)\in \Lambda^{\mathrm{coh}}_n$, set $n_{\boldsymbol{\lambda}}=\lambda_2-2 \, (=l_2-1)$ and $w_{\boldsymbol{\lambda}}=\dfrac{\lambda_3-3}{2}\, \left(=\dfrac{l_3}{2}-1\right)$. Then define $\boldsymbol{n}_{\boldsymbol{\lambda}}\in \mathbf{Z}^2$ and $\boldsymbol{w}_{\boldsymbol{\lambda}}\in \mathbf{Z}^3$ as 
\begin{align} \label{eq:w2}
 \boldsymbol{n}_{\boldsymbol{\lambda}} &=(n_{\boldsymbol{\lambda}},0)=(\lambda_2-2,0), & \boldsymbol{w}_{\boldsymbol{\lambda}} &=(w_{\boldsymbol{\lambda}},w_{\boldsymbol{\lambda}},w_{\boldsymbol{\lambda}})=\left(\dfrac{\lambda_3-3}{2}, \dfrac{\lambda_3-3}{2}, \dfrac{\lambda_3-3}{2}\right).
\end{align}
%
Note that $w_{\boldsymbol{\lambda}}$ is indeed an integer since $\lambda_3$ is odd due to the definition of $\Lambda^{\mathrm{coh}}_3$. For later use, we define  the set $\Xi^{\rm coh}_{3}$ of cohomological weights of local systems for ${\rm GL}_3$ to be $ \Xi^{\rm coh}_{3}
       = \left\{   \left(   w,w,w   \right)   \mid  w \in {\mathbf Z},   w\geq 0  \right\}$.

\subsection{Eichler--Shimura maps}\label{sec:ES}

Let $\boldsymbol{\lambda}^{(n)}$ be an element of $\Lambda^{\mathrm{coh}}_n$. Due to Clozel \cite[Section~3.5]{clo90}, the structure of the cuspidal cohomology group $H^i_{\rm cusp}( Y^{(n)}_{{\mathcal K_n}},  {\mathcal L}^{(n)}(*; {\mathbf C})  )$ with coefficients in appropriate local systems is well understood, and described by various cuspidal automorphic representations of given minimal $\mathrm{O}_n(\mathbf{R})$-type.
In this subsection, we construct a specific class in the cuspidal cohomology group from an appropriate cusp form on $\mathrm{GL}_n(\mathbf{Q}_{\mathbf{A}})$. For later computations, we here give an appropriate form of this construction, which we call the {\em Eichler--Shimura map}.

For $n=2$ and $3$, let $\pi^{(n)}$ be a cohomological irreducible cuspidal automorphic representation of $\mathrm{GL}_n(\mathbf{Q}_{\mathbf{A}})$ whose archimedean part $\pi^{(n)}_{\infty}$ has the minimal $\mathrm{O}_n(\mathbf{R})$-type $\tau^{(n)}_{\boldsymbol{\lambda}}$. Set $K_n = {\mathbf R}^\times_{>0} {\rm SO}_n({\mathbf R})$,  
and let $H_{\pi^{(n)}, K_n}$ denote the $(\mathfrak{gl}_n({\mathbf R}), K_n)$-module associated with $\pi^{(n)}_\infty$. 
In (\ref{eq:NormNu}), 
we have normalized the action of ${\mathbf R}^\times_{>0}$ on $H_{\pi^{(n)}, K_n}$ so that
\begin{align*}
   \pi_\infty^{(n)}(t) (v) 
   = t^{n\nu_n} v
   = \begin{cases}  t^{-l_2+1}v =t^{-\lambda_2+2}v=t^{ -n_{\boldsymbol{\lambda}}} v    &    \text{for }n=2,   \\
                                   t^{-\frac{3}{2}l_3+3}v=t^{-3\frac{\lambda_3-3}{2}}v=        t^{-3w_{\boldsymbol{\lambda}} }  v        &   \text{for }n=3   \end{cases}           
\end{align*}          
holds for each $t\in {\mathbf R}^\times_{>0}$ and $v\in H_{\pi^{(n)}, K_n}$.

Let $*_{\boldsymbol{\lambda}}$ denote $\boldsymbol{n}_{\boldsymbol{\lambda}}$ if $n$ equals $2$ and $\boldsymbol{w}_{\boldsymbol{\lambda}}$ if $n$ equals $3$. Then it is known that the $\pi^{(n)}_{\rm fin}$-isotypic component $H^i(Y^{(n)}_{\mathcal K_n}, {\mathcal L}^{(n)}( *_{\boldsymbol{\lambda}} ; {\mathbf C})  )[\pi^{(n)}_{\rm fin}]$ 
of $H^i(Y^{(n)}_{\mathcal K_n}, {\mathcal L}^{(n)}( *_{\boldsymbol{\lambda}} ; {\mathbf C})  )$ satisfies 
\begin{align*}
H^i(Y^{(n)}_{\mathcal K_n}, {\mathcal L}^{(n)}( *_{\boldsymbol{\lambda}} ; {\mathbf C})  )[\pi^{(n)}_{\rm fin}]
= H^i( \mathfrak{gl}_n({\mathbf R}), K_n ; H_{\pi^{(n)}, K_n}  \otimes L^{(n)}( *_{\boldsymbol{\lambda}} ; {\mathbf C})    )
     \otimes_{\mathbf C} \pi^{(n), {\mathcal K}_n}_{\rm fin}.
\end{align*}

In order to construct an distinguished  cohomology class, it is necessary to describe the $(\mathfrak{gl}_n(\mathbf{R}),K_n)$-cohomology group $H^i( \mathfrak{gl}_n({\mathbf R}), K_n ; H_{\pi^{(n)}, K_n}  \otimes L^{(n)}( *_{\boldsymbol{\lambda}} ; {\mathbf C})    )$ in an explicit manner. 
The following well-known theorem is due to Clozel \cite[Lemme 3.14]{clo90}. 
Here we adopt the formalism presented by Mahnkopf in \cite{mah05}.

\begin{thm}(refer to \cite[Section 3.1.2]{mah05} and \cite[Section 5.5]{gr14})
{\itshape 
We have 
\begin{align*}
     H^i( \mathfrak{gl}_2({\mathbf R}), K_2 ; H_{\pi^{(2)}, K_2} \otimes L^{(2)}( \boldsymbol{n}_{\boldsymbol{\lambda}}; {\mathbf C})  )
      \cong \begin{cases}  {\mathbf C}^2  &  \text{for }i=1,    \\
                                       0  & \text{otherwise}.       \end{cases}   
\end{align*}
We also find that  
\begin{align*}                                                       
     H^i( \mathfrak{gl}_3({\mathbf R}), K_3 ; H_{\pi^{(3)}, K_3}  \otimes L^{(3)}( \boldsymbol{w}_{\boldsymbol{\lambda}}; {\mathbf C})    )
      \cong \begin{cases}  {\mathbf C}  &  \text{for } i=2 \text{ and }3,    \\
                                       0  &   \text{otherwise}.  \end{cases}    \end{align*}
}
\end{thm}
In the present article, 
we are interested in the cohomology group 
$H^{b_n}( \mathfrak{gl}_n({\mathbf R}), K_n ; H_{\pi^{(n)}, K_n}  \otimes L^{(n)}( *_{\boldsymbol{\lambda}}; {\mathbf C})    )$
of bottom degree $b_n:=[\frac{n^2}{4}]$.
Note that the cohomology of top degree has also independent interest beyond the scope of this article,
hence we give an explicit basis of the cohomology  group $H^3( \mathfrak{gl}_3({\mathbf R}), K_3 ; H_{\pi^{(3)}, K_3}  \otimes L^{(3)}( \boldsymbol{w}_{\boldsymbol{\lambda}}; {\mathbf C})    )$ 
in Appendix~\ref{sec:AppA}.

Let ${\mathcal P}_n$ be the quotient Lie algebra $\mathfrak{gl}_n({\mathbf R}) / {\rm Lie}(K_n)$ and  ${\mathcal P}_{n, {\mathbf C}} = {\mathcal P}_n \otimes_{\mathbf R} {\mathbf C}$ its complexification. 
By using \cite[II, Proposition 3.1]{bw80}, we obtain a natural isomorphism
\begin{align}\label{eq:HgKHom}
    H^i( \mathfrak{gl}_n({\mathbf R}), K_n ; H_{\pi^{(n)}, K_n}  \otimes L^{(n)}( *_{\boldsymbol{\lambda}} ; {\mathbf C})    )
    \cong 
    {\rm Hom}_{ {\rm SO}_n({\mathbf R})  } \left(  \bigwedge^i {\mathcal P}_{n, {\mathbf C}},  H_{\pi^{(n)}, K_n}  \otimes  L^{(n)}( *_{\boldsymbol{\lambda}} ; {\mathbf C})    \right).  
\end{align} 
To study the module appearing in the right-hand side of (\ref{eq:HgKHom}), we first analyze the adjoint action of $\mathrm{SO}_n(\mathbf{R})$ on $\mathcal{P}_{n,\mathbf{C}}$ with respect to a certain specific basis. Let $\{H,E\}$ be a basis of $\mathcal{P}_{2,\mathbf{C}}$ defined by
\begin{align*}
  H &= \frac{1}{2} \begin{pmatrix}  1 & 0 \\ 0 & -1  \end{pmatrix}, & 
  E &= \frac{1}{2}  \begin{pmatrix}  0 & 1 \\ 1 & 0  \end{pmatrix}.
\end{align*}

Similarly, let $\{X_2,X_1,X_0,X_{-1},X_{-2}\}$ be a basis of $\mathcal{P}_{3,\mathbf{C}}$ defined as follows:
  \begin{align}\label{eq:defX}
    \begin{aligned}
             A_1 &=  \begin{pmatrix}  1 & 0 & 0  \\  0 & -1 & 0 \\ 0 & 0 & 0   \end{pmatrix}, &  
             A_2 &=  \begin{pmatrix}  1 & 0 & 0  \\  0 & 1 & 0 \\ 0 & 0 & -2   \end{pmatrix},  & 
             N_1 &=  \begin{pmatrix}  0 & 1 & 0  \\  1 & 0 & 0 \\ 0 & 0 & 0   \end{pmatrix},&
             N_2 &=  \begin{pmatrix}  0 & 0 & 1  \\  0 & 0 & 0 \\ 1 & 0 & 0   \end{pmatrix},&
             N_3 &=  \begin{pmatrix}  0 &  0  & 0  \\  0 & 0 & 1 \\ 0 & 1 & 0   \end{pmatrix}, \\ 
            X_{2} &= A_1+\sqrt{-1}N_1, &
            X_{1} &= N_2+\sqrt{-1}N_3, &
            X_{0} &= A_2, &
            X_{-1} &= N_2-\sqrt{-1}N_3, &
            X_{-2} &= A_1-\sqrt{-1}N_1.  
 \end{aligned} 
\end{align}

The following lemma is a consequence of a direct calculation: 

\begin{lem}\label{lem:Lieact}
{\itshape 
For each $u\in {\rm SO}_2({\mathbf R})$, we have ${\rm Ad}(u) 
\begin{pmatrix}
 E & H
\end{pmatrix} = 
\begin{pmatrix}
 E & H
\end{pmatrix}u^{-2}$.   
For each $u\in {\rm SO}_3({\mathbf R})$, we have 
\begin{align*}
    {\rm Ad}(u) \begin{pmatrix} X_2 & X_1 & X_0 & X_{-1} & X_{-2} \end{pmatrix}
   &=  \begin{pmatrix} X_2 & X_1 & X_0 & X_{-1} & X_{-2} \end{pmatrix}
        M_{(2,0)}(u), 
\end{align*}
where the matrix $M_{(2,0)}(u)\in {\rm GL}_5({\mathbf C})$ is defined as in {\rm (\ref{eq:matM})}. 
}
\end{lem}

First let us consider the $\mathrm{GL}_2$ case.
Let  $\{\mathrm{d}H, \mathrm{d}E \}$ denote the dual basis of $\mathcal{P}_{2,\mathbf{C}}^*$, where $\mathcal{P}_{2,\mathbf{C}}^*$ denotes the linear dual space of $\mathcal{P}_{2,\mathbf{C}}$.  Let ${\rm Ad}^*$ be the contragredient action of  ${\rm Ad}$. 
Then the following formula is easily deduced from the first satement of Lemma \ref{lem:Lieact}; for each $u=\begin{pmatrix} \cos\theta & \sin\theta \\ -\sin \theta & \cos \theta \end{pmatrix} \in {\rm SO}_2({\mathbf R})$, 
      we have 
\begin{align*}
     {\rm Ad}^\ast(u) \begin{pmatrix}
		        {\rm dE} - \sqrt{-1} {\rm d}H &  -{\rm dE} - \sqrt{-1} {\rm d}H  
		      \end{pmatrix} 
     = \begin{pmatrix}{\rm dE} - \sqrt{-1} {\rm d}H&  - {\rm dE} - \sqrt{-1} {\rm d}H 
       \end{pmatrix}   \begin{pmatrix} e^{-2 \sqrt{-1}\theta}   & 0 \\ 0 & e^{2\sqrt{-1} \theta} \end{pmatrix}.
\end{align*}
Hence, for each holomorphic (resp.\ anti-holomorphic) cusp  form $h^+$ (resp.\ $h^-$) in ${\mathcal S}^{(2)}_{\boldsymbol{\lambda}}({\mathcal K}_2)$, 
the  element 
\begin{align*}
     \delta^{(2)}(h^\pm) = h^\pm ( \mp X + \sqrt{-1} Y)^{n_{\boldsymbol{\lambda}} }  ( \pm {\rm dE} - \sqrt{-1} {\rm d}H) \quad \bigl(=h^\pm ( \mp X + \sqrt{-1} Y)^{\lambda_2-2}  ( \pm {\rm dE} - \sqrt{-1}\mathrm{d}H) \bigr)
\end{align*}
is indeed contained in 
${\rm Hom}_{ {\rm SO}_2({\mathbf R}) } \left(  {\mathcal P}_{2, {\mathbf C}}, H_{\pi^{(2)},K_2} \otimes  L^{(2)}(  \boldsymbol{n}_{\boldsymbol{\lambda}}  ; {\mathbf C})    \right)\otimes_{\mathbf C} (\pi^{(2) }_{\rm fin})^{\mathcal K_2}$.
We thus obtain $\delta^{(2)}$ a map 
\begin{align*}
  \delta^{(2)} : {\mathcal S}^{(2)}_{\boldsymbol{\lambda}}({\mathcal K}_2) \longrightarrow H^1_{\rm cusp}(Y^{(2)}_{{\mathcal K}_2}, {\mathcal L}^{(2)}( {\boldsymbol{n}}_{\boldsymbol{\lambda}} ; {\mathbf C}))\, ; h^\pm \mapsto \delta^{(2)}(h^\pm) 
\end{align*}
which is called the Eichler--Shimula map for $\mathrm{GL}_2$. 
It is obvious that $\delta^{(2)}$ is a  Hecke equivariant homomorphism from its construction.

Now let us construct a similar map for the ${\rm GL}_3$ case;
to be precise, we will define a Hecke equivariant homomorphism 
\begin{align*}
  \delta^{(3)} : {\mathcal S}^{(3)}_{\boldsymbol{\lambda}}({\mathcal K}_3) \longrightarrow H^2_{\rm cusp}(Y^{(3)}_{{\mathcal K}_3}, {\mathcal L}^{(3)}( \boldsymbol{w}_{\boldsymbol{\lambda}} ; {\mathbf C}))
\end{align*}
for $\boldsymbol{\lambda}^{(3)}=(\lambda_3,\delta) \in \Lambda^{\rm coh}_3$.   
For this purpose, we have to introduce some more notation. 
Let $P^2\in {\rm M}_{10, 7}({\mathbf C})$ be the matrix which is introduced in (\ref{eq:defP2}).  
Then we define ${\boldsymbol \omega}_i \in \bigwedge^2 {\mathcal P}^\ast_{3, {\mathbf C}}$ for $-3\leq i \leq 3$ to be   
\begin{align}\label{eq:def_bfomega}
  \begin{aligned}
     \begin{pmatrix}
      {\boldsymbol \omega}_3 &  \boldsymbol{\omega}_2 & \dotsc&  {\boldsymbol \omega}_{-3}
     \end{pmatrix}      =   \begin{pmatrix}	 
	 {\rm d}X_2 \wedge {\rm d}X_1 & {\rm d}X_2 \wedge {\rm d}X_0 & \dotsc & {\rm d}X_{-1} \wedge {\rm d}X_{-2} 
\end{pmatrix} P^2.
 \end{aligned} 
\end{align}
Here $\{ {\rm d}X_2,{\rm d}X_1,{\rm d}X_0,{\rm d}X_{-1},{\rm d}X_{-2}\}$ is the dual basis of $\mathcal{P}^*_{3,\mathbf{C}}$ and $\{\mathrm{d}X_i\wedge \mathrm{d}X_j\}_{2\geq i>j\geq -2}$ is ordered lexicographically.
Consider 
\begin{align*}
   P(X,Y,Z,A,B,C, z_1, z_2, z_3)    
& :=   \left(  (X, Y, Z) \begin{pmatrix} z_1 \\ z_2 \\ z_3 \end{pmatrix}   \right)^{w_{\boldsymbol{\lambda}}} 
                   \otimes  \left(  (A, B, C) \begin{pmatrix} z_1 \\ z_2 \\ z_3 \end{pmatrix}   \right)^{w_{\boldsymbol{\lambda}}} 
     \begin{pmatrix}
      v^{(3, \delta)}_{3}& v^{(3,\delta)}_2 & \dotsc & v^{(3,\delta)}_{-3}
     \end{pmatrix}     \\  
 &=  \begin{pmatrix}
      v^{\boldsymbol{\lambda}}_{\lambda_3} & v^{\boldsymbol{\lambda}}_{\lambda_3-1} & \dotsc & v^{\boldsymbol{\lambda}}_{   -\lambda_3} 
     \end{pmatrix}    {\mathcal P}(X,Y,Z, A,B,C) 
\end{align*}
(recall that $w_{\boldsymbol{\lambda}}=\dfrac{\lambda_3-3}{2}$). Then we can write down ${\mathcal P}(X,Y,Z, A,B,C)$ as follows: 
\begin{align}\label{eq:matP}
      {\mathcal P}(X,Y,Z, A,B,C)  
      = \begin{pmatrix} 
             {\mathcal P}_{\lambda_3,3} (X,Y,Z, A,B,C) & \cdots & {\mathcal P}_{\lambda_3,-3} (X,Y,Z, A,B,C)   \\  
                  \vdots &   \ddots &   \vdots  \\
              {\mathcal P}_{-\lambda_3,3} (X,Y,Z, A,B,C) & \cdots & {\mathcal P}_{-\lambda_3,-3} (X,Y,Z, A,B,C)                  \end{pmatrix},  
\end{align}
which is an element of ${\rm M}_{2\lambda_3+1, 7} ( {\mathbf C}[X,Y,Z ; A,B,C]_{w_{\boldsymbol{\lambda}}, w_{\boldsymbol{\lambda}}}  )$.

\begin{lem} \label{lem:ES3}
{\itshape 
\begin{enumerate}[label={\rm (\roman*)}]
\item 
The tuple $\{  {\boldsymbol \omega}_3,  \ldots,  {\boldsymbol \omega}_{-3} \}$ forms a basis of $\bigwedge^2 {\mathcal P}^\ast_{3, {\mathbf C}}$.  
Moreover, for each $u\in {\rm SO}_3({\mathbf R})$,   
we have 
\begin{align*}
     \wedge^2 {\rm Ad}^\ast(u) 
\begin{pmatrix}
 {\boldsymbol \omega}_3 & \boldsymbol{\omega}_2 &  \dotsc &  {\boldsymbol \omega}_{-3}
\end{pmatrix}   
  = \begin{pmatrix}
     {\boldsymbol \omega}_3 & \boldsymbol{\omega}_2 & \dotsc &  {\boldsymbol \omega}_{-3}   
    \end{pmatrix}  {}^{\rm t}\! M_{(3,0)}(u)^{-1}. 
\end{align*}

     \item  Every entry
                     ${\mathcal P}_{\alpha, \beta} (X,Y,Z, A,B,C)$ of $\mathcal{P}(X,Y,Z,A,B,C)$ is an element of $L^{(3)}( \boldsymbol{w}_{\boldsymbol{\lambda}} ; {\mathbf C})$ all of whose coefficients are contained in $\mathbf{Z}[2^{-1},\sqrt{-1}]$.      
     \item For each $u\in {\rm SO}_3({\mathbf R})$, we have 
\begin{align*}
        \varrho^{(3)}_{\boldsymbol{w}_{\boldsymbol{\lambda}}} (u)  {\mathcal P}(X,Y,Z,A,B,C)
   =   M^{-1}_{\boldsymbol{\lambda}}(u)  {\mathcal P}(X,Y,Z,A,B,C) M_{(3,0)}(u).
\end{align*}
  \end{enumerate}  
  }
\end{lem}
\begin{proof}
See Lemma~\ref{lem:basis_wedge} for (i) and Lemma \ref{lem:construction_delta} for (ii) and (iii) in Appendix \ref{sec:AppA}. 
\end{proof}

Let $\boldsymbol{f} = 
\begin{pmatrix}
 f^{\boldsymbol{\lambda}}_{\lambda_3} & f^{\boldsymbol{\lambda}}_{\lambda_3-1}& \dotsc & f^{\boldsymbol{\lambda}}_{-\lambda_3}
\end{pmatrix}$ be an element in  ${\mathcal S}^{(3)}_{\boldsymbol{\lambda}}({\mathcal K}_3)$.  
Define 
\begin{align*}
   \delta^{(3)}(\boldsymbol{f}) 
   =&    \begin{pmatrix}
	  f^{\boldsymbol{\lambda}}_{\lambda_3}&  f^{\boldsymbol{\lambda}}_{\lambda_3-1} & \dotsc& f^{\boldsymbol{\lambda}}_{-\lambda_3}
	 \end{pmatrix}  {\mathcal P}(X,Y,Z, A,B,C) 
          \begin{pmatrix}  {\boldsymbol \omega}_3 \\ \vdots \\ {\boldsymbol \omega}_{-3}   \end{pmatrix}.
\end{align*}
Since $\delta^{(3)}(\boldsymbol{f})$ gives an element in ${\rm Hom}_{ {\rm SO}_3({\mathbf R})  } \left(  \bigwedge^2 {\mathcal P}_{3, {\mathbf C}}, H_{\pi^{(3)},K_3} \otimes  L^{(3)}( \boldsymbol{w}_{\boldsymbol{\lambda}} ; {\mathbf C})    \right)\otimes (\pi^{(3)}_{\rm fin})^{\mathcal K_3}$ due to Lemma~\ref{lem:ES3} (iii), we obtain a Hecke equivariant map 
\begin{align*}
 \delta^{(3)} \colon 
     {\mathcal S}^{(3)}_{\boldsymbol{\lambda}}({\mathcal K}_3) \longrightarrow H^2_{\rm cusp}(Y^{(3)}_{{\mathcal K}_3}, {\mathcal L}^{(3)}( \boldsymbol{w}_{\boldsymbol{\lambda}} ;
  {\mathbf C}))\, ; \boldsymbol{f} \longmapsto \delta^{(3)}(\boldsymbol{f}). 
\end{align*}

\section{Cohomological cup products and global zeta integrals for ${\rm GL}_3 \times {\rm GL_2}$}\label{sec:zetaint}

The purpose of this section is to introduce 
 the connection between the cohomology classes constructed in the previous section and the global zeta integrals for $\mathrm{GL}_3\times \mathrm{GL}_2$,  
  which is given by precise study of the cohomological cup product pairing based on an explicit description of the branching rule for the pair (${\rm GL}_3, {\rm GL}_2$) provided in Section~\ref{sec:branch}.   
After introducing the cup product pairing which we shall use in Section~\ref{sec:branch*}, we consider in Section~\ref{sec:cup} the cup product of the Eichler--Shimura classes constructed in the previous section. 
In Section \ref{sec:birch}, 
  we firstly give a cohomological interpretation of the partial zeta integrals (see Proposition \ref{prop:intsign}).  
This immediately deduces a cohomological interpretation of the global zeta integrals, which we call {\em Birch lemma}, borrowing the name customarily used in the research of $L$-functions (see Corollary \ref{cor:signbirch}).

\subsection{Description of the cup product pairing} \label{sec:branch*}

Let $\mathcal{Y}^{(n)}_{\mathcal{K}_n}$ be the symmetric space ${\rm GL}_n({\mathbf Q}) \backslash  {\rm GL}_n({\mathbf Q}_{\mathbf A})   / {\rm SO}_n({\mathbf R})  {\mathcal K}_n$, which is equipped with the natural surjection $ {\rm p}_n:{\mathcal Y}^{(n)}_{\mathcal K_n} \longrightarrow Y^{(n)}_{\mathcal K_n}$. 
  We denote by  $ \iota: {\mathcal Y}^{(2)}_{\mathcal K_2} \longrightarrow {\mathcal Y}^{(3)}_{\mathcal K_3}$ the natural map induced by the inclusion $\iota: {\rm GL}_2({\mathbf Q}_{\mathbf A}) \to {\rm GL}_3({\mathbf Q}_{\mathbf A}) \, ; g\mapsto 
\begin{pmatrix}
 g & \\ & 1
\end{pmatrix}$.   
Note that $ \iota: {\mathcal Y}^{(2)}_{\mathcal K_2} \longrightarrow {\mathcal Y}^{(3)}_{\mathcal K_3}$ is a proper map. Hereafter, we always assume that ${\mathcal K}_2 = \iota^{-1}({\mathcal K}_3)$.
In this subsection, we consider a certain cup product on ${\mathcal Y}^{(2)}_{\mathcal K_2}$, 
which will play a crucial role for a cohomological interpretation of critical values of $L(s,{\mathcal M}(\pi^{(3)} \times \pi^{(2)} ))$.      

We shall use the same notation as the previous sections. In particular, let us consider the weight parameters $\boldsymbol{w}=(w,w,w)\in \Xi_3^{\mathrm{coh}}$ and $\boldsymbol{n}=(n_1,n_2)\in \Xi_2(\boldsymbol{w})$ for the local systems (see Sections~\ref{sec:LS} and \ref{sec:branch} for the definitions of $\Xi^{\mathrm{coh}}_3$ and $\Xi_2(\boldsymbol{w})$ respectively). 
Let $A$ be a subfield of ${\mathbf C}$. 
Recall from Section~\ref{sec:branch} that we have constructed the explicit branching rule isomorphism (\ref{eq:branch}):
    \begin{align*}
       {\boldsymbol \nabla}  = (\nabla^{{\boldsymbol n}})_{\boldsymbol{n}\in \Xi_2(\boldsymbol{w})}:  
       L^{(3)}(\boldsymbol{w}; A) |_{{\rm GL}_2( A )  }   \longrightarrow \bigoplus_{\boldsymbol{n}\in \Xi_2(\boldsymbol{w})} L^{(2)}(\boldsymbol{n}; A).
    \end{align*}
For each $\boldsymbol{n}=(n_1,n_2)\in \Xi_2(\boldsymbol{w})$, define a morphism between local systems on $\mathcal{Y}^{(2)}_{\mathcal{K}_2}$
\begin{align}\label{eq:nablocsys}
    \iota^* \mathrm{p}_3^*{\mathcal L}^{(3)} ( \boldsymbol{w} ; A)_{/ {\mathcal Y}^{(2)}_{{\mathcal K}_2} }  \longrightarrow  \mathrm{p}_2^* {\mathcal L}^{(2)} (\boldsymbol{n} ; A)_{/ {\mathcal Y}^{(2)}_{{\mathcal K}_2} }   ; 
     (g, P_{\iota (g)})  \longmapsto (g, \nabla^{\boldsymbol{n}} P_{ \iota(g) } )
\end{align}
induced by $\nabla^{\boldsymbol{n}}$, for which we use the same symbol $\nabla^{\boldsymbol{n}}$.

Now consider the ${\rm GL}_2({\mathbf Q})$-equivariant non-degenerate pairing induced by   
\begin{align*}
  [\cdot, \cdot]_{\boldsymbol{n}} \colon   L^{(2)} (\boldsymbol{n}, A) \otimes_A L^{(2)}(n_1, 0; A) \longrightarrow  
   L^{(2)}(0, [\boldsymbol{n}]; A  );\,  X^{i} Y^{n_1-i} \otimes X^{j} Y^{n_1-j} \longmapsto 
   \begin{cases}  (-1)^i \binom{n_1}{i}^{-1}   &  \text{for }i+j = n_1,    \\
                          0   &  \text{otherwise} \end{cases}    
\end{align*}
where we set $[\boldsymbol{n}]:=n_1+n_2$. Let $\widetilde{A}_{[\boldsymbol{n}]}$ be the local system on ${\mathcal Y}^{(2)}_{{\mathcal K}_2}$, which is determined by $L^{(2)}(0, [\boldsymbol{n}]; A  )$. For an  open compact subgroup $U$ of $\widehat{\mathbf Z}^\times$, 
   we define the narrow ray class group ${\rm Cl}^+_{\mathbf Q}(U) $ of $\mathbf{Q}$ modulo  $U$ by ${\rm Cl}^+_{\mathbf Q}(U) = {\mathbf Q}^\times \backslash {\mathbf Q}^\times_{\mathbf A} / {\mathbf R}^\times_{>0} U$. 
Then the determinant map defines a morphism
\begin{align*}
  {\rm det}_{\mathcal K_2} 
         \colon {\mathcal Y}^{(2)}_{\mathcal K_2}  \longrightarrow {\rm Cl}^+_{\mathbf Q}(\det {\mathcal K}_2) 
         ; \,  [g] \longmapsto [\det g].
\end{align*}
In the following arguments, we write ${\rm Cl}^+_{\mathbf Q}(\det {\mathcal K}_2)$ as ${\rm Cl}^+_{\mathbf Q}( {\mathcal K}_2)$ for the sake of simplicity.
For each $x\in {\rm Cl}^+_{\mathbf Q}({\mathcal K}_2)$, let ${\mathcal Y}^{(2)}_{\mathcal K_2, x}$ denote the inverse image of $x$ under $\det_{\mathcal{K}_2}$.
We use the same symbol  $\widetilde{A}_{[\boldsymbol{n}]}$ for the pullback of $\widetilde{A}_{[\boldsymbol{n}]}$  
        to ${\mathcal Y}^{(2)}_{\mathcal K_2, x}$ under the natural inclusion $ {\mathcal Y}^{(2)}_{{\mathcal K_2},x} \rightarrow  {\mathcal Y}^{(2)}_{\mathcal K_2}$.
Then  the paring $[\cdot, \cdot]_{\boldsymbol{n}}$ induces a cup product on cohomology groups of $\mathcal{Y}^{(2)}_{\mathcal{K}_2,x}$:  
\begin{align}\label{eq:cohcup}
 \cup \colon H^2( {\mathcal Y}^{(2)}_{\mathcal K_2, x},    \nabla^{\boldsymbol{n}} \iota^\ast  {\rm p}^\ast_3 {\mathcal L}^{(3)} ( \boldsymbol{w} ; A)   )   
          \otimes_A   H^1_{\rm c}( {\mathcal Y}^{(2)}_{\mathcal K_2, x},  {\rm p}^\ast_2 {\mathcal L}^{(2)} ( n_1, 0 ; A)   )    
  \longrightarrow   H^3_{\rm c}( {\mathcal Y}^{(2)}_{\mathcal K_2, x},  \widetilde{A}_{[\boldsymbol{n}]}   ).
\end{align} 
Let $\widetilde{A}$ be the trivial local system on ${\mathcal Y}^{(2)}_{\mathcal K_2, x}$, which is determined by $A$.    
Since a character 
\begin{align*}
    {\mathbf Q}^\times_{\mathbf A} \longrightarrow {\mathbf C}^\times;   
          x \longmapsto \frac{ |x|_{\mathbf A}  }{ x_\infty }
\end{align*}
takes values in ${\mathbf Q}^\times$,   
the following morphism between local systems on ${\mathcal Y}^{(2)}_{ \mathcal{K}_2, x }$ is well defined: 
\begin{align*}
    {\rm Tw}_{ [\boldsymbol{n}] }: \widetilde{A}_{[\boldsymbol{n}]   / {\mathcal Y}^{(2)}_{{\mathcal K}_2, x} }   
            \longrightarrow  { \widetilde{A} }_{/ {\mathcal Y}^{(2)}_{{\mathcal K}_2, x} }   ; 
     (g, a_g)  \longmapsto (g,  \frac{  |\det g|^{ [\boldsymbol{n}] }_{\mathbf A} }{(\det g_\infty)^{ [\boldsymbol{n}] }   }  a_g ).
\end{align*}
Composing (\ref{eq:cohcup}) with the map induced from $ {\rm Tw}_{ [\boldsymbol{n}] }$, the pairing (\ref{eq:cohcup}) is rewritten as 
\begin{align}\label{eq:cohcupTw}       
     \cup \colon H^2( {\mathcal Y}^{(2)}_{\mathcal K_2, x},    \nabla^{\boldsymbol{n}} \iota^\ast  {\rm p}^\ast_3 {\mathcal L}^{(3)} ( \boldsymbol{w} ; A)   )   
          \otimes_A   H^1_{\rm c}( {\mathcal Y}^{(2)}_{\mathcal K_2, x},  {\rm p}^\ast_2 {\mathcal L}^{(2)} ( n_1, 0 ; A)   )    
  \longrightarrow   H^3_{\rm c}( {\mathcal Y}^{(2)}_{\mathcal K_2, x},  \widetilde{A}    ) 
  \stackrel{ \sim }{ \longrightarrow } A. 
\end{align}
Here the rightmost map in (\ref{eq:cohcupTw}) is the same as
the upper horizontal one in the commutative diagram
\begin{align*}
   \begin{CD}
      H^{ 3 }_{\rm c}( {\mathcal Y}^{(2)}_{\mathcal K_2, x}, \widetilde{A} )   @>\sim>>  A   \\
      @VVV   @VVV \\
      H^{ 3 }_{\rm c}( {\mathcal Y}^{(2)}_{\mathcal K_2, x}, \widetilde{\mathbf C}  )   @>\sim>{\int}>  {\mathbf C}   
   \end{CD}
\end{align*}
where the vertical maps are induced by the natural inclusion $A \to {\mathbf C}$
           and the lower horizontal map is obtained by the integration on ${\mathcal Y}^{(2)}_{\mathcal K_2, x}$.

\subsection{Cup product of the Eichler--Shimura classes}\label{sec:cup}

In this subsection, we return to the automorphic setting and consider the cup product of the Eichler--Shimura classes $\delta^{(3)}(\boldsymbol{f})$ and $\delta^{(2)}(h^\pm)$ constructed in Section~\ref{sec:ES}. For $\boldsymbol{\lambda}^{(2)}=(\lambda_2,0) \in \Lambda^{\mathrm{coh}}_2$ and $\boldsymbol{\lambda}^{(3)}=(\lambda_3,\delta) \in \Lambda^{\mathrm{coh}}_3$ satisfying the inequality $\lambda_2<\lambda_3$, we consider a cohomological cuspidal automorphic representations $\pi^{(2)}$ and $\pi^{(3)}$ such that 
\begin{align*}
  & \pi^{(2)}_\infty \cong  D_{ \nu_2, l_2 }, \quad    
   \pi^{(3)}_\infty \cong  {\rm Ind}^{\rm GL_3({\mathbf R})}_{{\rm P}_{2,1}({\mathbf R})  } (   D_{ \nu_3, l_3 }  \boxtimes  \chi_{\nu, \delta_3}  ),  \\
  &         \lambda_2=l_2+1, \quad \nu_2 = -\frac{l_2}{2} +\frac{1}{2}, \quad  \quad  
      \lambda_3=l_3+1, \quad \nu_3  = \nu = -\frac{l_3}{2} + 1. 
\end{align*}
We specialize the construction of Section~\ref{sec:branch*} to  
\begin{align*}
 \boldsymbol{n}_{\boldsymbol{\lambda}}&=(n_{\boldsymbol{\lambda}},0)=(\lambda_2-2,0)=(l_2-1,0), \\
\boldsymbol{w}_{\boldsymbol{\lambda}}&=(w_{\boldsymbol{\lambda}},w_{\boldsymbol{\lambda}},w_{\boldsymbol{\lambda}})=\left(\dfrac{\lambda_3-3}{2}, \dfrac{\lambda_3-3}{2}, \dfrac{\lambda_3-3}{2}\right)=\left(\dfrac{l_3}{2}-1,\dfrac{l_3}{2}-1,\dfrac{l_3}{2}-1\right)
\end{align*}
as in (\ref{eq:w2}). We can easily calculate $\Xi_2(\boldsymbol{w}_{\lambda})$ as 
\begin{align*}
  \Xi_2(\boldsymbol{w}_{\boldsymbol{\lambda}}) &= \left\{  (n_1, n_2) \in {\mathbf Z}  \mid    w_{\boldsymbol{\lambda}} \leq n_1+n_2 \leq 2w_{\boldsymbol{\lambda}}, \,                                                                                 0 \leq n_2 \leq  w_{\boldsymbol{\lambda}}   \right\}    \\                
  &= \left\{  (n_1, n_2) \in {\mathbf Z} \ \middle| \   0 \leq n_1 \leq \dfrac{l_3}{2}-1,\, -n_1+ \frac{l_3}{2}-1  \leq n_2 \leq   \frac{l_3}{2}-1    \right\}     \\
              & \qquad \qquad     \cup  \left\{  (n_1, n_2) \in {\mathbf Z} \ \middle| \   \frac{l_3}{2}-1 < n_1 \leq  l_3-2,\,                                                                                           0  \leq n_2 \leq  -n_1 +l_3-2   \right\}.
\end{align*}

In order to consider the cup product of the Eichler--Shimura classes, we should focus on the branching rule map $\nabla^{\boldsymbol{n}}$ for $\boldsymbol{n}=(n_1,n_2)\in \Xi_2(\boldsymbol{w}_{\lambda})$ with $n_1=n_{\boldsymbol{\lambda}}=l_2-1$. The condition for $(n_{\boldsymbol{\lambda}},n_2)$ to be contained in $\Xi_2(\boldsymbol{w}_{\boldsymbol{\lambda}})$ is described as follows: 
\begin{multline}\label{eq:n2region}
   \begin{cases}
      \dfrac{l_3}{2}-l_2  \leq n_2 \leq   \dfrac{l_3}{2}-1 &   \text{for } 1\leq l_2 \leq \dfrac{l_3}{2},   \\  
      0  \leq n_2 \leq  l_3-l_2-1 &   \text{for } \dfrac{l_3}{2}    <    l_2 \leq  l_3-1
     \end{cases}  
\Longleftrightarrow 
    \begin{cases}
      \dfrac{l_3}{2}+1   \leq n_2 + l_2+1 \leq   \dfrac{l_3}{2}+l_2 &   \text{for }0\leq l_2 \leq \dfrac{l_3}{2},   \\  
      l_2+1 \leq n_2+l_2+1 \leq l_3    &  \text{for } \dfrac{l_3}{2}   <   l_2 \leq  l_3-1.
     \end{cases}
\end{multline}

If we set  $m=n_2+l_2+1$,   the region of $m$ in (\ref{eq:n2region}) completely coincides with 
that in Proposition \ref{prop:critregion}. We define $\mathrm{Crit}(\boldsymbol{\lambda})$ or $\mathrm{Crit}(\pi)$ as the set of integers $m$ satisfying the same inequalities as those which $n_2+l_2+1$ satisfies in (\ref{eq:n2region}), and for each $m\in \mathrm{Crit}(\boldsymbol{\lambda})$, set 
\begin{align} \label{eq:nm}
 \boldsymbol{n}_m = ( n_{m,1}, n_{m,2}) = (\lambda_2-2, m-\lambda_2)=(l_2-1,m-l_2-1).
\end{align}

\medskip
Now let us consider the cup product of the Eichler--Shimura classes. 
For $m\in \mathrm{Crit}(\boldsymbol{\lambda})$, let $\nabla^{\boldsymbol{n}_m} \iota^\ast {\rm p}^{\ast}_3 \delta^{(3)}(\boldsymbol{f})_x$
      (resp.\ ${\rm p}^{\ast}_2 \delta^{(2)}(h^\pm)_x$)
      be the restriction of $\nabla^{\boldsymbol{n}_m} \iota^* {\rm p}^{\ast}_3 \delta^{(3)}(\boldsymbol{f})$
           (resp.\ ${\rm p}^{\ast}_2 \delta^{(2)}(h^\pm)_x$)
       to ${\mathcal Y}^{(2)}_{ {\mathcal K}_2, x }$
       for $x \in {\rm Cl}^+_{\mathbf Q}({\mathcal K}_2)$.
Then the zeta integral for the Rankin--Selberg convolution ${\rm GL}_3 \times {\rm GL}_2$ is described as a summation of the cup products in (\ref{eq:cohcupTw}) of elements 
$ \nabla^{\boldsymbol{n}_m} \iota^\ast {\rm p}^{\ast}_3 \delta^{(3)}({\boldsymbol{f}})_x 
     \in H^2( {\mathcal Y}^{(2)}_{\mathcal K_2, x},     \nabla^{\boldsymbol{n}_m} \iota^\ast {\rm p}^\ast_3 {\mathcal L}^{(3)} ( \boldsymbol{w}_{\boldsymbol{\lambda}} ; A)   )$   
        and $ {\rm p}^{\ast}_2 \delta^{(2)}(h^\pm)_x  \in H^1_{\rm c}( {\mathcal Y}^{(2)}_{\mathcal K_2, x},  {\rm p}^\ast_2 {\mathcal L}^{(2)} ( \boldsymbol{n}_{\boldsymbol{\lambda}} ; A)   )$.
To be more precise, the evaluation of the summation
\begin{align}\label{eq:defcup}
  \sum_{x \in {\rm Cl}^+_{\mathbf Q} ({\mathcal K}_2)  }
   \varphi(x) 
     {\rm Tw}_{ [\boldsymbol{n}_m] } \left(\nabla^{\boldsymbol{n}_m} \iota^\ast  {\rm p}^{\ast}_3 \delta^{(3)}(\boldsymbol{f})_x  
    \cup 
    {\rm p}^{\ast}_2 \delta^{(2)}(h^\pm)_x\right)
\end{align}
coincides with the global zeta integral, where $\varphi \colon {\mathbf Q}^\times \backslash {\mathbf Q}^\times_{\mathbf A}  \to {\mathbf C}^\times$ is a Hecke character of finite order factoring through $\mathrm{Cl}^+_{\mathbf{Q}}(\mathcal{K}_2)$. 
To evaluate the cup product in an explicit way, it is necessary to write down 
   $\nabla^{\boldsymbol{n}_m} \iota^\ast  {\rm p}^{\ast}_3 \delta^{(3)}(\boldsymbol{f})$ 
    and ${\rm p}^{\ast}_2 \delta^{(2)}(h^\pm)$   
   as differential forms (with local coefficients) on ${\mathcal Y}^{(2)}_{\mathcal K_2}$. We shall do it in the next subsection.

For later use, we here summarize as the following Lemma the computation of the image of the polynomial ${\mathcal P}_{\alpha,\beta}(X, Y, Z, A, B, C)\in L^{(3)}(w_{\boldsymbol{\lambda}};\mathbf{C})$ introduced in Section~\ref{sec:ES} under the map $\nabla^{\boldsymbol{n}}$.  
We postpone its proof to Appendix \ref{sec:AppB} because it is based on a little  complicated but straightforward calculations.   
For $e, i, j , k\in {\mathbf Z}$ with $e\geq 0$, define an integer ${}_{e}H_{i,j,k}\in {\mathbf Z}$ by   
\begin{align}\label{eq:defH}
    (X+Y+Z)^e = \sum_{i,j,k}  {}_{e}H_{i,j,k}  X^i Y^j Z^k.  
\end{align}

\begin{lem}\label{lem:NabPab}
{\itshape 
Set $S=-X+\sqrt{-1}Y$ and $T= X+\sqrt{-1}Y$.  
Then, for every $\boldsymbol{n}=(n_1,n_2) \in \Xi_2(\boldsymbol{w}_{\boldsymbol{\lambda}})$, we have 
\begin{multline*} 
        \nabla^{\boldsymbol{n}} {\mathcal P}_{\alpha, \beta} (X, Y)
  = 2^{-n_1}(-1)^{w_{\boldsymbol{\lambda}}+\alpha}\sqrt{-1}^{n_2+w_{\boldsymbol{\lambda}}}
           S^{\frac{n_1+\alpha -\beta}{2}} T^{ \frac{n_1- \alpha + \beta}{2}}  \\     \times   \sum_{ b\in \mathbf{Z}} (-1)^b
                         {}_{w_{\boldsymbol{\lambda}}} H_{b+\frac{n_1+\alpha+\beta}{2}+n_2-w_{\boldsymbol{\lambda}},-b+\frac{n_1-\alpha-\beta}{2},2w_{\boldsymbol{\lambda}}-n_1-n_2}
                         \cdot  {}_{w_{\boldsymbol{\lambda}}} H_{-(b+\beta)+w_{\boldsymbol{\lambda}}-n_2, b+\beta, n_2}.
\end{multline*}
 In particular, if $\alpha-\beta = \pm n_1$ holds, $\nabla^{\boldsymbol{n}} {\mathcal P}_{\alpha, \beta} (X, Y)$ is described as
\begin{align*}
        \nabla^{\boldsymbol{n}} {\mathcal P}_{\alpha, \beta} (X, Y)
  =  (-2^{-1})^{n_1}  \sqrt{-1}^{  n_2 + w_{\boldsymbol{\lambda}} } 
         \binom{w_{\boldsymbol{\lambda}}}{ n_1+n_2-w_{\boldsymbol{\lambda}} }  
        \binom{w_{\boldsymbol{\lambda}}}{ w_{\boldsymbol{\lambda}}-n_2 }   
      \times \begin{cases}  (-1)^{n_2} T^{n_1}  &   \text{when } \alpha - \beta =- n_1,  \\                                                  
                                     (-1)^{w_{\boldsymbol{\lambda}}}      S^{n_1}  &    \text{when }\alpha - \beta = n_1.  
                                          \end{cases}          
\end{align*}
}
\end{lem} 
\begin{proof}
See Lemmas \ref{lem:AppB2} and \ref{lem:AppB3} in Appendix \ref{sec:AppB}. 
\end{proof}

\subsection{Birch lemma}\label{sec:birch}

To evaluate the cup product (\ref{eq:defcup}), 
  it is necessary to write it down as an integral on ${\mathcal Y}^{(2)}_{\mathcal K_2}$. We can do it by using  the explicit form of the Eichler--Shimura map introduced in Section \ref{sec:ES}.    
Some of explicit formulas proposed here require longsome but straightforward computations for their proofs, and hence we omit such calculations in this subsection and summarize them in Appendix \ref{sec:AppA}. 

First let us recall the  explicit formula for ${\rm p}^\ast_2\delta^{(2)}(h^\pm)$, which is well known in terms of elliptic modular forms.

\begin{lem}
{\itshape 
For each $g\in {\mathcal Y}^{(2)}_{\mathcal K_2}$,  
    we have 
\begin{align*}
   {\rm p}^\ast_2\delta^{(2)}(h^\pm)_g
   =   h^\pm(g) 
        \varrho^{(2)}_{\boldsymbol{n}_{\boldsymbol{\lambda}}}(g_\infty)( \mp X + \sqrt{-1}Y)^{n_{\boldsymbol{\lambda}}} 
        ( \pm \frac{1}{2} {\rm d}E - \frac{\sqrt{-1} }{2}  {\rm d}H   ) \circ L^{-1}_{g_\infty},         
\end{align*}  
where $L_{g_\infty}$ denotes the left translation $\left(T_{g_{\mathrm{fin}}}{\mathcal Y}^{(2)}_{\mathcal K_2}\right)_{\mathbf{C}}\cong \left(\mathfrak{gl}_2(\mathbf{R})/\mathfrak{so}_2(\mathbf{R})\right)_{\mathbf{C}} \to \left(T_g{\mathcal Y}^{(2)}_{\mathcal K_2}\right)_{\mathbf{C}}$ by $g_\infty$.
}
\end{lem}

Due to the Iwasawa decomposition of $\mathrm{GL}_2(\mathbf{R})$, we define the coordinates $(x_2,y_1,y_2)$ of the archimedean component of ${\mathcal Y}^{(2)}_{\mathcal K_2}$ as 
\begin{align*}
{\rm GL}_2({\mathbf Q})
 y_1 \begin{pmatrix}  y_2^{\frac{1}{2}} &  y_2^{-\frac{1}{2}} x_2  \\   0  & y_2^{ -\frac{1}{2} }   \end{pmatrix}  g_{\rm fin} {\rm SO}_2({\mathbf R}){\mathcal K}_2,  \quad g_{{\mathrm{fin}}} \in {\rm GL}_2({\mathbf Q}_{\mathbf{A}, \mathrm{fin}}).
\end{align*}
With respect to this coordinate, the differential form ${\boldsymbol \xi}_\pm
   :=  (  \pm \frac{1}{2}{\rm d}E - \frac{\sqrt{-1}}{2}{\rm d}H)\circ L^{-1}_g$ is described as 
\begin{align*}
  {\boldsymbol \xi}_\pm 
   =  e^{2\sqrt{-1}\theta  }    \frac{   \pm {\rm d} x_2 + \sqrt{-1} {\rm d}y_2 }{y_2}
\end{align*}
for each $g=y_1\begin{pmatrix} y_2^{\frac{1}{2}} & y_2^{-\frac{1}{2}} x_2  \\ 0 & y_2^{-\frac{1}{2}} \end{pmatrix} 
            \begin{pmatrix} \cos\theta & \sin\theta \\ -\sin\theta & \cos\theta  \end{pmatrix}$.

We next describe an explicit formula for $\nabla^{\boldsymbol{n}_m}  \iota^*{\rm p}^\ast_3\delta^{(3)} (\boldsymbol{f})$. Using the explicit Iwasawa decomposition of $\mathrm{GL}_3(\mathbf{R})$ proposed in Lemma~\ref{lem:explicit_Iwasawa},  we define  the coordinates $(t,x_1,x_2,x_3,y_1,y_2)$  of the archimedean component of ${\mathcal Y}^{(3)}_{\mathcal K_3}$ as 
\begin{align*}
{\rm GL}_3({\mathbf Q})
 t \begin{pmatrix}  y_1y_2 & y_1 x_2 & x_3  \\ 0  &  y_1 & x_1 \\ 0 & 0 & 1   \end{pmatrix} 
   g_{\rm fin}  {\rm SO}_3({\mathbf R}){\mathcal K}_3, \quad g_{{\mathrm{fin}}} \in {\rm GL}_3({\mathbf Q}_{\mathbf{A}, \mathrm{fin}}).
\end{align*}
Proposition~\ref{prop:bfomega} describes how the basis $\{  {\boldsymbol \omega}_3,  \ldots,  {\boldsymbol \omega}_{-3} \}$ of $\bigwedge^2 {\mathcal P}^\ast_{3, {\mathbf C}}$ introduced in (\ref{eq:def_bfomega})
 is presented with respect to this coordinates; namely
\begin{align} \label{eq:QQ}
\begin{pmatrix}
 {\boldsymbol \omega}_3 & \boldsymbol{\omega}_2 &  \dotsc &  {\boldsymbol \omega}_{-3}
\end{pmatrix}   \circ L^{-1}_{g_\infty}   =  \begin{pmatrix}
				     {\rm d}y_1 \wedge {\rm d}y_2 & {\rm d}y_1 \wedge {\rm d}x_1 & \dotsc & {\rm d}x_2 \wedge {\rm d}x_3
				    \end{pmatrix}Q^2
\end{align}
 holds where $Q^2$ is the matrix introduced in Proposition \ref{prop:bfomega} (in the right-hand side, the $2$-forms are ordered lexicographically with respect to the order ${\rm d}y_1 \prec {\rm d}y_2 \prec {\rm d}x_1 \prec {\rm d}x_2 \prec {\rm d}x_3$).
Since $\iota$ induces $(x_2,y_1,y_2) \mapsto (0,0,x_2,0,y_1y_2^{-\frac{1}{2}},y_2)$, the pullback of the differential forms on $\mathcal{Y}^{(3)}_{\mathcal{K}_3}$ along $\iota$ is described as 
\begin{align*}
 \iota^* \begin{pmatrix}
  {\rm d}t & {\rm d}x_1 & {\rm d}x_2 & {\rm d}x_3 & {\rm d}y_1 & {\rm d}y_2
 \end{pmatrix} =  \begin{pmatrix}
			0 & 0 & {\rm d}x_2 & 0 & y_2^{-\frac{1}{2}}{\rm d}y_1-\dfrac{1}{2}y_1y_2^{-\frac{3}{2}}{\rm d}y_2 & {\rm d}y_2 
		       \end{pmatrix}.
\end{align*} 
Taking this into accounts, we immediately obtain 
\begin{align*}
 \iota^\ast  {\rm p}^\ast_3{\boldsymbol \omega}_{\pm3} 
 &= \iota^\ast  {\rm p}^\ast_3{\boldsymbol \omega}_{\pm1}
 =  0, 
 \quad  \iota^\ast  {\rm p}^\ast_3{\boldsymbol \omega}_0 
           =  \frac{\sqrt{-1}}{2y^2_2} {\rm d}y_2 \wedge {\rm d}x_2,  \\   
 \iota^\ast  {\rm p}^\ast_3{\boldsymbol \omega}_{\pm 2} 
           &= \pm \frac{ 1 }{2y_1y_2}  {\rm d}y_1 \wedge {\rm d}y_2 
                                   - \frac{\sqrt{-1}}{2y_1y_2} {\rm d}y_1 \wedge {\rm d}x_2
                                   - \frac{\sqrt{-1}}{4y^2_2} {\rm d}y_2 \wedge {\rm d}x_2   
\end{align*}
by computing the product of matrices in (\ref{eq:QQ}) explicitly.
Hence 
the differential forms $\iota^\ast  {\rm p}^\ast_3{\boldsymbol \omega}_i \wedge {\boldsymbol \xi}_\mp$ for $i=3, \ldots, -3$ are computed as follows:
\begin{align}\label{eq:DiffFormPM}
\begin{aligned}
  \iota^\ast  {\rm p}^\ast_3{\boldsymbol \omega}_i \wedge {\boldsymbol \xi}_\mp 
 & =    0,  \quad \text{for }i \neq \pm 2, \\
  \iota^\ast  {\rm p}^\ast_3{\boldsymbol \omega}_{\pm 2} \wedge {\boldsymbol \xi}_\mp 
 & =  \left(  
          \pm \frac{ 1 }{ 2 y_1y_2}  {\rm d}y_1 \wedge {\rm d}y_2 
                                   - \frac{\sqrt{-1}}{2y_1y_2} {\rm d}y_1 \wedge {\rm d}x_2
                                   - \frac{\sqrt{-1}}{4y^2_2} {\rm d}y_2 \wedge {\rm d}x_2
           \right) 
           \wedge 
           \frac{ \mp {\rm d} x_2 + \sqrt{-1}{\rm d}y_2  }{y_2}  \\
  & = \frac{ 1 +1 }{2} \times \frac{{\rm d}y_1 \wedge {\rm d}x_2 \wedge {\rm d}y_2}{ y_1y^2_2}  
         =   \frac{{\rm d}y_1 \wedge {\rm d}x_2 \wedge {\rm d}y_2}{ y_1y^2_2} \quad (\text{double sign in the same order}). 
\end{aligned}
\end{align}
Now define  ${\rm d}g={\rm d}g_\infty {\rm d}g_{\mathrm{fin}}$ to be a Haar measure on ${\rm GL}_2({\mathbf Q}_{\mathbf A})$ which satisfies 
\begin{align}\label{eq:Haar}
   {\rm d}g_\infty =   \frac{{\rm d}y_1 \wedge {\rm d}x_2 \wedge {\rm d}y_2}{ y_1y^2_2} {\rm d} u  \qquad 
          \text{for }g_\infty    =   y_1  \begin{pmatrix}  y^{\frac{1}{2}}_2   & y^{-\frac{1}{2}}_2 x_2  \\ 0 & y^{ - \frac{1}{2}}_2 \end{pmatrix}u,  \; 
          u\in {\rm O}_2({\mathbf R}),   \; 
           {\rm vol}({\rm d}u, {\rm O}_2({\mathbf R})) =1
\end{align}
and ${\rm vol}({\rm d} g_{\rm fin}, {\rm GL}_2(\widehat{\mathbf Z}))=1$.
%
We summarize the arguments so far in the following proposition:

\begin{prop}\label{prop:intsign}
{\itshape 
For $\boldsymbol{n}_m=(n_{m,1},n_{m,2})=(\lambda_2-2,m-\lambda_2)=(l_2-1,m-l_2-1)$, set 
\begin{align*}
   \epsilon &= \bigotimes_{v\in \Sigma_{\mathbf{Q}}\setminus \{\infty\}} 1_2 \otimes \begin{pmatrix} -1 & 0 \\ 0 & 1 \end{pmatrix} \in {\rm GL}_2({\mathbf Q_A}), \\
    \quad C(\boldsymbol{w}_{\boldsymbol{\lambda}}, \boldsymbol{n}_m)  &=   
         \binom{w_{\boldsymbol{\lambda}}}{ n_{m,1}+n_{m,2}-w_{\boldsymbol{\lambda}} }  
        \binom{w_{\boldsymbol{\lambda}}}{ w_{\boldsymbol{\lambda}}-n_{m,2} } =\binom{\frac{l_3}{2}-1}{m-\frac{l_3}{2}-1}\binom{\frac{l_3}{2}-1}{\frac{l_3}{2}+l_2-m}. 
\end{align*}
Then we have 
\begin{align*}
   {\rm Tw}_{ [\boldsymbol{n}_m] }& \left(\nabla^{\boldsymbol{n}_m} \iota^* {\rm p}^{\ast}_3 \delta^{(3)}(\boldsymbol{f})_x  \cup {\rm p}^{\ast}_2\delta^{(2)}(h^+)_x  \right)   \\
&=  
\sqrt{-1}^{\frac{l_3}{2}+2l_2-m-1}C(\boldsymbol{w}_{\boldsymbol{\lambda}}, \boldsymbol{n}_m)
   \cdot 2[ {\rm GL}_2(\widehat{\mathbf Z})  : {\mathcal K}_2 ]  
      \int_{ {\mathcal Y}^{(2)}_{\mathcal K_2},x  } 
            f^{\boldsymbol{\lambda}}_{ -\lambda_2} \left( \iota(g) \right) h^+ (g)
             \lvert \det g\rvert_{\mathbf{A}}^{ m-2} 
              {\rm d}g.
\end{align*}
We also find that 
\begin{align*}
  {\rm Tw}_{ [\boldsymbol{n}_m] } \left(\nabla^{\boldsymbol{n}_m} \iota^* {\rm p}^{\ast}_3 \delta^{(3)}(\boldsymbol{f})_{x\det \epsilon }  \cup {\rm p}^{\ast}_2\delta^{(2)}(h^-)_{ x \det \epsilon  }\right)
 = (-1)^{m+\delta+1+w_{\boldsymbol{\lambda}}}
     {\rm Tw}_{ [\boldsymbol{n}_m] } \left(\nabla^{\boldsymbol{n}_m} \iota^* {\rm p}^{\ast}_3 \delta^{(3)}(\boldsymbol{f})_x  \cup {\rm p}^{\ast}_2\delta^{(2)}(h^+)_x \right). 
\end{align*}
}
\end{prop}
\begin{proof}
The identities (\ref{eq:DiffFormPM}) and the normalization (\ref{eq:Haar}) of the measure ${\rm d}g$ imply that 
\begin{align*}
  &   {\rm Tw}_{  [  \boldsymbol{n}_m ] } \left(\nabla^{\boldsymbol{n}_m} \iota^* {\rm p}^{\ast}_3 \delta^{(3)}(\boldsymbol{f})_x \cup {\rm p}^{\ast}_2\delta^{(2)}(h^+)_x  \right)   \\      
&\begin{aligned}
 &=   \sum_\alpha  2[ {\rm GL}_2(\widehat{\mathbf Z})  : {\mathcal K}_2 ]  
          \int_{ {\mathcal Y}^{(2)}_{\mathcal K_2}, x  }    
          f^{\boldsymbol{\lambda}}_{\alpha } (\iota(g))  
          h^+(g)
   \left[   \varrho^{(2)}_{\boldsymbol{n}_m }(g_\infty)  \nabla^{\boldsymbol{n}_m} {\mathcal P}_{\alpha, -2} (X, Y) , 
               \varrho^{(2)}_{\boldsymbol{n}_{\boldsymbol{\lambda}}}(g_\infty)(-X+ \sqrt{-1} Y)^{n_{\boldsymbol{\lambda}}}   \right]_{\boldsymbol{n}_m}            
   \frac{  |\det g|^{ [\boldsymbol{n}_m ] }_{\mathbf A} }{(\det g_\infty)^{ [\boldsymbol{n}_m ] }   } 
         {\rm d}g   \\
& =   2[ {\rm GL}_2(\widehat{\mathbf Z})  : {\mathcal K}_2 ]   \sum_\alpha
          [   \nabla^{\boldsymbol{n}_m} {\mathcal P}_{\alpha, -2} (X, Y) , 
                (-X+ \sqrt{-1} Y)^{n_{\boldsymbol{\lambda}}}   ]_{\boldsymbol{n}_m}                     
          \int_{ {\mathcal Y}^{(2)}_{\mathcal K_2}, x  }    
          f^{\boldsymbol{\lambda}}_{ \alpha } (\iota(g))  
          h^+(g)
          \lvert\det g \rvert_{\mathbf{A}}^{ [\boldsymbol{n}_m] }
         {\rm d}g.    
\end{aligned}
\end{align*}
Set $S=-X+\sqrt{-1}Y$, $T= X+\sqrt{-1}Y$ and $g_0= \begin{pmatrix} -1 & 1 \\ \sqrt{-1} & \sqrt{-1} \end{pmatrix}$. 
Then, due to the equivariance of the pairing $[\cdot,\cdot]_{\boldsymbol{n}_m}$, we find that 
\begin{align*}
[T^{n_{\boldsymbol{\lambda}}}, S^{n_{\boldsymbol{\lambda}}}]_{\boldsymbol{n}_m}
= [\det(g_0)^{-n_{m,2}}\varrho^{(2)}_{\boldsymbol{n}_m}(g_0) Y^{n_{\boldsymbol{\lambda}}}, \varrho^{(2)}_{\boldsymbol{n}_{\boldsymbol{\lambda}}}(g_0) X^{n_{\boldsymbol{\lambda}}}]_{\boldsymbol{n}_m}
= (\det g_0)^{[\boldsymbol{n}_m]-n_{m,2}}[Y^{n_{\boldsymbol{\lambda}}}, X^{n_{\boldsymbol{\lambda}}}]_{\boldsymbol{n}_m} 
=(-2\sqrt{-1})^{n_{m,1}}.
\end{align*}
Hence Lemma \ref{lem:NabPab} yields that 
\begin{align}  \label{eq:pairing_ST}
\begin{aligned}
      {} [ \nabla^{\boldsymbol{n}_m} {\mathcal P}_{ \alpha,  - 2} (X,Y),  S^{n_{\boldsymbol{\lambda}}}]_{\boldsymbol{n}_m}  
  &=  \delta_{\alpha, -n_{m,1}-2} 
         [T^{n_{\boldsymbol{\lambda}},}, S^{n_{\boldsymbol{\lambda}}}]_{\boldsymbol{n}_m} 2^{-n_{m,1}}(-1)^{[\boldsymbol{n}_m]}\sqrt{-1}^{n_{m,2}+w_{\boldsymbol{\lambda}}} 
        C({\boldsymbol{w}_{\boldsymbol{\lambda}}}, \boldsymbol{n}_m)  \\
 & = \delta_{\alpha, -\lambda_2} \sqrt{-1}^{n_{m,1}-n_{m,2} + w_{\boldsymbol{\lambda}}}  C(\boldsymbol{w}_{\boldsymbol{\lambda}}, \boldsymbol{n}_m), \\
   [ \nabla^{\boldsymbol{n}_m} {\mathcal P}_{ \alpha,   2} (X,Y),  T^{n_{\boldsymbol{\lambda}}}]_{\boldsymbol{n}_m}  
 &=  \delta_{\alpha, n_{m,1}+2} [S^{n_{\boldsymbol{\lambda}}},T^{n_{\boldsymbol{\lambda}}}]_{\boldsymbol{n}_m} 2^{-n_{m,1}}(-1)^{w_{\boldsymbol{\lambda}}+n_{m,1}}\sqrt{-1}^{n_{m,2}+w_{\boldsymbol{\lambda}}} C(\boldsymbol{w}_{\boldsymbol{\lambda}},\boldsymbol{n}_m) \\
&= \delta_{\alpha,\lambda_2}  \sqrt{-1}^{-n_{m,1}+n_{m,2} - w_{\boldsymbol{\lambda}}}   C(\boldsymbol{w}_{\boldsymbol{\lambda}},\boldsymbol{n}_m).
\end{aligned}
\end{align}
Here $\delta_{\ast, \ast}$ denotes Kronecker's delta.  The formulas $[\boldsymbol{n}_m]=m-2$, $n_{m,1}-n_{m,2}+w_{\boldsymbol{\lambda}}=\dfrac{l_3}{2}+2l_2-m-1$ and the first identity of (\ref{eq:pairing_ST}) show the first statement.  

We prove the second statement. 
A similar argument  shows that 
\begin{align*}
     {\rm Tw}_{  [  \boldsymbol{n}_m ] } & \left(\nabla^{\boldsymbol{n}_m}  \iota^\ast {\rm p}^{\ast}_3 \delta^{(3)}(\boldsymbol{f})_{x\det \epsilon}  \cup {\rm p}^{\ast}_2\delta^{(2)}(h^-)_{x\det \epsilon} \right)    \\      
  &=   2[ {\rm GL}_2(\widehat{\mathbf Z})  : {\mathcal K}_2 ]   \sum_\alpha
          \left[   \nabla^{\boldsymbol{n}_m} {\mathcal P}_{\alpha, 2} (X, Y) , 
                (X+ \sqrt{-1} Y)^{n_{\boldsymbol{\lambda}}}   \right]_{\boldsymbol{n}_m}                     
          \int_{ {\mathcal Y}^{(2)}_{\mathcal K_2}, x\det \epsilon  }    
          f^{\boldsymbol{\lambda}}_{ \alpha } (\iota(g))  
          h^-(g)
          \lvert \det g\rvert_{\mathbf{A}}^{ [\boldsymbol{n}_m] }
        {\rm d}g   \\
  &=  2[ {\rm GL}_2(\widehat{\mathbf Z})  : {\mathcal K}_2 ]    \sum_\alpha
          [    \nabla^{\boldsymbol{n}_m} {\mathcal P}_{\alpha, 2} (X, Y) , 
                (X+ \sqrt{-1} Y)^{n_{\boldsymbol{\lambda}}}   ]_{\boldsymbol{n}_m}                     
          \int_{   {\mathcal Y}^{(2)}_{\mathcal K_2}, x   }    
          f^{\boldsymbol{\lambda}}_{ \alpha } (\iota(g \epsilon ))  
          h^-(g \epsilon )
          \lvert\det g \epsilon \rvert_{\mathbf{A}}^{ [\boldsymbol{n}_m] }
         {\rm d}(g\epsilon)   \\
 & 
     =  2[ {\rm GL}_2(\widehat{\mathbf Z})  : {\mathcal K}_2 ]    \sum_\alpha 
          [    \nabla^{\boldsymbol{n}_m} {\mathcal P}_{\alpha, 2} (X, Y) , 
                (X+ \sqrt{-1} Y)^{n_{\boldsymbol{\lambda}}}   ]_{\boldsymbol{n}_m}    \int_{   {\mathcal Y}^{(2)}_{\mathcal K_2}, x   }    
          (-1)^{\delta}  f^{\boldsymbol{\lambda}}_{ -\alpha } (\iota( g ))  
          h^+(g  )
          \lvert\det g \rvert_{\mathbf{A}}^{ [\boldsymbol{n}_m] }
        \cdot (-1) {\rm d} g. 
\end{align*}
This and the second identity of (\ref{eq:pairing_ST}) show that 
\begin{align*}
    {\rm Tw}_{  [  \boldsymbol{n}_m ] } & \left( \nabla^{\boldsymbol{n}_m}   \iota^* {\rm p}^{\ast}_3 \delta^{(3)}(\boldsymbol{f})_{ x \det \epsilon}  \cup {\rm p}^{\ast}_2\delta^{(2)}(h^-)_{ x  \det \epsilon }  \right) \\
 & =       \sqrt{-1}^{-n_{m,1}+n_{m,2} - w_{\boldsymbol{\lambda}}}  C(\boldsymbol{w}_{\boldsymbol{\lambda}},\boldsymbol{n}_m)    
     \cdot (-1)^{\delta + 1} 
      2[ {\rm GL}_2(\widehat{\mathbf Z})  : {\mathcal K}_2 ] 
     \int_{   {\mathcal Y}^{(2)}_{\mathcal K_2}, x   }    
           f^{\boldsymbol{\lambda}}_{ -\lambda_2 } (\iota( g ))  
          h^+(g  )
          \lvert\det g \vert_{\mathbf{A}}^{ [\boldsymbol{n}_m] }
        {\rm d} g    \\ 
  &=   (-1)^{  [\boldsymbol{n}_m] + w_{\boldsymbol{\lambda}} + \delta+ 1} 
          \nabla^n \iota^*{\rm p}^{\ast}_3 \delta^{(3)}(\boldsymbol{f})_{x}  \cup {\rm p}^{\ast}_2\delta^{(2)}(h^+)_{x}. 
\end{align*}
This proves the statement.
\end{proof}

\begin{cor}\label{cor:signbirch}
(Birch Lemma)
{\itshape 
Let $\varphi : {\mathbf Q}^\times \backslash {\mathbf Q}^\times_{\mathbf A} \to {\mathbf C}^\times$ be a  Hecke character of finite order. 
Suppose that $\varphi$ is trivial on $\det \mathcal{K}_2$.
Then we have 
\begin{align*}
 I(m, \boldsymbol{f},  h^\pm , \varphi  )
  &:=   \sum_{x\in {\rm Cl}^+_{\mathbf Q}({\mathcal K}_2)}  \varphi(x)        
      {\rm Tw}_{ [\boldsymbol{n}_m] } \left(
       \nabla^{\boldsymbol{n}_m} \iota^* {\rm p}^{\ast}_3 \delta^{(3)}(\boldsymbol{f})_x  \cup \left(  {\rm p}^{\ast}_2\delta^{(2)}(h^+)_x  \pm  {\rm p}^{\ast}_2\delta^{(2)}(h^-)_x \right)   \right)   \\
  &
\begin{aligned}
 \; \, =          \sqrt{-1}^{\frac{l_3}{2}+2l_2-m-1} &
         C(\boldsymbol{w}_{\boldsymbol{\lambda}}, \boldsymbol{n}_m) \cdot  2[ {\rm GL}_2(\widehat{\mathbf Z})  : {\mathcal K}_2 ]   \\
  &    \times   \left(  1 \pm 
                             (-1)^{m+\delta+1+w_{\boldsymbol{\lambda}}}  
                                   \varphi_\infty(-1)  \right)   
          \int_{ {\mathcal Y}^{(2)}_{\mathcal K_2}  } 
            f^{\boldsymbol{\lambda}}_{ -\lambda_2 } \left( \iota(g) \right) 
             h^+ (g)  
             \varphi(\det g) |\det g|^{m-2}_{\mathbf A}  {\rm d}g.   
\end{aligned}
\end{align*}
}
\end{cor}
\begin{proof}
Proposition \ref{prop:intsign} shows that 
\begin{align*}
   I(m, \boldsymbol{f},  h^\pm , \varphi  )  
  &=     \sum_{x\in {\rm Cl}^+_{\mathbf Q}({\mathcal K}_2)}  \varphi(x)        
         {\rm Tw}_{ [\boldsymbol{n}_m] } \left(\nabla^{\boldsymbol{n}_m} \iota^* {\rm p}^{\ast}_3 \delta^{(3)}(\boldsymbol{f})_x  \cup   {\rm p}^{\ast}_2\delta^{(2)}(h^+)_x  \right)
     \\
     & \quad \quad \quad \quad \quad 
         \pm \sum_{x\in {\rm Cl}^+_{\mathbf Q}({\mathcal K}_2)}   \varphi( x \det \epsilon ) 
               {\rm Tw}_{ [\boldsymbol{n}_m] }  \left( \nabla^{\boldsymbol{n}_m} \iota^* {\rm p}^{\ast}_3 \delta^{(3)}(\boldsymbol{f})_{x\det \epsilon}  \cup   {\rm p}^{\ast}_2\delta^{(2)}(h^-)_{x \det \epsilon} \right)   \\
  &=     \sum_{x\in {\rm Cl}^+_{\mathbf Q}({\mathcal K}_2)}  \varphi(x)        
       {\rm Tw}_{ [\boldsymbol{n}_m] } \left(\nabla^{\boldsymbol{n}_m} \iota^* {\rm p}^{\ast}_3 \delta^{(3)}(\boldsymbol{f})_x  \cup   {\rm p}^{\ast}_2\delta^{(2)}(h^+)_x  \right)
     \\
     & \quad \quad \quad \quad \quad 
       \pm   (-1)^{m+\delta+1+w_{\boldsymbol{\lambda}}}     \sum_{x\in {\rm Cl}^+_{\mathbf Q}({\mathcal K}_2)} 
                     \varphi_\infty(-1) \varphi(x)
                     {\rm Tw}_{ [\boldsymbol{n}_m] } \left( \nabla^{\boldsymbol{n}_m} \iota^* {\rm p}^{\ast}_3 \delta^{(3)}(\boldsymbol{f})_{x}  \cup   {\rm p}^{\ast}_2\delta^{(2)}(h^+)_{x}   \right) \\
  &=   \left(  1 \pm  (-1)^{m+\delta+1+w_{\boldsymbol{\lambda}}}
                               \varphi_\infty(-1)  \right)
          \sum_{x\in {\rm Cl}^+_{\mathbf Q}({\mathcal K}_2)}  \varphi(x)        
      {\rm Tw}_{ [\boldsymbol{n}_m] }\left( \nabla^{\boldsymbol{n}_m} \iota^* {\rm p}^{\ast}_3 \delta^{(3)}(\boldsymbol{f})_x  \cup   {\rm p}^{\ast}_2\delta^{(2)}(h^+)_x\right).  
\end{align*}
This proves the statement.
\end{proof}

Now define the zeta integral ${\mathcal Z}(s; \boldsymbol{f}, h^+, \varphi)$ to be 
\begin{align*}
{\mathcal Z}(s; \boldsymbol{f}, h^+, \varphi)
  =  
      \int_{ {\mathcal Y}^{(2)}_{\mathcal K_2}  } 
            f^{\boldsymbol{\lambda}}_{ - \lambda_2} \left( \iota(g) \right) h^+ (g)
             \varphi(\det g ) |\det g|^{ s -\frac{1}{2}}_{\mathbf A} 
               {\rm d}g. 
\end{align*}
In the next section, we will calculate explicit evaluations of the zeta integrals ${\mathcal Z}(m-\frac{3}{2}; \boldsymbol{f}, h^+,\mathbf{1})$ for each $m\in \mathrm{Crit}(\boldsymbol{\lambda})$,
and relate them with the critical values $L(m, {\mathcal M}(\pi^{(3)} \times \pi^{(2)}))$ for certain vectors $\boldsymbol{f}$ and $h^\pm$. 
We also clarify the relation between these critical values and the periods of ${\mathcal M}(\pi^{(3)}\times \pi^{(2)})$. 
The choice of data, namely the vectors $\boldsymbol{f}$, $h^\pm$ and the periods, will be also  proposed there.

\begin{rem}
Corollary \ref{cor:signbirch} implies that the condition on signature $(-1)^{m+\delta+1+w_{\boldsymbol{\lambda}}}  \neq \pm \varphi_\infty(-1)$ is necessary for the zeta integral $\mathcal{Z}(m-\frac{3}{2};\boldsymbol{f},h^+,\varphi)$ (and thus the critical value $L(m,\mathcal{M}(\pi^{(3)}\times \pi^{(2)})(\varphi))$) not to vanish. 
Such a condition on signature also appears in \cite[Theorem 1.1]{rag16}; see also \cite[Sections 2.5.3.5, 2.5.5.2]{rag16}.  
We have derived the equality $(-1)^{m+\delta+1+w_{\boldsymbol{\lambda}}}  = \pm \varphi_\infty(-1)$ from the explicit evaluation of cup products in the proof of Corollary~\ref{cor:signbirch}. 
However, one can also observe the necessity of the same condition formally by studying the action of $\epsilon$ on the cup product paring as follows. First note that the cup product $ I(m, \boldsymbol{f},  h^\pm , \varphi  )$ coincides with  the following summation: 
\begin{align*}
         \sum_{x\in {\rm Cl}^+_{\mathbf Q}({\mathcal K}_2)}  \varphi(x \det {\epsilon})        
      {\rm Tw}_{ [\boldsymbol{n}_m] }  \left( \nabla^{\boldsymbol{n}_m} \iota^* {\rm p}^{\ast}_3 \delta^{(3)}(\boldsymbol{f})_{x \det \epsilon}  
          \cup \left(  {\rm p}^{\ast}_2\delta^{(2)}(h^+)_{x \det \epsilon}  \pm  {\rm p}^{\ast}_2\delta^{(2)}(h^-)_{x \det \epsilon} \right)\right).
\end{align*}
Since $\epsilon$ normalizes $\mathcal{K}_2K_2$, we see that  
$\epsilon$ acts as the right translation on the local systems appearing in Sections~\ref{sec:LS} and \ref{sec:branch*}. We denote the action of $\epsilon$ on the local systems by $\epsilon \, \bullet -$.    
Then we find that 
\begin{align*}
      {\rm Tw}_{ [\boldsymbol{n}_m] } & \left( \nabla^{\boldsymbol{n}_m} \iota^* {\rm p}^{\ast}_3 \delta^{(3)}(\boldsymbol{f})_{x \det \epsilon}  
          \cup \left(  {\rm p}^{\ast}_2\delta^{(2)}(h^+)_{x \det \epsilon}  \pm  {\rm p}^{\ast}_2\delta^{(2)}(h^-)_{x \det \epsilon} \right) \right)  \\ 
 &=  \epsilon \bullet 
        \left( 
                  {\rm Tw}_{ [\boldsymbol{n}_m] } \left( \nabla^{\boldsymbol{n}_m} \iota^* {\rm p}^{\ast}_3 \delta^{(3)}(\boldsymbol{f})_{x }  
                  \cup \left(  {\rm p}^{\ast}_2\delta^{(2)}(h^+)_{x }  \pm  {\rm p}^{\ast}_2\delta^{(2)}(h^-)_{x } \right)   
                \right)    \right)
           \\     
 &=    {\rm Tw}_{ [\boldsymbol{n}_m] } 
        \left(  \det(\epsilon)^{ - [ \boldsymbol{n}_m ]  }
                    \epsilon \bullet \left( 
                                                   \nabla^{\boldsymbol{n}_m} \iota^* {\rm p}^{\ast}_3 \delta^{(3)}(\boldsymbol{f})_{x }  
                                                 \right)    
                  \cup  
                  \epsilon \bullet    \left(  {\rm p}^{\ast}_2\delta^{(2)}(h^+)_{x}  \pm  {\rm p}^{\ast}_2\delta^{(2)}(h^-)_{x } \right)   
                \right).    
\end{align*}
Lemma \ref{lem:nabkl} shows that 
\begin{align*}
      \epsilon \bullet \left(      \nabla^{\boldsymbol{n}_m} \iota^* {\rm p}^{\ast}_3 \delta^{(3)}(\boldsymbol{f})_{x }  
                                                 \right)    
  =    \nabla^{\boldsymbol{n}_m} \iota^* {\rm p}^{\ast}_3 
        \left(  \iota(\epsilon) \bullet  \delta^{(3)}(\boldsymbol{f})_{x }  \right).      
\end{align*}
The action of $\iota(\epsilon)$ on $H^2( \mathfrak{gl}_3({\mathbf R}), K_3 ; H_{\pi^{(3)}, K_3}  \otimes L^{(3)}( \boldsymbol{w}_{\boldsymbol{\lambda}}; {\mathbf C})    )$ 
    coincides with the action of $-1_3 \in {\rm GL}_3({\mathbf R})$, which is given by the multiplication of $\omega_{\pi^{(3)}, \infty } (-1) \omega_{ \boldsymbol{w}_{\boldsymbol{\lambda}} }(-1)$; recall from Section~\ref{sec:repGL3} that $\omega_{\boldsymbol{w_\lambda}}$ denotes the central character of $L^3(\boldsymbol{w}_{\boldsymbol{\lambda}};\mathbf{C})$. Now we may readily check from the definitions that 
\begin{align*}
 \omega_{\pi^{(3)}, \infty } (-1)  
     &=  (-1)^{ \delta +  1 } & \text{and} &&       
 \omega_{ \boldsymbol{w}_{\boldsymbol{\lambda} } }(-1) 
     &=  (-1)^{\frac{l_3}{2} +  1  }.
\end{align*}
Meanwhile, the definition of $\delta^{(2)}(h^+)  \pm  \delta^{(2)}(h^-) $ immediately shows that 
\begin{align*}
\epsilon \bullet    \left(  {\rm p}^{\ast}_2\delta^{(2)}(h^+)_{x}  \pm  {\rm p}^{\ast}_2\delta^{(2)}(h^-)_{x } \right) 
   = (\pm 1)  \left(  {\rm p}^{\ast}_2\delta^{(2)}(h^+)_{x}  \pm  {\rm p}^{\ast}_2\delta^{(2)}(h^-)_{x } \right).    
\end{align*}
Summarizing the above argument, we deduce the identity
\begin{align*}
  &  I(m, \boldsymbol{f},  h^\pm , \varphi  )   
=      \varphi_\infty(-1) 
     (-1)^{ [\boldsymbol{n}_m] }   
     \omega_{\pi^{(3)}, \infty } (-1) 
     \omega_{ \boldsymbol{w}_{\boldsymbol{\lambda} } }(-1) 
    (\pm 1)  
     I(m, \boldsymbol{f},  h^\pm , \varphi  ).  
\end{align*}
Therefore it is necessary to assume the condition 
$ \pm \varphi_\infty(-1) 
    = (-1)^{  [\boldsymbol{n}_m] + \delta + \frac{l_3}{2}} 
    = (-1)^{m + \delta +  w_{ \boldsymbol{\lambda} } +1  }$        
for the non-vanishing of the cup product $I(m,\boldsymbol{f},h^{\pm},\varphi)$.

\end{rem}

\section{Evaluation of period integrals}\label{sec:permain}

In this section, we prove the main theorem (Theorem \ref{thm:main}) in this article.   
In Section \ref{sec:per}, we recall Raghuram and Shahidi's definition of Whittaker periods for $\pi^{(2)}$ and $\pi^{(3)}$ (introduced in \cite{rs08}), and then refine them.    In Section \ref{sec:main}, we propose a proof of Theorem \ref{thm:main}. We also discuss the algebraicity of critical values $L(m, {\mathcal M}(\pi^{(3)} \times \pi^{(2)} ))$ in Section \ref{sec:alg}. 
Namely, we try to give  a motivic  interpretation of the periods defined in Section \ref{sec:per} in terms of Yoshida's period invariants (see Remark \ref{rem:RelDelPer}).

\subsection{Periods}\label{sec:per}

In this subsection, we introduce a specific vector $\boldsymbol{f} \in {\mathcal S}^{(3)}_{\boldsymbol{\lambda}}({\mathcal K}_3)$  
   (resp.\ specific vectors $h^\pm \in {\mathcal S}^{(2)}_{\boldsymbol{\lambda}}({\mathcal K}_2)$) associated with $\pi^{(3)}$ (resp.\ $\pi^{(2)}$).
According to the idea in \cite{rs08}, 
we define the period  $\Omega_{\pi^{(3)}}$ of $\pi^{(3)}$  (resp.\ the periods  $\Omega^\pm_{\pi^{(2)} }$ of $\pi^{(2)}$) depending on the choice of the vector $\boldsymbol{f}$ (resp.\ the vectors $h^\pm$), based on the comparison isomorphism between the Betti and the de Rham cohomology groups of the symmetric space $Y^{(3)}_{{\mathcal K}_3}$ (resp.\  $Y^{(2)}_{{\mathcal K}_2}$).
Note that the periods  $\Omega^\pm_{\pi^{(2)}}$ of $\pi^{(2)}$ constructed in this subsection are known as the {\em canonical periods} in literatures (see \cite{hid94} and \cite{vat99} for instance), 
     and they are determined uniquely up to multiplication of elements in a field of rationality for $\pi^{(2)}$.   
We will define the period $\Omega_{\pi^{(3)}}$ of $\pi^{(3)}$ in an analogous way to the construction of the canonical periods of $\pi^{(2)}$,  
   which is one of reasons why we introduce an explicit form of the Eichler--Shimura map for ${\rm GL}_3$.

We here emphasize the difference of the definition of periods between the present article and  \cite{rs08}.   
In \cite{rs08}, Raghuram and Shahidi define periods for cohomological irreducible cuspidal automorphic representations $\pi^{(n)}$ of ${\rm GL}_n({\mathbf Q}_{\mathbf A})$ 
   by fixing the following data:
\begin{itemize}
 \item[---] signature $\varepsilon$ of the action of ${\mathbf R}^\times {\rm O}_n({\mathbf R})/K_n$ on $H^{b_n}( \mathfrak{gl}_n({\mathbf R}), K_n; H_{\pi^{(n)}, K_n}  \otimes L^{(n)}(*_{\boldsymbol{\lambda}}; {\mathbf C})  )$ where $K_n$ denotes $\mathbf{R}^\times_{>0} \mathrm{SO}_n(\mathbf{R})$. 
          We denote the eigen space of its action by 
              $H^{b_n}( \mathfrak{gl}_n({\mathbf R}), K_n; H_{\pi^{(n)}, K_n}  \otimes L^{(n)}(*_{\boldsymbol{\lambda}}; {\mathbf C})  )[\varepsilon]$;  
 \item[---] a new vector in $\pi^{(n)}_{\rm fin}$;
 \item[---] a vector ${\boldsymbol w}_\infty$ in $H^{b_n}( \mathfrak{gl}_n({\mathbf R}), K_n; H_{\pi^{(n)}, K_n}  \otimes L^{(n)}(*_{\boldsymbol{\lambda}}; {\mathbf C})  )[\varepsilon]$. 
\end{itemize}    
Since Raghuram and Shahidi do not specify explicitly how to choose an archimedean vector ${\boldsymbol w}_\infty$, 
  a certain archimedean local integral depending on the choice of $\boldsymbol{w}_\infty$ and the critical points remains in their formulas of period integrals in \cite[Theorem 2.50]{rag16}; see \cite[(2.49)]{rag16} for the precise description of their archimedean local integral.
It is known that this archimedean local integral does not vanish by Kasten and Schmidt \cite{ks13} in the  ${\rm GL}_3 \times {\rm GL}_2$ case and by Sun \cite{sun17} in the general ${\rm GL}_{n+1} \times {\rm GL}_n$ case.
Moreover, Januszewski has shown that the ratio of archimedean local integrals at each critical points is given 
     by a multiple of a power of $2\pi \sqrt{-1}$ with an implicit rational number
   ; for details see \cite[Corollary 4.12]{jan19}.   
However, it is important to determine this implicit rational number for the application to the study of the Manin congruences among the critical values \cite{jan} and to the construction of the $p$-adic $L$-functions for ${\rm GL}_{n+1}\times {\rm GL}_n$.  
For this purpose, 
   we introduce a distinguished vector ${\boldsymbol w}_\infty$ by fixing a normalization of the Whittaker functions in this subsection.        
The algebraicity result on the critical values proposed in the present article (Corollary \ref{cor:critalg}) refines the one proposed in \cite[Theorem 2.50]{rag16} for the ${\rm GL}_3 \times {\rm GL}_2$ case.
See also Remark \ref{rem:jan19} for the relation with the series of recent works of Januszewski (for instance \cite{jan19} and \cite{jan}).   
  
We specify a distinguished vector in ${\boldsymbol w}_\infty$ in $H^{b_n}( \mathfrak{gl}_n({\mathbf R}), K_n; H_{\pi^{(n)}, K_n}  \otimes L^{(n)}(*_{\boldsymbol{\lambda}}; {\mathbf C})  )[\varepsilon]$   
   by using the explicit description of the Eichler--Shimura maps, which has been already studied in Section \ref{sec:ES},  
   and the explicit formulas of the radial parts of the archimedean Whittaker functions. 
In this subsection, using the explicit and thorough studies of the archimedean Whittaker functions for ${\rm GL}_3$ presented in the recent works by Hirano, Ishii and Miyazaki (see \cite{him16} and \cite{him}),  
  we refine the definition of the periods in \cite[Definition/Proposition 3.3]{rs08}.

Let $\psi: {\mathbf Q} \backslash {\mathbf Q}_{\mathbf A} \to {\mathbf C}^\times$ be the additive character satisfying
$\psi_\infty(x) = {\rm exp}(2\pi \sqrt{-1} x)$  for each $x \in {\mathbf R}$.
Denote by ${\mathcal W}(\pi^{(3)}, \psi)$ the $\psi$-Whittaker model of $\pi^{(3)}$ 
and by ${\mathcal W}(\pi^{(2)}, \psi^{-1})$ the $\psi^{-1}$-Whittaker model of $\pi^{(2)}$.   
Each of ${\mathcal W}(\pi^{(3)}, \psi)$ and ${\mathcal W}(\pi^{(2)}, \psi^{-1})$ decomposes into the restricted tensor product as follows:
\begin{align*}
    {\mathcal W}(\pi^{(3)}, \psi)   &= {\bigotimes_v}^\prime  {\mathcal W}(\pi^{(3)}_v, \psi_v),    &
    {\mathcal W}(\pi^{(2)}, \psi^{-1})   &=   {\bigotimes_v}^\prime  {\mathcal W}(\pi^{(2)}_v, \psi^{-1}_v). 
\end{align*}

Now let us recall the notion of the field of rationality $E_{\pi^{(n)}}$ of $\pi^{(n)}$ according to \cite[Section 3]{clo90}; see also \cite[Section 3.1]{rs08}. Write $V$ as a representation space of $\pi^{(n)}_{\rm fin}$. 
For each $\sigma \in {\rm Aut}({\mathbf C})$,  
  define ${}^\sigma\pi^{(n)}_{\rm fin}$ to be $V\otimes_{{\mathbf C}, \sigma^{-1}} {\mathbf C}$, and set 
$S( \pi^{(n)}_{\rm fin}) = \{  \sigma \in {\rm Aut}({\mathbf C})  \mid  {}^\sigma \pi^{(n)}_{\rm fin} = \pi^{(n)}_{\rm fin}   \}$.     
Then ${}^\sigma\pi^{(n)}_{\rm fin}$ is indeed the finite part of a cohomological cuspidal automorphic representation ${}^\sigma\pi$ of ${\rm GL}_n({\mathbf Q}_{\mathbf A})$, 
   and $E_{\pi^{(n)}} := \{x\in \mathbf{C} \mid x^\sigma=x \text{ for } {}^\forall \sigma \in S( \pi^{(n)}_{\rm fin})\}$ is a finite extension of ${\mathbf Q}$ (see \cite[Th\'{e}or$\grave{\rm e}$m 3.13]{clo90}).
We call $E_{\pi^{(n)}}$ the {\em field of rationality} of $\pi^{(n)}$. 

In the rest of this article, we impose several constraints on $\mathcal{K}_2$ and $\mathcal{K}_3$. First we suppose that $\pi^{(2)}_{\mathrm{fin}}$ has a nonzero vector fixed by the action of $\mathrm{GL}_2(\widehat{\mathbf{Z}})$, and set $\mathcal{K}_2=\mathrm{GL}_2(\widehat{\mathbf{Z}})$. Next, for an ideal ${\mathfrak N}$ of $\widehat{\mathbf Z}$, we introduce an open compact subgroup ${\mathcal K}_3({\mathfrak N})$ of $\mathrm{GL}_3(\widehat{\mathbf{Z}})$ defined as
\begin{align*}
   {\mathcal K}_3({\mathfrak N}) 
    = \left\{  g = \begin{pmatrix} g_{ij}   \end{pmatrix}_{1\leq i,j\leq 3}  \in {\rm GL}_3(\widehat{\mathbf Z}) \ \middle|  \   g_{31}, g_{32} \in {\mathfrak N},\  g_{33}\equiv 1\pmod{\mathfrak{N}}    \right\}.
\end{align*}
Suppose that there exists an ideal $\mathfrak{N}$ of $\widehat{\mathbf{Z}}$ such that $\pi^{(3)}_{\mathrm{fin}}$ has a nonzero vector fixed by the action of $\mathcal{K}_3(\mathfrak{N})$. We choose a minimal ideal
${\mathfrak N}$ satisfying this property, and set $\mathcal{K}_3=\mathcal{K}_3(\mathfrak{N})$. Refer to Remark~\ref{rem:tameint} for these assumptions.
 
\medskip
First let us discuss the period of $\pi^{(2)}$. For each finite place $v$ of $\mathbf{Q}$, the $v$-component $\pi^{(2)}_v$ of $\pi^{(2)}$ is unramified by assumption, and thus there exists a unique function $W^\circ_{\pi^{(2)},v}\in \mathcal{W}(\pi^{(2)}_v,\psi_v^{-1})^{\mathrm{GL}(\mathbf{Z}_v)}$ which is called by Shintani the {\em class-$1$ Whittaker function} on $\mathrm{GL}_2(\mathbf{Q}_v)$; see \cite{shin76} for details. We now define archimedean Whittaker functions $W^\pm_{\pi^{(2)}, \infty} \in \mathcal{W}(\pi^{(2)}_{\infty},\psi^{-1}_{\infty})$ according to \cite[Corollary 3.6, Section 9.3]{him}. Recall from Section~\ref{sec:motives} that $\pi^{(2)}_\infty$ has the minimal $\mathrm{O}_2(\mathbf{R})$-type $\tau^{(2)}_{\boldsymbol{\lambda}}$ with $\boldsymbol{\lambda}^{(2)}=(\lambda_2,0)=(l_2+1,0)$.
Let $\varphi^{(2)}_{-} \colon V^{(2)}_{\boldsymbol{\lambda}} \to {\mathcal W}(\pi^{(2)}_\infty, \psi^{-1}_\infty)$ be a $K_2$-homomorphism  satisfying the following formulas on the radial parts (see \cite[Corollary 3.6, Section 9.3]{him}): 
\begin{align*}
      \varphi^{(2)}_{-}  (  (- X+\sqrt{-1} Y)^{\lambda_2}  )  (  {\rm diag}(y_1y_2, y_2)  ) 
    &=   \displaystyle  \frac{ y^{\frac{1}{2}}_1  y^{2\nu_2}_2 }{2\pi\sqrt{-1}} \int_t \Gamma_{\mathbf C} \left( t + \nu_2 + \frac{ \lambda_2 -1}{2} \right)  y^{-t}_1  {\rm d} t,  \\
 \varphi^{(2)}_{-}  (  ( X+\sqrt{-1} Y)^{\lambda_2}  )  (  {\rm diag}(y_1y_2, y_2)  )  & =0.
\end{align*}
Here $\int_t$ denotes an integration on a vertical line from ${\rm Re}(t) - \sqrt{-1} \infty$ to ${\rm Re}(t) + \sqrt{-1} \infty$  
   for an appropriate complex number  $t \in {\mathbf C}$ with the sufficiently large real part  so that all the poles  of the integrand is in the left-side of the vertical line.
Then  define $W^\pm_{\pi^{(2)}, \infty}$  to be  $\varphi^{(2)}_{-}  (  (\pm X+\sqrt{-1} Y)^{\lambda_2}  )$, and set $W^\pm_{\pi^{(2)}}=W^\pm_{\pi^{(2)},\infty}\otimes \bigotimes_{v\colon \text{finite}} W^{\circ}_{\pi^{(2)},v}$. We define $h^{\pm}$ as the cuspidal automorphic forms on $\mathrm{GL}_2(\mathbf{Q}_{\mathbf{A}})$ corresponding to $W^\pm_{\pi^{(2)}}$, that is, $h^\pm$ are defined to be the inverse Fourier transform of $W^\pm_{\pi^{(2)}}$. Then define the cohomology classes $\delta^\pm_{\pi^{(2)}}$ in $H^1(Y^{(2)}_{\mathcal K_2},\mathcal{L}^{(2)}(\boldsymbol{n}_{\boldsymbol{\lambda}};\mathbf{C}))$ as $\delta^\pm_{\pi^{(2)}}   =  \delta^{(2)}(h^+) \pm \delta^{(2)}(h^-)$.   
Since $\epsilon= \bigotimes_{v\colon \text{finite}} 1_2 \otimes \begin{pmatrix} -1 & 0 \\ 0 & 1 \end{pmatrix}$ normalizes $\mathcal{K}_2K_2$ where $K_2=\mathbf{R}_{>0}^\times \mathrm{SO}_2(\mathbf{R})$, it acts on $H^\bullet_\ast(Y^{(2)}_{\mathcal K_2}, {\mathcal L}(\boldsymbol{n}_{\boldsymbol{\lambda}}; {\mathbf C}))$; hereafter $*$ denotes $\emptyset$ or ${\rm c}$. We denote by $H^\bullet_\ast(Y^{(2)}_{\mathcal K_2}, {\mathcal L}^{(2)}(\boldsymbol{n}_{\boldsymbol{\lambda}}; {\mathbf C}))[\pm]$ 
its $(\pm 1)$-eigenspaces with respect to the action of $\epsilon$, and 
write the $\pi^{(2)}_{\rm fin}$-isotypic components of $H^\bullet_\ast(Y^{(2)}_{\mathcal K_2}, {\mathcal L}^{(2)}(\boldsymbol{n}_{\boldsymbol{\lambda}}; {\mathbf C}))[\pm]$
as $H^\bullet_\ast(Y^{(2)}_{\mathcal K_2}, {\mathcal L}^{(2)}(\boldsymbol{n}_{\boldsymbol{\lambda}}; {\mathbf C}))[\pm, \pi_{\rm fin}]$.   
Then  each of $H^1_{\rm c}(Y^{(2)}_{\mathcal K_2}, {\mathcal L}^{(2)}(\boldsymbol{n}_{\boldsymbol{\lambda}}; {\mathbf C}))[\pm, \pi_{\rm fin}]$ is of dimension $1$  
   and it is spanned by $\delta^\pm_{\pi^{(2)}}$.
Furthermore the spaces $H^1_{\rm c}(Y^{(2)}_{\mathcal K_2}, {\mathcal L}^{(2)}(\boldsymbol{n}_{\boldsymbol{\lambda}}; {\mathbf C}))[\pm, \pi_{\rm fin}]$
are defined over the field of rationality $E_{\pi^{(2)}}$ of $\pi^{(2)}$, and their $E_{\pi^{(2)}}$-structures
    are given as 
\begin{align*}
  H^1_{\rm c}(Y^{(2)}_{\mathcal K_2}, {\mathcal L}^{(2)}(\boldsymbol{n}_{\boldsymbol{\lambda}}; E_{\pi^{(2)}})) [\pm, \pi^{(2)}_{\rm fin}]  
   =  H^1_{\rm c}(Y^{(2)}_{\mathcal K_2}, {\mathcal L}^{(2)}(\boldsymbol{n}_{\boldsymbol{\lambda}};  {\mathbf C}  )) [\pm, \pi^{(2)}_{\rm fin}] 
        \cap H^1_{\rm c}(Y^{(2)}_{\mathcal K_2}, {\mathcal L}^{(2)}(\boldsymbol{n}_{\boldsymbol{\lambda}}; E_{\pi^{(2)}})).  
\end{align*}
For each signature $\pm$, fix a generator $\eta^\pm_{\pi^{(2)}}$ of the $1$-dimensional $E_{\pi^{(2)}}$-space $H^1_{\rm c}(Y^{(2)}_{\mathcal K_2}, {\mathcal L}^{(2)}(\boldsymbol{n}_{\boldsymbol{\lambda}}; E_{\pi^{(2)}})) [\pm, \pi^{(2)}_{\rm fin}] $.   
 Then we define  complex numbers $\Omega^\pm_{\pi^{(2)}}\in {\mathbf C}^\times$ so that $   \delta^{\pm}_{\pi^{(2)}}  =   \Omega^\pm_{\pi^{(2)}} \eta^\pm_{\pi^{(2)}} \; (\text{double sign in the same order})$.  
Note that each of $\Omega^\pm_{\pi^{(2)}}$ is uniquely determined from $\pi^{(2)}$ up to multiplication of elements in $E^\times_{\pi^{(2)}}$. We call $\Omega_{\pi^{(2)}}$ the periods of $\pi^{(2)}$.

\medskip

Next let us discuss the period of $\pi^{(3)}$. 
For each finite place $v$ of ${\mathbf Q}$,   
 there exists a unique function $W^{\mathrm{ess}}_{\pi^{(3)},v}$ in ${\mathcal W}(\pi^{(3)}_v, \psi_v)^{{\mathcal K}_3({\mathfrak N})_v }$ satisfying
 \begin{align*}
  \int_{ {\rm N}_2({\mathbf Q}_v) \backslash {\rm GL}_2({\mathbf Q}_v)}   
      W^{\mathrm{ess}}_{\pi^{(3)},v}\left(  \iota(g)  \right) 
    W^{\circ}_{\pi^{(2)}, v} (g)  |\det g|_v^{s-\frac{1}{2} } 
          {\rm d}g   =   L(s, \pi^{(3)}_v \times \pi^{(2)}_v),
 \end{align*}
where $\mathrm{N}_2$ denotes the uniponent radical of the upper triangular Borel subgroup of $\mathrm{GL}_2$. The Whittaker function $W^{\mathrm{ess}}_{\pi^{(3)},v}$ is called  the {\em essential vector} of $\pi^{(3)}_v$, whose existence and uniqueness is verified in \cite[page 208, 211, Th\'{e}or$\grave{\rm e}$m]{jpss81}.
We also define archimedean Whittaker functions $W^\alpha_{\pi^{(3)}, \infty} \in {\mathcal W}(\pi^{(3)}_\infty, \psi_\infty)$ for $-\lambda_3 \leq \alpha \leq \lambda_3$ according to \cite[Sections 4.6 and 9.4]{him}. 
Recall that $\pi^{(3)}_\infty$ has the minimal $\mathrm{O}_3(\mathbf{R})$-type $\tau^{(3)}_{\mathbf{\lambda}}$ for $\boldsymbol{\lambda}^{(3)}=(\lambda_3,\delta)=(l_3+1,\delta)$. Let $\varphi^{(3)}_+ \colon V^{(3)}_{\boldsymbol{\lambda}} \to {\mathcal W}( \pi^{(3)}_\infty, \psi_\infty)$ be a $K_3$-homomorphism satisfying the following formula on the radial parts for each $z_{\boldsymbol{a}}:= z^{a_1}_1 z^{a_2}_2 z^{a_3}_3 \in V^{ (3) }_{\lambda}$ (see \cite[Theorem 4.21 and Section 9.4]{him}):
  \begin{align*}
                 \varphi^{(3)}_+&(z_{\boldsymbol{a}})( {\rm diag}(y_1y_2 y_3, y_2 y_3, y_3)  )    \\ 
            &  = (-1)^{a_1} \sqrt{-1}^{a_2}   
                       \frac{y_1 y_2 (y_2 y_3)^{3\nu_3} }{  (4\pi\sqrt{-1})^2 }    \\  
                &\qquad  \times  \int_{t_2} \int_{t_1}   
                    \frac{      \Gamma_{\mathbf C}(t_1 + \nu_3 + \frac{\lambda_3-1}{2})    
                                   \Gamma_{\mathbf R}(t_2 + \nu_3 + a_1 ) 
                               }{ \Gamma_{\mathbf R} ( t_1 + t_2 + a_1 + a_3 )  }
                    \Gamma_{\mathbf C}\left(t_2 - \nu_3 + \frac{\lambda_3-1}{2}\right) 
                    \Gamma_{\mathbf R}(t_2 - \nu_3 +  a_3 ) 
                   y^{-t_1}_1 y^{-t_2}_2\, {\rm d}t_1 {\rm d}t_2. 
\end{align*}  
The path integrals $ \int_{t_2}$ and  $\int_{t_1}$ are defined in a similar way to the case of ${\rm GL}_2$. 
Then define $W^{\alpha}_{\pi^{(3)}, \infty}$ to be $\varphi^{(3)}_+(  v^{\boldsymbol{\lambda}}_\alpha )$ for each $-\lambda_3 \leq \alpha \leq \lambda_3$, 
 and set $W^{\alpha}_{\pi^{(3)}} = W^{\alpha}_{\pi^{(3)}, \infty} \otimes \bigotimes_{v\text{:finite}} W^{\mathrm{ess}}_{\pi^{(3)}, v}$.  
We define $f^{\boldsymbol{\lambda}}_\alpha$ as the cuspidal automorphic form on ${\rm GL}_3({\mathbf Q}_{\mathbf A})$ corresponding to $W^{\alpha}_{\pi^{(3)}}$, 
that is, $f^{\boldsymbol{\lambda}}_\alpha$ is defined to be the inverse Fourier transform of $W^{\alpha}_{\pi^{(3)}}$.
Then we set  $\boldsymbol{f} = 
\begin{pmatrix}
 f^{\boldsymbol{\lambda}}_{\lambda_3} &  f^{\boldsymbol{\lambda}}_{\lambda_3-1} & \dotsc & f^{\boldsymbol{\lambda}}_{-\lambda_3}
\end{pmatrix}$ and consider its image 
$\delta^{(3)}(\boldsymbol{f})$ under the Eichler--Shimura map. Note that it is indeed an element of the $\pi^{(3)}_{\rm fin}$-isotypic component  $H^2_{\rm cusp}(Y^{(3)}_{\mathcal K_3},  {\mathcal L}^{(3)}(\boldsymbol{w}_{\boldsymbol{\lambda}}; {\mathbf C})  )[\pi^{(3)}_{\rm fin} ]$ 
       of  $H^2_{\rm cusp}(Y^{(3)}_{\mathcal K_3},  {\mathcal L}^{(3)}(\boldsymbol{w}_{\boldsymbol{\lambda}}; {\mathbf C})  )$, which is of dimension 1 over ${\mathbf C}$. 
Similarly to the $n=2$ case, the space $H^2_{\rm cusp}(Y^{(3)}_{\mathcal K_3},  {\mathcal L}^{(3)}(\boldsymbol{w}_{\boldsymbol{\lambda}}; {\mathbf C})  )[\pi^{(3)}_{\rm fin} ]$ is defined over the field of rationality $E_{\pi^{(3)}}$ of $\pi^{(3)}$, and its $E_{\pi^{(3)}}$-structure is given as 
\begin{align*}
  H^2_{\rm cusp}(Y^{(3)}_{\mathcal K_3}, {\mathcal L}^{(3)}(\boldsymbol{w}_{\boldsymbol{\lambda}}; E_{\pi^{(3)}})) [\pi^{(3)}_{\rm fin}]  
   =  H^2_{\rm cusp}(Y^{(3)}_{\mathcal K_3}, {\mathcal L}^{(3)}(\boldsymbol{w}_{\boldsymbol{\lambda}};  {\mathbf C}  )) [\pi^{(3)}_{\rm fin}] 
        \cap H^2(Y^{(3)}_{\mathcal K_3}, {\mathcal L}^{(3)}(\boldsymbol{w}_{\boldsymbol{\lambda}}; E_{\pi^{(3)}})).  
\end{align*}
 Hence we can choose a generator $\eta_{\pi^{(3)}}$ of the $1$-dimensional $E_{\pi^{(3)}}$-space $H^2_{\rm cusp}(Y^{(3)}_{\mathcal K_3}, {\mathcal L}^{(3)}(\boldsymbol{w}_{\boldsymbol{\lambda}}; E_{\pi^{(3)}})) [\pi^{(3)}_{\rm fin}] $, and the universal coefficient theorem yields that there exists a complex number $\Omega_{\pi^{(3)}} \in {\mathbf C}^\times$ satisfying 
$\delta^{(3)}(\boldsymbol{f}) = \Omega_{\pi^{(3)}} \eta_{\pi^{(3)}}$.   
From its construction, $\Omega_{\pi^{(3)}}$ is determined up to multiplication of elements of $E^\times_{\pi^{(3)}}$.   
We call  $\Omega_{\pi^{(3)}}$ the period for $\pi^{(3)}$.

\begin{rem}\label{rem:tameint}
We impose that $\mathcal{K}_2$ coincides with $\mathrm{GL}_2(\widehat{\mathbf{Z}})$ in the construction of the periods, but this condition seems too restrictive. 
In fact, 
   Iwahori fixed vectors play important roles 
   in the construction of the $p$-adic $L$-functions for ${\rm GL}_{n+1} \times {\rm GL}_n$ (for example, see \cite{jan}).
However, 
 carefully observing use of essential vectors in \cite{jpss81}, 
  we do not think that it is so easy to be obvious to give an explicit relation between local zeta integrals and the $L$-factors at ramified places when $\mathcal{K}_2$ does not coincide with $\mathrm{GL}_2(\widehat{\mathbf{Z}})$.   
Since the main theme of this article is to give a precise cohomological interpretation of archimedean zeta integrals, 
we postpone the study of local zeta integrals at ramified places to the future work. 
\end{rem}

\subsection{Main result}\label{sec:main}

Let $\nabla^{\boldsymbol{n}_m} \iota^* {\rm p}^{\ast}_3 \eta_{\pi^{(3)}, x} $
      (resp.\ ${\rm p}^{\ast}_2 \eta^\pm_{\pi^{(2)}, x}$)
      be the restriction of $\nabla^{\boldsymbol{n}_m} \iota^* {\rm p}^{\ast}_3 \eta_{\pi^{(3)}} $
           (resp.\ ${\rm p}^{\ast}_2 \eta^\pm_{\pi^{(2)} }$)
       to ${\mathcal Y}^{(2)}_{ {\mathcal K}_2, x }$
       for $x \in {\rm Cl}^+_{\mathbf Q}({\mathcal K}_2)$.
The aim of this paper is to evaluate the following summation of the cohomological cup products
\begin{align*}
   I(m, \pi^{(3)}, \pi^{(2)} , \varphi  )   
    :=    \sum_{x\in {\rm Cl}^+_{\mathbf Q}({\mathcal K}_2)}  \varphi(x)        
           {\rm Tw}_{ [\boldsymbol{n}_m] } \left(\nabla^{\boldsymbol{n}_m} \iota^* {\rm p}^{\ast}_3 \eta_{\pi^{(3)}, x}  \cup {\rm p}^{\ast}_2 \eta^\pm_{\pi^{(2)}, x}\right).
\end{align*}
As we  have already explained in Remark \ref{rem:tameint}, we suppose that 
   ${\mathcal K}_2 = {\rm GL}_2( \widehat{\mathbf Z} )$.  
Hence if the conductor of $\varphi$ is not trivial, 
$ I(m, \pi^{(3)}, \pi^{(2)} , \varphi  )$ vanishes.     
Therefore we may assume that $\varphi$ is the trivial character ${\mathbf 1}$.  
We write  $I(m, \pi^{(3)}, \pi^{(2)}   )$ (resp.\ ${\mathcal Z}(s; \boldsymbol{f}, h^\pm)$)
   for the cup product $I(m, \pi^{(3)}, \pi^{(2)} , {\mathbf 1}  )$  (resp.\ for the zeta integral ${\mathcal Z}(s; \boldsymbol{f}, h^\pm, \mathbf{1})$)
     for the sake of simplicity.

Set $(-1)^{m + \delta+  w_{ \boldsymbol{\lambda} } + 1 }=\pm 1$. Then
the constructions of $\eta_{\pi^{(3)}}$ and $\eta^\pm_{\pi^{(2)}}$ imply the following identity (see Proposition~\ref{prop:intsign} and Corollary~\ref{cor:signbirch}):
\begin{align}\label{eq:etacup}
 \begin{aligned} 
     I(m, \pi^{(3)}, \pi^{(2)}   )
   =  \frac{1}{ \Omega_{\pi^{(3)}}  \Omega^\pm_{\pi^{(2)}}  } 
        \cdot 
        \sqrt{-1}^{\frac{l_3}{2}+2l_2-m-1}  C(\boldsymbol{w}_{\boldsymbol{\lambda}}, \boldsymbol{n}_m) 
            {\mathcal Z}( s ; \boldsymbol{f}, h^+)  |_{s=\frac{1}{2}+[\boldsymbol{n}_m]}.  
 \end{aligned}  
\end{align}
The main theorem in this article is the evaluation of the integral in the right-hand side of (\ref{eq:etacup}).    
The result is as follows:

\begin{thm}\label{thm:main}
{\itshape
Suppose that $ (-1)^{m+ \delta + \frac{l_3}{2} }  
                          = \pm 1$.
Then we have 
\begin{align*}
     I(m, \pi^{(3)}, \pi^{(2)}   ) 
     =  (-1)^{ \delta  } 
        \sqrt{-1}^{\frac{l_3}{2}-m+1}    
         \binom{ \frac{l_3}{2} - 1  }{ m-\frac{l_3}{2}-1 }  
                 \binom{  \frac{l_3}{2} - 1  }{  \frac{l_3}{2}+l_2-m }
           \frac{  L(m, {\mathcal M}(\pi^{(3)} \times \pi^{(2)})   ) }{ \Omega_{\pi^{(3)}} \Omega^\pm_{\pi^{(2)}} }
\end{align*}
If $(-1)^{m+ \delta + \frac{l_3}{2} } 
      \neq \pm 1$, 
       then 
       $I(m, \pi^{(3)}, \pi^{(2)}   )=0$.
}
\end{thm}

Note that a standard argument yields that the zeta integral appearing in (\ref{eq:etacup}) is unfolded to
\begin{align*}
 {\mathcal Z}( s ; \boldsymbol{f}, h^+) 
  =   \int_{{\rm N}_2({\mathbf Q}_{\mathbf A})  \backslash {\rm GL}_2( {\mathbf Q}_{\mathbf A} ) }   
        W^{ -\lambda_2 }_{ \pi^{(3)} }  (\iota(g))  
         W^+_{ \pi^{(2)} }    (g)       
         |\det g|^{s-\frac{1}{2}}_{\mathbf A}
         {\rm d}g,
\end{align*}
and hence the proof of Theorem~\ref{thm:main} is reduced to evaluation of the local zeta integral above. 

\begin{proof}
We firstly recall the notation and settings; we have $\pi^{(3)}_\infty\cong\mathrm{Ind}^{\mathrm{GL}_3(\mathbf{R})}_{\mathrm{P}_{2,1}(\mathbf{R})}(D_{\nu_3,l_3}\boxtimes \chi_{\nu_3}^{\delta})$ and $\pi^{(2)}_\infty\cong D_{\nu_2,l_2}$ where
 \begin{align*}
   \nu_2 &=  - \frac{l_2}{2} + \frac{1}{2}, &
   \nu_3 &=                - \frac{l_3}{2} + 1 \\  
  \boldsymbol{\lambda}^{(3)}&=(\lambda_3,\delta), \qquad \lambda_3=l_3+1, &  \boldsymbol{\lambda}^{(2)}&=(\lambda_2,0), \qquad \lambda_2=l_2+1, \\
   \boldsymbol{n}_{\boldsymbol{\lambda}} &=(n_{\boldsymbol{\lambda}},0)=(l_2-1), &  \boldsymbol{w}_{\boldsymbol{\lambda}}&=(w_{\boldsymbol{\lambda}},w_{\boldsymbol{\lambda}},w_{\boldsymbol{\lambda}})=\left(\dfrac{l_3}{2}-1, \dfrac{l_3}{2}-1, \dfrac{l_3}{2}-1\right),  \\   
   \boldsymbol{n}_m &= ( n_{m,1}, n_{m,2})=(l_2-1,m-l_2-1), & [\boldsymbol{n}_m] &=m-2.
 \end{align*}  
See (\ref{eq:NormNu}), (\ref{eq:Lambda}), (\ref{eq:w2}) and (\ref{eq:nm}) for details. 
  In particular, we see that $(-1)^{ m+ \delta + w_{\boldsymbol{\lambda}} + 1}  = (-1)^{m+ \delta + \frac{l_3}{2} }$ holds and obtain the identity of the $L$-functions $L(s,\mathcal{M}(\pi^{(3)}\times \pi^{(3)}))=L(s-\frac{3}{2},\pi^{(3)}\times \pi^{(2)})$; see (\ref{eq:Gfactor}).
Since $\frac{1}{2} + [\boldsymbol{n}_m] = m -\frac{3}{2}$, the identity yields that 
\begin{align*}
       L(s, \pi^{(3)} \times \pi^{(2)}  ) |_{s=\frac{1}{2} + [\boldsymbol{n}_m] } 
   =  L(m, {\mathcal M}(\pi^{(3)} \times \pi^{(2)}  )) 
\end{align*}
holds. For  $s\in {\mathbf C}$ with ${\rm Re}(s) \gg 0$, 
the unfolding of the zeta integral yields that 
\begin{align*} 
{\mathcal Z}(s ; \boldsymbol{f}, h^+) 
  &=   \int_{ {\rm N}_2({\mathbf R} ) \backslash {\rm GL}_2({\mathbf R})}   
      W^{ - \lambda_2}_{\pi^{(3)}, \infty}\left(  \iota(g)  \right) 
    W^+_{\pi^{(2)}, \infty} (g)  |\det g|^{s-\frac{1}{2} }_\infty 
          {\rm d}g_\infty  \\
    & \qquad \quad  \times 
          \prod_{v\text{:finite}}
             \int_{ {\rm N}_2({\mathbf Q}_v) \backslash {\rm GL}_2({\mathbf Q}_v)}   
      W_{\pi^{(3)}, v }\left(  \iota(g)  \right) 
    W_{\pi^{(2)}, v} (g)  |\det g|^{s-\frac{1}{2} }_v 
          {\rm d}g_v.  
\end{align*} 

By the property of the essential vector (see \cite[page 208, Th\'{e}or$\grave{\rm e}$m (ii)]{jpss81} for example), we find that 
\begin{align*}
    \int_{ {\rm N}_2({\mathbf Q}_v) \backslash {\rm GL}_2({\mathbf Q}_v)}   
      W_{\pi^{(3)}, v }\left(  \iota(g)  \right) 
    W_{\pi^{(2)}, v} (g)  |\det g|^{ [\boldsymbol{n}_m] } 
          {\rm d}g
  = L(s, \pi^{(3)}_v \times \pi^{(2)}_v ) |_{s=\frac{1}{2}+[\boldsymbol{n}_m]}   
  = L_v(m, {\mathcal M}(\pi^{(3)} \times \pi^{(2)})   ).
\end{align*}

We compute the archimedean local integral. 
Recall the $K_2$-homomorphism $\varphi^{(2)}_- \colon V^{(2)}_{\boldsymbol{\lambda}^{(2)}} \to {\mathcal W}(\pi^{(2)}_\infty, \psi^{-1} )$ 
(resp.\ the  $K_3$-homomorphism  $\varphi^{(3)}_+\colon V^{(3)}_{\boldsymbol{\lambda}^{(3)}} \to {\mathcal W}(\pi^{(3)}_\infty, \psi )$) introduced in Section \ref{sec:per}. Also recall the $K_2$-equivariant homomorphisms $\iota^{\boldsymbol{\lambda}^{(3)}}_{\boldsymbol{\lambda}^{(2)}}\colon V_{\boldsymbol{\lambda}^{(2)}}^{(2)}\hookrightarrow V^{(3)}_{\boldsymbol{\lambda}^{(3)}}$ and  $\mathrm{P}^{\boldsymbol{\lambda}^{(3)}}_{\boldsymbol{\lambda}^{(2)}}\colon V^{(3)}_{\boldsymbol{\lambda}^{(3)}} \twoheadrightarrow V_{\boldsymbol{\lambda}^{(2)}}^{(2)}$ introduced in Section~\ref{sec:branch_orth}, which we abbreviate as $\iota_{\boldsymbol{\lambda}}$ and $\mathrm{P}_{\boldsymbol{\lambda}}$ to lighten the notation (see also \cite[Section 9.5]{him}).
Define a ${\mathbf C}$-bilinear pairing $\langle \cdot, \cdot \rangle: V^{(2)}_{\boldsymbol{\lambda}^{(2)}} \otimes_{\mathbf C} V^{(2)}_{\boldsymbol{\lambda}^{(2)}} \to {\mathbf C}$ by 
\begin{align*}
   \langle  v^{\boldsymbol{\lambda}^{(2)}}_{q}, v^{\boldsymbol{\lambda}^{(2)}}_{q^\prime} \rangle = \delta_{0, q+q^\prime}  
\end{align*}
for each $q, q^\prime \in \{ \pm \lambda_2 \}$; see \cite[(9.2)]{him} for details.
Here $\delta_{\ast, \ast}$ denotes Kronecker's delta symbol. 
Then due to \cite[Corollary 9.7]{him}, we have 
\begin{align*}
       \int_{ {\rm N}_2({\mathbf R} ) \backslash {\rm GL}_2({\mathbf R})}  & 
           W^{ -\lambda_2 }_{\pi^{(3)}, \infty  } \left(  \iota(g)  \right)  
           W^+_{\pi^{(2)},\infty}\left(  g  \right)     
           \lvert \det g \rvert^{s-\frac{1}{2}}_\infty  {\rm d}g_\infty     \\
&=      \int_{ {\rm N}_2({\mathbf R} ) \backslash {\rm GL}_2({\mathbf R})}   
          \varphi^{(3)}_{+}( v^{\boldsymbol{\lambda}^{(3)}}_{-\lambda_3}    )(\iota(g))
          \varphi^{(2)}_-(    v^{\boldsymbol{\lambda}^{(2)}}_{\lambda_2}    )(g)
          | \det g |^{s-\frac{1}{2}}_\infty
          {\rm d}g_\infty     \\
   &=   (-1)^{\lambda_2}             
          \langle \mathrm{P}_{\boldsymbol{\lambda}}\left(   v^{\boldsymbol{\lambda}^{(3)}}_{-\lambda_2}        \right),  
                                                          v^{\boldsymbol{\lambda}^{(2)}}_{\lambda_2}   \rangle 
          \cdot  L(s, \pi^{(3)}_\infty \times \pi^{(2)}_\infty).   
\end{align*}
Since $ \iota_{\boldsymbol{\lambda}}  ( (-1)^{\delta} v^{\boldsymbol{\lambda}^{(2)}}_{-\lambda_2} ) = v^{\boldsymbol{\lambda}^{(3)}}_{-\lambda_2}$ holds, 
we have the identity
\begin{align*}
        \langle \mathrm{P}_{\boldsymbol{\lambda}} \left(   v^{\boldsymbol{\lambda}^{(3)}}_{-\lambda_2}  \right),  
                                                          v^{\boldsymbol{\lambda}^{(2)}}_{\lambda_2}    \rangle     
 =     \langle (-1)^{\delta} v^{\boldsymbol{\lambda}^{(2)}}_{-\lambda_2},      v^{\boldsymbol{\lambda}^{(2)}}_{\lambda_2}    \rangle 
 =  (-1)^{\delta}.
\end{align*}
Hence, using this identity, we find that 
\begin{align*}
 \sqrt{-1}^{ n_{m,1} - n_{m,2} + w_{\boldsymbol{\lambda}}} C(\boldsymbol{w}_{\boldsymbol{\lambda}}, \boldsymbol{n}_m)    
  & \int_{ {\rm N}_2({\mathbf R} ) \backslash {\rm GL}_2({\mathbf R})}   
           W^{ -n_1 - 2 }_{\pi^{(3)}, \infty  } \left(  \iota(g)  \right)  
           W^+_{\pi^{(2)},\infty}\left(  g  \right)     
           | \det g |^{s-\frac{1}{2}}_\infty  {\rm d}g_\infty   \\
& =  \sqrt{-1}^{ n_{m,1} - n_{m,2} + w_{\boldsymbol{\lambda}}} C(\boldsymbol{w}_{\boldsymbol{\lambda}}, \boldsymbol{n}_m )   
        \cdot (-1)^{\lambda_2+\delta}  
         \cdot L(s, \pi^{(3)}_\infty \times \pi^{(2)}_\infty )  \\ 
& =    (-1)^{ \delta + l_2+1 } 
        \sqrt{-1}^{ \frac{l_3}{2}+2l_2-m-1}      
         \binom{ \frac{l_3}{2} - 1  }{ m-\frac{l_3}{2}-1 }  
                  \binom{  \frac{l_3}{2} - 1  }{  \frac{l_3}{2}+l_2-m }
         L(s, \pi^{(3)}_\infty \times \pi^{(2)}_\infty ).
\end{align*}
This proves the statement (note that we have $(-1)^{l_2+1}\sqrt{-1}^{\frac{l_3}{2}+2l_2-m-1}=\sqrt{-1}^{\frac{l_3}{2}-m+1}$).
\end{proof}

\begin{rem}\label{rem:relCP}
According to \cite[page 103 (4)]{coa89}, 
 the modified Euler factor ${\mathcal L}^{\sqrt{-1}}_\infty( {\mathcal M}(\pi^{(3)} \times \pi^{(2)} ) (m)  )$ at $\infty$ of the  $p$-adic $L$-functions for ${\mathcal M}(\pi^{(3)} \times \pi^{(2)} )$ is given by   
 \begin{align*}
 {\mathcal L}^{ \sqrt{-1}}_\infty( {\mathcal M}(\pi^{(3)} \times \pi^{(2)} )(m)  )
 &= ( \sqrt{-1})^{\frac{l_3}{2}+l_2-3m} 
     L_\infty(m, {\mathcal M} (\pi^{(3)} \times \pi^{(2)} ) ) \\
 &= (-1)^{ \delta + m } \sqrt{-1}^{l_2 -1} 
           \cdot (-1)^{\delta }\sqrt{-1}^{\frac{l_3}{2} -m +1} 
     L_\infty(m, {\mathcal M} (\pi^{(3)} \times \pi^{(2)} ) ).          
 \end{align*}
 Hence, if we normalize $\nabla^{\boldsymbol{n}_m}  = \nabla_{2w_{\boldsymbol{\lambda}}-n_{m,1}-n_{m,2}, n_{m,2}}$ 
                                        in (\ref{eq:nabkl}) so that  
 \begin{align*}
     \widetilde\nabla^{\boldsymbol{n}_m} 
     =  (-1)^{ \delta + m } \sqrt{-1}^{l_2 -1}
      C(\boldsymbol{w}_{\boldsymbol{\lambda}}, \boldsymbol{n}_m)^{-1}    
       \nabla^{\boldsymbol{n}_m},  
 \end{align*}
 then  $\widetilde\nabla^{\boldsymbol{n}_m}$ defines a map 
 $ \iota^* {\mathcal L}^{(3)} ( \boldsymbol{w}_{\boldsymbol{\lambda}} ; A)_{/ {\mathcal Y}^{(2)}_{{\mathcal K}_2} }  \longrightarrow   {\mathcal L}^{(2)} (\boldsymbol{n}_m ; A)_{/ {\mathcal Y}^{(2)}_{{\mathcal K}_2} }$
 between local systems in the same way as (\ref{eq:nablocsys}).  
 Then Theorem \ref{thm:main} implies that 
    the archimedean local integral appearing in the evaluation of ${\rm Tw}_{ [\boldsymbol{n}_m] } \left(\widetilde\nabla^{\boldsymbol{n}_m} \iota^* {\rm p}^{\ast}_3 \eta_{\pi^{(3)}}  \cup {\rm p}^{\ast}_2 \eta^\pm_{\pi^{(2)}}\right)$
     coincides with the modified local $L$-factor ${\mathcal L}^{\sqrt{-1}}_\infty( {\mathcal M}(\pi^{(3)} \times \pi^{(2)} )(m)  )$ at $\infty$ for each $m\in \mathrm{Crit}(\boldsymbol{\lambda})$. 
\end{rem}

\begin{rem}\label{rem:jan19}
We compare Theorem \ref{thm:main} with \cite[Theorem 2.48, Theorem 2.50]{rag16} and \cite[Corollary 4.12]{jan19}. 
\begin{enumerate}
\item We briefly recall a description of the critical values for ${\rm GL}_3 \times {\rm GL}_2$ presented by Raghuram in \cite{rag16}.   
         Let $\Pi, \Sigma, \epsilon, \eta$ be as defined in \cite{rag16}, which we can interpret as follows by using notations in this article: 
         \begin{align*}   
            \Pi &= \pi^{(3)}, & \Sigma &= \pi^{(2)}, & \epsilon &= (-1)^{ m+ \delta + 1 + w_{\boldsymbol{\lambda}}} = (-1)^{ m + \delta + \frac{l_3}{2}  }, & \eta &= \pm. 
         \end{align*}
         Then Raghuram has defined a complex number $\langle [\Pi_\infty]^\epsilon, [\Sigma_\infty]^\eta \rangle$ in \cite[(2.47)]{rag16} by using a cup product of cohomology classes. 
         The result in \cite{sun17} yields that  $\langle [\Pi_\infty]^\epsilon, [\Sigma_\infty]^\eta \rangle$ does not equal zero (see \cite[Theorem 2.48]{rag16}). 
         Raghram also defines $p^{\epsilon, \eta}_\infty(\boldsymbol{w}_{\boldsymbol{\lambda}}, \boldsymbol{n}_{\boldsymbol{\lambda}})$ to be $\langle [\Pi_\infty]^\epsilon, [\Sigma_\infty]^\eta \rangle^{-1}$,  
         and this $p^{\epsilon, \eta}_\infty(\boldsymbol{w}_{\boldsymbol{\lambda}}, \boldsymbol{n}_{\boldsymbol{\lambda}})$ shows up in his formula for the critical values in \cite[Theorem 2.50]{rag16}. 
         Note that, in this article, we write the weight of $\Pi$ and $\Sigma$ (in the sense of \cite{rag16}) as $\boldsymbol{w}_{\boldsymbol{\lambda}}$ and $\boldsymbol{n}_{\boldsymbol{\lambda}}$ respectively. 
         In \cite[Theorem 2.50]{rag16}, the complex number $p^{\epsilon, \eta}_\infty(\boldsymbol{w}_{\boldsymbol{\lambda}}, \boldsymbol{n}_{\boldsymbol{\lambda}})$ appears as an ``unknown constant'', but
          due to the proof of Theorem \ref{thm:main}, we find the following explicit formula on it:
         \begin{multline*}
           \qquad \qquad    \left(  \Omega_{\pi^{(3)}} \Omega^\pm_{\pi^{(2)}} p^{\epsilon, \eta}_\infty(\boldsymbol{w}_{\boldsymbol{\lambda}}, \boldsymbol{n}_{\boldsymbol{\lambda}}) \right)^{-1} 
              = \frac{ \langle [\Pi_\infty]^\epsilon, [\Sigma_\infty]^\eta \rangle 
                            }{   \Omega_{\pi^{(3)}} \Omega^\pm_{\pi^{(2)}}  } \\
               =    (-1)^{ \delta  } 
                    \sqrt{-1}^{\frac{l_3}{2}-m+1}                  
                      \binom{ \frac{l_3}{2} - 1  }{m-\frac{l_3}{2}-1}  
                     \binom{  \frac{l_3}{2} - 1  }{\frac{l_3}{2}+l_2-m }
                     \frac{  L_\infty(m, {\mathcal M} ( \pi^{(3)} \times \pi^{(2)}   ) )  }{  \Omega_{\pi^{(3)}} \Omega^\pm_{\pi^{(2)}} }. 
         \end{multline*}
         Hence Theorem \ref{thm:main} refines \cite[Theorem 2.50]{rag16} and it is consistent with \cite[page 107, Period Conjecture]{coa89}. 

\item Let us recall one of the main results proposed in \cite{jan19}.   
According to \cite[(25) and Section 4.3, p.~6569]{jan19}, 
define a pairing $e_\infty( |\cdot|^m_\infty, \cdot, \cdot ) \colon  {\mathcal W}(\pi^{(3)}_\infty, \psi_\infty )^{K_3}  \otimes {\mathcal W}(\pi^{(2)}_\infty, \psi^{-1}_\infty )^{K_2}   \to {\mathbf C}$ so that
\begin{align*}
    e_\infty( |\cdot|^m_\infty, {W}_3, W_2) 
    =    \frac{1}{L(m-\frac{3}{2}, \pi^{(3)}_\infty \times \pi^{(2)}_\infty  )} 
         \int_{ {\rm N}_2({\mathbf R} )  \backslash {\rm GL}_2({\mathbf R})  } 
          {W}_3 ( \iota(g) ) W_2 ( g )
           |\det g|^{m-\frac{3}{2}}_\infty  
         {\rm d}g_\infty 
\end{align*}
for ${W}_3\in {\mathcal W}(\pi^{(3)}_\infty,\psi_\infty )^{K_3}$ and $W_2 \in {\mathcal W}(\pi^{(2)}_\infty, \psi^{-1}_\infty )^{K_2}$.
Let $E/ {\mathbf Q}$ be a sufficiently large finite extension of $\mathbf{Q}$ (depending on the weight of $\pi_\infty^{(3)}$). 
Take arbitrary $m, m^\prime \in {\mathbf Z}$ so that
  that $m-m^\prime$ is even 
  and both $L(m-\frac{3}{2}, \pi^{(3)} \times \pi^{(2)} )$ and $L(m^\prime-\frac{3}{2}, \pi^{(3)} \times \pi^{(2)} )$ are critical.
Then  \cite[Corollary 4.12]{jan19} yields that, 
for each ${W}_3\in {\mathcal W}(\pi^{(3)}_\infty, \psi_\infty )^{K_3}$ and $W_2 \in {\mathcal W}(\pi^{(2)}_\infty, \psi^{-1}_\infty )^{K_2}$,   
there exists an element $c_{m,m^\prime} \in E$ (a priori depending on $m$ and $m'$) satisfying    
\begin{align*}
e_\infty( |\cdot|^m_\infty, {W}_3, W_2)
  = c_{m,m^\prime} e_\infty( |\cdot|^{m^\prime}_\infty, {W}_3, W_2).   
\end{align*}
On the other hand, 
assuming $W_3 = W^{-\lambda_2}_{ \pi^{(3)}, \infty }$ and $W_2 = W^+_{\pi^{(2)}, \infty}$, we can verify by using  \cite[Corollary 9.7]{him} that
\begin{align}\label{eq:perintrel}
e_\infty( |\cdot|^m_\infty, {W}_3, W_2) 
=    (-1)^{\delta +\lambda_2} \in E
\end{align}
holds for each $m$, which is independent of $m$. 
Therefore we can conclude that $c_{m,m'}$ equals $1$ for each $m$ and $m'$, and Theorem \ref{thm:main} thus refines the period relations in \cite[Theorem 4.13]{jan19} in the $\mathrm{GL}_3\times \mathrm{GL}_2$ case.
For general $n$, \cite{jan19} obtains similar result under some conditions; see \cite[Conjecture 4.3, Conjecture 4.6]{jan19} for details. 
Meanwhile, Ishii and Miyazaki recently generalize the results of \cite{him} to the case ${\rm GL}_{n+1} \times {\rm GL}_n$ over ${\mathbf C}$ in \cite{im}, 
   which also refine the results in \cite{jan19} unconditionally.      
\end{enumerate}
 Since there had been no explicit formula for the archimedean zeta integral  $\langle [\Pi_\infty]^\epsilon, [\Sigma_\infty]^\eta \rangle$ or $e_\infty( |\cdot|^m_\infty, {W}_3, W_2)$,   
    the definition of periods for ${\mathcal M}(\pi^{(3)} \times \pi^{(2)} )$ in \cite{rag16} and \cite{jan} depends on the critical points $m$; compare with \cite[Theorem A]{jan}. 
From the view point of Coates and Perrin-Riou's conjecture (refer to \cite{cp89}, \cite[page 107, Period Conjecture, page 111, Principal Conjecture]{coa89}),  
   the archimedean zeta integral must be given as a product of the period $\Omega_{\pi^{(3)}} \Omega^\pm_{\pi^{(2)}}$ and ${\mathcal L}^{\sqrt{-1}}_\infty( {\mathcal M}(\pi^{(3)} \times \pi^{(2)} )(m)  )$ in Remark \ref{rem:relCP}.    
  Januszewski's result implies that his periods 
     coincides with products of  $\Omega_{\pi^{(3)}} \Omega^\pm_{\pi^{(2)}}$, ${\mathcal L}^{\sqrt{-1}}_\infty( {\mathcal M}(\pi^{(3)} \times \pi^{(2)} )(m)  )$ 
      and implicit algebraic numbers depending on $m$, but we think that it is crucial to remove this ambiguous multiple. 
   See also \cite[Introduction]{jan} on this point.
We expect that the identity (\ref{eq:perintrel}) should be useful for the further refinement of the interpolation formula of the $p$-adic Rankin--Selberg $L$-functions for ${\rm GL}_3 \times {\rm GL}_2$  constructed in \cite[Theorem A]{jan}.
\end{rem}

\subsection{Algebraicity}\label{sec:alg}

Since Theorem \ref{thm:main} refines the formula in \cite[Theorem 2.50]{rag16}, 
   it immediately deduces the following corollary: 

\begin{cor}\label{cor:critalg}
(compare with \cite[Theorem 2.50]{rag16})
{\itshape 
Let $E(\pi^{(3)}, \pi^{(2)} )$ be a number field containing $\sqrt{-1}$, $E_{\pi^{(2)}}$ and $E_{\pi^{(3)}}$. 
Then, if $ (-1)^{m+\delta_3+  \frac{l_3}{2} }
                = \pm 1$, we have 
\begin{align*}
\frac{  L(m, {\mathcal M}(\pi^{(3)} \times \pi^{(2)})   ) }{ \Omega_{\pi^{(3)}} \Omega^\pm_{\pi^{(2)}} } \in E(\pi^{(3)}, \pi^{(2)} ).  
\end{align*}
}
\end{cor}

Assuming Deligne's conjecture  on critical values \cite[Conjecture~1.8]{del79} holds, we can deduce from 
   Corollary \ref{cor:critalg} an interesting period relation, 
   which implies that Raghuram and Shahidi's period $\Omega_{\pi^{(3)}}$ for $\pi^{(3)}$ has a motivic background as well as the canonical periods $\Omega^\pm_{\pi^{(2)}}$. We summarize the argument in the following remark: 

\begin{rem}\label{rem:RelDelPer} 
Let us admit the existence of the motive ${\mathcal M}[ \pi^{(3)} ]$ and the validity of Deligne's conjecture \cite[Conjecture 1.8]{del79}. In this remark, we write $a\sim b$ for $a,b\in \mathbf{C}$ if $a=bc$ holds for some $c\in \overline{\mathbf Q}^\times$.
For a pure motive ${\mathcal M}$ over ${\mathbf Q}$, let $\delta({\mathcal M})$ and $c^\pm({\mathcal M})$ be the periods for ${\mathcal M}$ defined in \cite{del79}.   
Besides these periods, 
    Yoshida defines period invariants $c_p({\mathcal M})$ of type $\{(\underbrace{2,\dotsc,2}_{p},\underbrace{1,\dotsc,1}_{t-2p},\underbrace{0,\dotsc,0}_p);(1,1)\}$ for integers $p$ satisfying a certain admissibility condition in \cite[Section 2]{yos01}; here $t$ denotes the  number of jumps of the Hodge filtration on $H_B(\mathcal{M})\otimes \mathbf{C}$.
Then applying \cite[Proposition 12]{yos01},  we find that\footnote{Here we temporarily denote $\pi^{(3)}_\infty=\mathrm{Ind}^{\mathrm{GL}_3(\mathbf{R})}_{\mathrm{P}_{2,1}(\mathbf{R})}(D_{\nu_3,l_3}\boxtimes \chi_{\nu_3}^{\delta_3})$ to avoid the notational confusion concerning $\delta$.} 
\begin{align}\label{eq:yos}
  c^\pm({\mathcal M}(\pi^{(3)} \times \pi^{(2)}  )  )  
     \sim c_1( {\mathcal M}[\pi^{(3)}] )  \delta(  {\mathcal M}[\pi^{(2)}]  )
           c^{\pm (-1)^{  \delta_3 +  \frac{l_3}{2}    } }(  {\mathcal M}[\pi^{(2)}]  ) 
\end{align}
where $c_1(\mathcal{M}[\pi^{(3)}])$ denotes the period invariant of $\mathcal{M}[\pi^{(3)}]$ of type $\{(2,1,0);(1,1)\}$. Note that $\delta(  {\mathcal M}[\pi^{(2)}]  ) \sim  (2\pi\sqrt{-1})^{-l_2}$ holds by \cite[(7.8.2)]{del79}. Moreover from
\cite[(5.1.8)]{del79}  we also have the relation 
\begin{align*}
    c^\pm(      {\mathcal M}(\pi^{(3)} \times \pi^{(2)})  )
   \sim  ( 2\pi\sqrt{-1} )^{    - d^\pm(   {\mathcal M}(\pi^{(3)} \times \pi^{(2)}  )  )  m     }  
        c^{  \pm (-1)^m}({\mathcal M}(\pi^{(3)} \times \pi^{(2)}  )(m)     ).  
\end{align*}
Then Corollary \ref{cor:critalg} and Deligne's conjecture yield that, if $(-1)^{m + \delta_3  + \frac{l_3}{2} } = \pm 1$ holds, both of the following ratios are algebraic numbers:
\begin{align*}
  \frac{  L(0, {\mathcal M}(\pi^{(3)} \times \pi^{(2)}  )  (m)  ) }{\Omega_{\pi^{(3)}} \Omega^\pm_{\pi^{(2)}} } 
   & \sim 
    \frac{  L_{\rm fin}(0, {\mathcal M}(\pi^{(3)} \times \pi^{(2)}  )  (m)  )  
                   }{  (2\pi \sqrt{-1} )^{  3m - l_2 - \frac{l_3}{2} }   \Omega_{\pi^{(3)}} \Omega^\pm_{\pi^{(2)}} }     
    \sim  
    \frac{  L_{\rm fin}(0, {\mathcal M}(\pi^{(3)} \times \pi^{(2)}  )  (m)  )  
                   }{  (2\pi \sqrt{-1} )^{  3m  - \frac{l_3}{2} }   \delta({\mathcal M} [\pi^{(2)}]  ) \Omega_{\pi^{(3)}} \Omega^\pm_{\pi^{(2)}} },      
                   \\    
  \frac{  L_{\rm fin}(0, {\mathcal M}(\pi^{(3)} \times \pi^{(2)}  )  (m)  ) }{  c^{  + }({\mathcal M}(\pi^{(3)} \times \pi^{(2)}  )(m)     )   } 
   & \sim  
  \frac{  L_{\rm fin}(0, {\mathcal M}(\pi^{(3)} \times \pi^{(2)}  )  (m)  ) }{    (2\pi\sqrt{-1})^{3m}   c^{(-1)^m}(      {\mathcal M}(\pi^{(3)} \times \pi^{(2)})  )   }.   
\end{align*}
Hence (\ref{eq:yos}) shows that 
\begin{align*}
    (2\pi \sqrt{-1} )^{  - \frac{l_3}{2} }   \delta({\mathcal M} [\pi^{(2)}]  ) \Omega_{\pi^{(3)}} \Omega^\pm_{\pi^{(2)}}
    \sim c^{(-1)^m}(      {\mathcal M}(\pi^{(3)} \times \pi^{(2)})  ) 
    \sim c_1( {\mathcal M}[\pi^{(3)}] )  \delta(  {\mathcal M}[\pi^{(2)}]  ) 
            c^{(-1)^{\delta_3 + \frac{l_3}{2} +m}}(  {\mathcal M}[\pi^{(2)}]  ).
\end{align*}
On the canonical periods $\Omega^\pm_{\pi^{(2)}}$, Hida shows in \cite[Theorem 8.1]{hid94} that 
\begin{align*}
   2\pi\sqrt{-1} \Omega^\mp_{\pi^{(2)}} \sim c^\pm(  {\mathcal M}[\pi^{(2)}](1)  ) 
     \sim (2\pi\sqrt{-1}) c^\mp(  {\mathcal M}[\pi^{(2)}]  )
\end{align*}
holds. Furthermore we also find that  
  $c_1( {\mathcal M}[\pi^{(3)}] ) \sim c^+( {\mathcal M}[\pi^{(3)}] ) c^-( {\mathcal M}[\pi^{(3)}] )$ holds due to \cite[Theorem~3]{yos01} and comparison of the types of both sides. 
Combining all of them, we can derive the period relation 
\begin{align*}
 \Omega_{\pi^{(3)}}
 \sim (2\pi \sqrt{-1} )^{   \frac{l_3}{2}  } c_1( {\mathcal M}[\pi^{(3)}] )
 \sim (2\pi \sqrt{-1} )^{   \frac{l_3}{2}  } c^+( {\mathcal M}[\pi^{(3)}] ) c^-( {\mathcal M}[\pi^{(3)}] )  
\end{align*}
under the condition  $(-1)^{m+\delta_3+ \frac{l_3}{2}  } = \pm 1$. This explains the motivic background of the Whittaker period $\Omega_{\pi^{(3)}}$.
\end{rem}

\renewcommand{\thesection}{\Alph{section}}
\renewcommand{\theequation}{\Alph{section}.\arabic{equation}}
\setcounter{section}{0}
\setcounter{thm}{0}
\setcounter{equation}{0}

\section{Appendix: Construction of the Eichler--Shimura maps for ${\rm GL}_3$}\label{sec:AppA}

We retain the settings and notation in Section~\ref{sec:ESmap}; in particular,  let $\pi^{(3)}$ denote a cohomological irreducible  cuspidal automorphic representation of $\mathrm{GL}_3(\mathbf{Q}_{\mathbf A})$, which has the minimal $\mathrm{O}_3(\mathbf{R})$-type $\tau^{(3)}_{\boldsymbol{\lambda}}$ with $\boldsymbol{\lambda}=\boldsymbol{\lambda}^{(3)}=(\lambda_3,\delta)\in \Lambda_3^{\mathrm{coh}}$ at the infinite place. In Appendix~\ref{sec:AppA} we summarize the construction of the {\em Eichler--Shimura map} for $\mathrm{GL}_3$; namely, we briefly explain how to attach a specific class $\delta^{(3),i}(\boldsymbol{f})$ of the $(\mathfrak{gl}_3(\mathbf{R}), K_3)$-cohomology  $H^i(\mathfrak{gl}_3(\mathbf{R}), K_3;H_{\pi^{(3)},K_3}\otimes L^{(3)}(\boldsymbol{w}_{\boldsymbol{\lambda}};\mathbf{C}))$ to a cusp form $\boldsymbol{f}=\begin{pmatrix}f^{\boldsymbol{\lambda}}_{\lambda_3} & f^{\boldsymbol{\lambda}}_{\lambda_3-1} &\dotsc&f^{\boldsymbol{\lambda}}_{-\lambda_3}\end{pmatrix}\in \mathcal{S}^{(3)}_{\boldsymbol{\lambda}}(\mathcal{K}_3)$ for $i=2$ and $3$, where $K_3=\mathbf{R}^\times_{>0} \mathrm{SO}_3(\mathbf{R})$ and $\mathcal{K}_3$ is an open compact subgroup of $\mathrm{GL}_3(\mathbf{Q}_{\mathbf{A},\mathrm{fin}})$ acting trivially on $\pi^{(3)}$. 

Recall that we have defined $\mathcal{P}_3$ as $\mathfrak{gl}_3(\mathbf{R})/\mathrm{Lie}(K_3)$ and $\mathcal{P}_{3,\mathbf{C}}$ as its complexification. Set $L^{(3)}(\boldsymbol{w}_{\boldsymbol{\lambda}};\mathbf{C})$ be the irreducible representation of $\mathrm{GL}_3(\mathbf{R})$ defined as in Section~\ref{sec:repGL3} for $\boldsymbol{w}_{\boldsymbol{\lambda}}=(w_{\boldsymbol{\lambda}},w_{\boldsymbol{\lambda}},w_{\boldsymbol{\lambda}})=\left(\frac{\lambda_3-3}{2}, \frac{\lambda_3-3}{2}, \frac{\lambda_3-3}{2}\right)$. Then we have already seen  in Section~\ref{sec:ES} that the cohomology group $H^i(\mathfrak{gl}_3(\mathbf{R}), K_3;H_{\pi^{(3)},K_3}\otimes L^{(3)}(\boldsymbol{w}_{\boldsymbol{\lambda}};\mathbf{C}))$ is of dimension $1$ and is isomorphic to 
\begin{align*}
 \mathrm{Hom}_{\mathrm{SO}_3(\mathbf{R})}\left(\bigwedge^i \mathcal{P}_{3,\mathbf{C}},H_{\pi^{(3)},K_3}\otimes L^{(3)}(\boldsymbol{w}_{\boldsymbol{\lambda}};\mathbf{C})\right) \cong \left(H_{\pi^{(3)},K_3}\otimes L^{(3)}(\boldsymbol{w}_{\boldsymbol{\lambda}};\mathbf{C})\otimes \bigwedge^i\mathcal{P}_{3,\mathbf{C}}^*\right)^{\mathrm{SO}_3(\mathbf{R})}.
\end{align*}
Therefore, in order to define the Eichler--Shimura map $\delta^{(3),i}$, it suffices to construct a non-trivial element of the rightmost group from a given cusp form $\boldsymbol{f}=\begin{pmatrix}f^{\boldsymbol{\lambda}}_{\lambda_3} & f^{\boldsymbol{\lambda}}_{\lambda_3-1} &\dotsc&f^{\boldsymbol{\lambda}}_{-\lambda_3}\end{pmatrix}$. We first analyze the action of $\mathrm{SO}_3(\mathbf{R})$ on the wedge product $\bigwedge^i \mathcal{P}_{3,\mathbf{C}}^*$, and then we construct a cohomology class so that the actions of $\mathrm{SO}_3(\mathbf{R})$ on $H_{\pi^{(3)},K_3}$, $L^{(3)}(\boldsymbol{w}_{\boldsymbol{\lambda}};\mathbf{C})$ and $\bigwedge^i \mathcal{P}^*_{3,\mathbf{C}}$ are canceled out. 

\subsection{Action on the wedge product of $\mathcal{P}_{3,\mathbf{C}}^*$} \label{sec:wedge_p}

Recall the definition of the basis $\{X_2, X_1,X_0,X_{-1},X_{-2}\}$ of $\mathcal{P}_{3,\mathbf{C}}$ from Section~\ref{sec:ES} (see (\ref{eq:defX})). Lemma~\ref{lem:Lieact} implies that this basis gives an isomorphism between $\mathcal{P}_{3,\mathbf{C}}$ and $\tau^{(3)}_{(2,0)}$ as $\mathrm{SO}_3(\mathbf{R})$-modules. For $i=2$ and $3$, the wedge product $\bigwedge^i \mathcal{P}_{3,\mathbf{C}}$ is of dimension 10 and, due to  weight calculation, we readily observe that its $\mathbf{C}$-linear dual $\bigwedge^i \mathcal{P}_{3,\mathbf{C}}^*$ is isomorphic to $\tau^{(3),\vee}_{(1,0)}\oplus \tau^{(3),\vee}_{(3,0)}$ as $\mathrm{SO}_3(\mathbf{R})$-modules (here $\tau^{(3),\vee}_{\boldsymbol{\lambda}}$ denotes the contragredient of $\tau^{(3)}_{\boldsymbol{\lambda}}$).   Now we determine an explicit basis of the $\tau^{(3),\vee}_{(3,0)}$-isotypic component of $\bigwedge^i \mathcal{P}_{3,\mathbf{C}}^*$. 

Let $\{v^{\boldsymbol{\lambda}}_j  \}_{-\lambda_3\leq j \leq \lambda_3}$ be the basis of $\tau^{(3)}_{\boldsymbol{\lambda}}$ defined in Section~\ref{sec:repSO3} 
      and $\{v^{\boldsymbol{\lambda},*}_j \}_{-\lambda_3\leq j \leq \lambda_3}$ the  dual basis of $\tau^{(3),\vee}_{\boldsymbol{\lambda}}$.

\begin{lem} \label{lem:basis_wedge}
{\itshape  
 Suppose that $i=2$ or $3$. Then there exist explicit elements $\boldsymbol{\omega}^i_{j}$ $(-3\leq j\leq 3)$ such that the correspondence $v^{(3,0),*}_j\mapsto \boldsymbol{\omega}^i_j$ gives rise to an $\mathrm{SO}_3(\mathbf{R})$-equivariant injection of $\tau^{(3),\vee}_{(3,0)}$ into $\bigwedge^i \mathcal{P}_{3,\mathbf{C}}^*$. In particular, we have $\wedge^i \mathrm{Ad}^*(u)
\begin{pmatrix}
 \boldsymbol{\omega}_3^i & \dotsc & \boldsymbol{\omega}_{-3}^i
\end{pmatrix}=\begin{pmatrix}
 \boldsymbol{\omega}_3^i & \dotsc & \boldsymbol{\omega}_{-3}^i
\end{pmatrix}{}^{\mathrm{t}}\!M_{(3,0)}(u)^{-1}$ for every $u\in \mathrm{SO}_3(\mathbf{R})$. 
}
\end{lem} 

\begin{proof}
 For $2\geq j>j'>j''\geq -2$, let us abbreviate the wedge product $X_j\wedge X_{j'}$ and $X_{j}\wedge X_{j'} \wedge X_{j''}$ as $X_{j,j'}$ and $X_{j,j',j''}$ respectively. Note that $\mathbf{B}^2=\{X_{j,j'}\}_{2\geq j>j'\geq -2}$ (resp.\ $\mathbf{B}^3=\{X_{j,j',j''}\}_{2\geq j>j'>j''\geq -2}$) forms a basis of $\bigwedge^2\mathcal{P}_{3,\mathbf{C}}$ (resp.\ $\bigwedge^3\mathcal{P}_{3,\mathbf{C}}$). Set
\begin{multline*}
 \mathbf{X}^{i}= \\
\begin{cases}
 \begin{pmatrix}
  X_{2,1} & X_{2,0} & X_{2,-1} & X_{2,-2} & X_{1,0} & X_{1,-1} & X_{1,-2} & X_{0,-1} & X_{0,-2} & X_{-1,-2}
 \end{pmatrix} & \text{for $i=2$}, \\
\begin{pmatrix}
  X_{2,1,0} & X_{2,1,-1} & X_{2,1,-2} & X_{2,0,-1} & X_{2,0,-2} & X_{2,-1,-2} & X_{1,0,-1} & X_{1,0,-2} & X_{1,-1,-2} & X_{0,-1,-2}
 \end{pmatrix} & \text{for $i=3$}. 
\end{cases}
\end{multline*}
For each $u\in \mathrm{SO}_3(\mathbf{R})$, let  $\widetilde{M}^i(u)$ denote the matrix presentation of the action of $u$ on $\bigwedge^i \mathcal{P}_{3,\mathbf{C}}$ with respect to the basis $\mathbf{X}^i$; namely set $\wedge^i \mathrm{Ad}(u) \mathbf{X}^i=\mathbf{X}^i \widetilde{M}^i(u)$. Then, by direct calculation using Mathematica and Maxima, we observe that $\widetilde{P}^{i,-1}\widetilde{M}^i(u)\widetilde{P}^{i}=
\begin{pmatrix}
 M_{(1,0)}(u) & 0 \\ 0 & M_{(3,0)}(u)
\end{pmatrix}$ holds for
\begin{align*}
 \widetilde{P}^{2} &=\begin{pmatrix}
  0 &0 &0 & 1 &0 &0 &0 &0 &0 &0 \\
0 &0 &0 &0 & \frac{1}{2} &0 &0 &0 &0 &0 \\
-\frac{1}{5} &0 &0 &0 &0 & \frac{1}{5} &0 &0 &0 &0 \\
0 &-\frac{1}{10} &0 &0 &0 &0 & \frac{1}{20} &0 &0 &0 \\
\frac{3}{5} &0 &0 &0 &0 &\frac{2}{5} &0 &0 &0 &0 \\
0 & \frac{1}{5} &0 &0 &0 &0 &\frac{2}{5} &0 &0 &0 \\
0 &0 &-\frac{1}{5} &0 &0 &0 &0 & \frac{1}{5} &0 &0 \\
0 &0 & \frac{3}{5} &0 &0 &0 &0 &\frac{2}{5} & 0&  0\\
0 &0 &0 &0 & 0 & 0 & 0 & 0 & \frac{1}{2} & 0 \\
0 &0 &0 &0 &0 &0 &0 & 0&0 & 1
      \end{pmatrix}, & 
\widetilde{P}^{3} &=
\begin{pmatrix}
 0 &0 &0 & 3 &0 &0 &0 &0 &0 &0 \\
0 &0 &0 &0 & 1 &0 &0 &0 &0 &0 \\
1 &0 &0 &0 &0 & \frac{1}{5} &0 &0 &0 &0 \\
-2 & 0 &0 &0 &0 &\frac{3}{5} &0 &0 &0 &0 \\
0 & \frac{1}{2} &0 &0 &0 &0 &\frac{3}{10} &0 &0 &0 \\
0 &0 &1 &0 &0 &0 &0 & \frac{1}{5} &0 &0 \\
0 &-4 & 0 &0 &0 &0 & \frac{3}{5} & 0&  0 & 0\\
0 &0 &-2 &0 & 0 & 0 & 0 & \frac{3}{5} & 0 & 0 \\
0 &0 &0 &0 &0 &0 &0 & 0&1 & 0 \\
0 &0 &0 &0 &0 &0 &0 & 0&0 & 3
\end{pmatrix}.
\end{align*}
Let us define $\mathrm{d}\mathbf{X}^i$ similarly to $\mathbf{X}^i$ replacing $X_{j,j'}$ and $X_{j,j',j''}$ by $\mathrm{d}X_{j,j'}:=\mathrm{d}X_j\wedge \mathrm{d}X_{j'}$ and $\mathrm{d}X_{j,j',j''}:=\mathrm{d}X_j\wedge \mathrm{d}X_{j'}\wedge \mathrm{d}X_{j''}$ respectively (recall that $\{\mathrm{d}X_j\}_{-2\leq j\leq 2}$ denotes the dual basis of $\mathcal{P}_{3,\mathbf{C}}$). Then it suffices to define $\{\boldsymbol{\omega}^i_j\}_{-3\leq j\leq 3}$ as 
\begin{align*}
 \begin{pmatrix}
\boldsymbol{\omega}^i_3 & \boldsymbol{\omega}^i_2 & \boldsymbol{\omega}^i_1 & \boldsymbol{\omega}^i_0 & \boldsymbol{\omega}^i_{-1} & \boldsymbol{\omega}^i_{-2} & \boldsymbol{\omega}^i_{-3}
 \end{pmatrix}=\mathrm{d}\mathbf{X}^iP^i
\end{align*}
for 
\begin{align}  \label{eq:defP2}
 P^2 &=\begin{pmatrix}
	1 &0 &0 &0 &0 &0 &0 \\
0 & 2 &0 &0 &0 &0 &0 \\
0 &0 & 3 &0 &0 &0 &0 \\
0 &0 &0 & 4 &0 &0 &0 \\
0 &0 &1 &0 &0 &0 &0 \\
0 &0 &0 &2 &0 &0 &0 \\
0 &0 &0 &0 & 3 &0 &0 \\
0 &0 &0 &0 &1 & 0&  0\\
0 & 0 & 0 & 0 & 0 & 2 & 0 \\
0 &0 &0 &0 & 0&0 & 1
       \end{pmatrix}, & 
P^3 &=
\begin{pmatrix}
 \frac{1}{3} &0 &0 &0 &0 &0 &0 \\
0 & 1 &0 &0 &0 &0 &0 \\
0 &0 & 2 &0 &0 &0 &0 \\
0 &0 &1 & 0 &0 &0 &0 \\
0 &0 &0 &\frac{8}{3} &0 &0 &0 \\
0 &0 &0 & 0 & 2 &0 &0 \\
0 &0 &0 &\frac{1}{3} &0 & 0&  0\\
0 & 0 & 0 & 0 & 1 & 0 & 0 \\
0 &0 &0 &0 & 0& 1 & 0 \\
0 &0 &0 &0 & 0& 0 &\frac{1}{3} 
\end{pmatrix}.
\end{align}
Note that $P^i \in \mathrm{M}_{10,7}(\mathbf{C})$ is a matrix obtained by removing first three columns of ${}^{\mathrm{t}}\!\widetilde{P}^{i,-1}$.
\end{proof}

\subsection{Construction of the cohomology class} \label{sec:const_class}

Recall from Section~\ref{sec:repGL3} that the representation space of the irreducible representation $L^{(3)}(\boldsymbol{w}_{\boldsymbol{\lambda}}; \mathbf{C})$ is a subspace of $\mathbf{C}[X,Y,Z;A,B,C]_{w_{\boldsymbol{\lambda}},w_{\boldsymbol{\lambda}}}$, the space of polynomials in variables $X,Y,Z,A,B,C$ which are homogeneous of degree $w_{\boldsymbol{\lambda}}$ both in $X,Y,Z$ and $A,B,C$. Recall also that the irreducible representation $\tau^{(3)}_{\boldsymbol{\lambda}}$ of $\mathrm{SO}_3(\mathbf{R})$ is realized as a certain quotient of the space of homogeneous polynomials of degree $\lambda_3$ in variables $z_1,z_2,z_3$. Now let us consider
\begin{multline*}
 P(X,Y,Z,A,B,C,z_1,z_2,z_3):= \left( 
\begin{pmatrix}
 X & Y & Z
\end{pmatrix} 
\begin{pmatrix}
 z_1 \\ z_2 \\ z_3
\end{pmatrix}\right)^{w_{\boldsymbol{\lambda}}} \otimes \left( 
\begin{pmatrix}
 A & B & C
\end{pmatrix} 
\begin{pmatrix}
 z_1 \\ z_2 \\ z_3
\end{pmatrix}\right)^{w_{\boldsymbol{\lambda}}} 
\begin{pmatrix}
 v_3^{(3,\delta)} & \dotsc & v_{-3}^{(3,\delta)}
\end{pmatrix} \\
\in (\mathbf{C}[X,Y,Z;A,B,C]_{w_{\boldsymbol{\lambda}}} \otimes \mathbf{C}[z_1,z_2,z_3]_{\lambda_3}/\mathcal{V}_{\boldsymbol{\lambda}})^{\oplus 7}.
\end{multline*}

\begin{lem} \label{lem:construction_delta}
{\itshape 
 Let $P(X,Y,Z,A,B,C,z_1,z_2,z_3)$ be as above.
\begin{enumerate}[label={\rm (\roman*)}]
 \item There exists a unique matrix 
\begin{multline*}
\hspace*{4em} \mathcal{P}(X,Y,Z,A,B,C) =
\begin{pmatrix}
  \mathcal{P}_{\lambda_3,3}(X,Y,Z,A,B,C) & \dotsc &  \mathcal{P}_{\lambda_3,-3}(X,Y,Z,A,B,C) \\
\vdots & \ddots & \vdots \\
  \mathcal{P}_{-\lambda_3,3}(X,Y,Z,A,B,C) & \dotsc &  \mathcal{P}_{-\lambda_3,-3}(X,Y,Z,A,B,C) 
\end{pmatrix} \\
\in \mathrm{M}_{2\lambda_3+1,7}(\mathbf{C}[X,Y,Z;A,B,C]_{w_{\boldsymbol{\lambda}}, w_{\boldsymbol{\lambda}}})
\end{multline*}
such that $P(X,Y,Z,A,B,C,z_1,z_2,z_3)=
\begin{pmatrix}
 v^{\boldsymbol{\lambda}}_{\lambda_3} & \dotsc & v^{\boldsymbol{\lambda}}_{-\lambda_3}
\end{pmatrix} \mathcal{P}(X,Y,Z,A,B,C)$ holds. Moreover every entry of $\mathcal{P}(X,Y,Z,A,B,C)$ is an element of $L^{(3)}(\boldsymbol{w}_{\boldsymbol{\lambda}};\mathbf{C})$ all of whose coefficients are contained in $\mathbf{Z}[2^{-1},\sqrt{-1}]$.
 \item For each $u\in \mathrm{SO}_3(\mathbf{R})$, we have 
\begin{align*}
 (\varrho_{\boldsymbol{\lambda}}^{(3)} (u) \mathcal{P})(X,Y,Z,A,B,C)=M_{\boldsymbol{\lambda}}^{-1}(u) \mathcal{P}(X,Y,Z,A,B,C)M_{(3,0)}(u).
\end{align*}
 \item For a cusp form $\boldsymbol{f}=\begin{pmatrix}f^{\boldsymbol{\lambda}}_{\lambda_3} & f^{\boldsymbol{\lambda}}_{\lambda_3-1} &\dotsc&f^{\boldsymbol{\lambda}}_{-\lambda_3}\end{pmatrix} \in \mathcal{S}^{(3)}_{\boldsymbol{\lambda}}(\mathcal{K}_3)$, the element $\delta^{(3),i}(\boldsymbol{f})$ defined as 
\begin{align*}
 \delta^{(3),i}(\boldsymbol{f})=
\begin{pmatrix}
f^{\boldsymbol{\lambda}}_{\lambda_3} & f^{\boldsymbol{\lambda}}_{\lambda_3-1} &\dotsc&f^{\boldsymbol{\lambda}}_{-\lambda_3}
\end{pmatrix} \mathcal{P}(X,Y,Z,A,B,C) 
\begin{pmatrix}
 \boldsymbol{\omega}^i_3 \\ \vdots \\ \boldsymbol{\omega}^i_{-3}
\end{pmatrix}
\end{align*}
gives rise to a class of $H^i_{\mathrm{cusp}}(Y^{(3)}_{\mathcal{K}_3}, \mathcal{L}^{(3)}(\boldsymbol{w}_{\boldsymbol{\lambda}};\mathbf{C}))$.
\end{enumerate}
}
\end{lem}

\begin{proof}
 \begin{enumerate}
  \item The first statement obviously follows from the identity
\begin{multline*}
\left( 
\begin{pmatrix}
 X & Y & Z
\end{pmatrix} 
\begin{pmatrix}
 z_1 \\ z_2 \\ z_3
\end{pmatrix}\right)^{w_{\boldsymbol{\lambda}}} \otimes \left( 
\begin{pmatrix}
 A & B & C
\end{pmatrix} 
\begin{pmatrix}
 z_1 \\ z_2 \\ z_3
\end{pmatrix}\right)^{w_{\boldsymbol{\lambda}}} =(Xz_1+Yz_2+Zz_3)^{w_{\boldsymbol{\lambda}}}\otimes (Az_1+Bz_2+Cz_3)^{w_{\boldsymbol{\lambda}}}  \\
=\left(\dfrac{X-\sqrt{-1}Y}{2}v^{(1,0)}_1-\dfrac{X+\sqrt{-1}Y}{2}v^{(1,0)}_{-1}+Zv^{(1,0)}_0 \right)^{w_{\boldsymbol{\lambda}}} \otimes \left(\dfrac{A-\sqrt{-1}B}{2}v^{(1,0)}_1-\dfrac{A+\sqrt{-1}B}{2}v^{(1,0)}_{-1}+Cv^{(1,0)}_0 \right)^{w_{\boldsymbol{\lambda}}}
\end{multline*}
and the multiplicativity of $v^{\boldsymbol{\lambda}}_{\mu}$'s (see (\ref{eq:vmult}); recall that we have $2w_{\boldsymbol{\lambda}}+3=\lambda_3$). We can also verify the second statement by observing 
\begin{align*}
 \iota_{w_{\boldsymbol{\lambda}},w_{\boldsymbol{\lambda}}} &\left( (Xz_1+Yz_2+Zz_3)^{w_{\boldsymbol{\lambda}}}\otimes (Az_1+Bz_2+Cz_3)^{w_{\boldsymbol{\lambda}}}\right) \\
&=\left(\dfrac{\partial^2}{\partial X\partial A}+\dfrac{\partial^2}{\partial Y\partial B}+\dfrac{\partial^2}{\partial Z \partial Z}\right) (Xz_1+Yz_2+Zz_3)^{w_{\boldsymbol{\lambda}}}\otimes (Az_1+Bz_2+Cz_3)^{w_{\boldsymbol{\lambda}}} \\
&= w^2_{\boldsymbol{\lambda}} (z_1^2+z_2^2+z_3^2)  (Xz_1+Yz_2+Zz_3)^{w_{\boldsymbol{\lambda}}-1}\otimes (Az_1+Bz_2+Cz_3)^{w_{\boldsymbol{\lambda}}-1} \equiv 0 \mod{\mathcal{V}_{\boldsymbol{\lambda}}}.
\end{align*}
  \item Let $u\in \mathrm{SO}_3(\mathbf{R})$. By definition, we have
\begin{align*}
 (\varrho_{\boldsymbol{w}_{\boldsymbol{\lambda}}}(u)P)&(X,Y,Z,A,B,C,z_1,z_2,z_3)  \\
&=\left( 
\begin{pmatrix}
 X & Y & Z
\end{pmatrix}u {}^tu
\begin{pmatrix}
 z_1 \\ z_2 \\ z_3
\end{pmatrix}\right)^{w_{\boldsymbol{\lambda}}} \otimes \left( 
\begin{pmatrix}
 A & B & C
\end{pmatrix} {}^tu^{-1} {}^tu
\begin{pmatrix}
 z_1 \\ z_2 \\ z_3
\end{pmatrix}\right)^{w_{\boldsymbol{\lambda}}} 
\tau^{(3)}_{(3,0)}(u) \begin{pmatrix}
 v_3^{(3,0)} & \dotsc & v_3^{(3,0)}
\end{pmatrix} \\
&=P(X,Y,Z,A,B,C,z_1,z_2,z_3) M_{(3,0)}(u)=\begin{pmatrix}
 v^{\boldsymbol{\lambda}}_{\lambda_3} & \dotsc & v^{\boldsymbol{\lambda}}_{-\lambda_3}
\end{pmatrix} \mathcal{P}(X,Y,Z,A,B,C) M_{(3,0)}(u).
\end{align*}
On the other hand, we can calculate $(\varrho_{\boldsymbol{w}_{\boldsymbol{\lambda}}}(u)P)(X,Y,Z,A,B,C,z_1,z_2,z_3)$ by using (i) as 
\begin{align*}
 (\varrho_{\boldsymbol{w}_{\boldsymbol{\lambda}}}(u)P)(X,Y,Z,A,B,C,z_1,z_2,z_3)&=\tau^{(3)}_{\boldsymbol{\lambda}}(u)
\begin{pmatrix}
 v^{\boldsymbol{\lambda}}_{\lambda_3} & \dotsc & v^{\boldsymbol{\lambda}}_{-\lambda_3}
\end{pmatrix} (\varrho_{\boldsymbol{w}_{\boldsymbol{\lambda}}}(u)\mathcal{P})(X,Y,Z,A,B,C)  \\
&=\begin{pmatrix}
 v^{\boldsymbol{\lambda}}_{\lambda_3} & \dotsc & v^{\boldsymbol{\lambda}}_{-\lambda_3}
\end{pmatrix}M_{\lambda}(u) (\varrho_{\boldsymbol{w}_{\boldsymbol{\lambda}}}(u)\mathcal{P})(X,Y,Z,A,B,C).
\end{align*}
Comparing  these two equations, we obtain the desired result.
  \item Since $H^i_{\mathrm{cusp}}(Y^{(3)}_{\mathcal{K}_3}, \mathcal{L}^{(3)}(\boldsymbol{w}_{\boldsymbol{\lambda}};\mathbf{C}))\cong (H_{\pi^{(3)},K_3}\otimes L^{(3)}(\boldsymbol{w}_{\boldsymbol{\lambda}};\mathbf{C})\otimes \bigwedge^i \mathcal{P}_{3,\mathbf{C}}^*)^{\mathrm{SO}_3(\mathbf{R})} \otimes (\pi^{(3)}_{\mathrm{fin}})^{\mathcal{K}_3}$, it suffices to check that the archimedean component of $\delta^{(3),i}(\boldsymbol{f})_\infty$ is fixed by the action of $\mathrm{SO}_3(\mathbf{R})$, which obviously follows from (ii); for each $u\in \mathrm{SO}_3(\mathbf{R})$ we have 
\begin{align*}
\qquad \quad  (\varrho_{\boldsymbol{w}_{\boldsymbol{\lambda}}}&(u)\delta^{(3),i})(\boldsymbol{f})_\infty  \\
&=\tau^{(3)}_{\boldsymbol{\lambda}}(u) 
\begin{pmatrix}
 f^{\boldsymbol{\lambda}}_{\lambda_3,\infty} & \dotsc & f^{\boldsymbol{\lambda}}_{-\lambda_3,\infty}
\end{pmatrix} (\varrho_{\boldsymbol{w}_{\boldsymbol{\lambda}}}(u)\mathcal{P}(X,Y,Z,A,B,C)) \,\,  {\vphantom{\wedge^i}}^{\mathrm{t}}\!\!\left(\wedge^i \mathrm{Ad}^*(u)  \begin{pmatrix} \boldsymbol{\omega}^i_3 & \dotsc & \boldsymbol{\omega}^i_{-3}									  \end{pmatrix}\right) \\
&=\begin{pmatrix}
 f^{\boldsymbol{\lambda}}_{\lambda_3,\infty} & \dotsc & f^{\boldsymbol{\lambda}}_{-\lambda_3,\infty}
\end{pmatrix} M_{\boldsymbol{\lambda}}  (u) \bigl\{ M_{\boldsymbol{\lambda}}^{-1}(u) \mathcal{P}(X,Y,Z,A,B,C)M_{(3,0)}(u)\bigr\} \,  M_{(3,0)}(u)^{-1}
\begin{pmatrix}
 \boldsymbol{\omega}^i_3 \\ \vdots \\ \boldsymbol{\omega}^i_{-3}
\end{pmatrix} \\
&=\begin{pmatrix}
 f^{\boldsymbol{\lambda}}_{\lambda_3,\infty} & \dotsc & f^{\boldsymbol{\lambda}}_{-\lambda_3,\infty}
\end{pmatrix}  \mathcal{P}(X,Y,Z,A,B,C)
\begin{pmatrix}
 \boldsymbol{\omega}^i_3 \\ \vdots \\ \boldsymbol{\omega}^i_{-3}
\end{pmatrix}=\delta^{(3),i}(\boldsymbol{f})_\infty.
\end{align*} 
 \end{enumerate}
\end{proof}

\subsection{Change of the coordinates}

We first introduce the coordinates $(y_1,y_2,x_1,x_2,x_3)$ of (the archimedean component of) the symmetric space $Y^{(3)}_{\mathcal{K}_3,\infty}=\mathrm{GL}_3(\mathbf{Q})\backslash \mathrm{GL}_3(\mathbf{R})/\mathbf{R}^\times_{>0} \mathrm{SO}_3(\mathbf{R})$, and then calculate the presentations of the differential forms $\{\boldsymbol{\omega}^i_j\}_{-3\leq j\leq 3}$ obtained in the previous subsection with respect to this coordinates. All the results of this subsection are obtained by longsome but straightforward computations (especially concerning big matrices), and thus we omit their proofs\footnote{We have checked out all the computations by using Mathematica and Maxima.}. We first introduce the explicit Iwasawa decomposition formula for $\mathrm{GL}_3(\mathbf{R})$ as follows;

\begin{lem} \label{lem:explicit_Iwasawa}
{\itshape 
 For $g=
\begin{pmatrix}
 a & b & c \\ d & e & f \\ g& h & i
\end{pmatrix}\in \mathrm{GL}_3(\mathbf{R})$, define $x_1, x_2,x_3,y_1$ and $y_2$ as follows:
\begin{align*}
 x_1&=x_1(g)=\dfrac{dg+eh+fi}{g^2+h^2+i^2}, & x_3&=x_3(g)=\dfrac{ag+bh+ci}{g^2+h^2+i^2}, \\
y_1&=y_1(g)=\sqrt{
\dfrac{d^2+e^2+f^2}{g^2+h^2+i^2}-x^2_1}, & x_2&=x_2(g)=\dfrac{1}{y_1^2}\left(
\dfrac{ad+be+cf}{g^2+h^2+i^2}-x_1x_3\right), \\
y_2&=y_2(g)=\dfrac{1}{y_1}\sqrt{\dfrac{a^2+b^2+c^2}{g^2+h^2+i^2}-x_3^2-x_2^2y_1^2}.
\end{align*}
Then we have 
\begin{align*}
 g=\begin{pmatrix}
    y_1y_2 & y_1x_2 & x_3 \\ 0 & y_1 & x_1 \\ 0 & 0 & 1
   \end{pmatrix} \cdot \pm(g^2+h^2+i^2)k \qquad \text{for }k\in \mathrm{SO}_3(\mathbf{R}).
\end{align*}
}
\end{lem}

Due to Lemma~\ref{lem:explicit_Iwasawa}, we can consider $(y_1(g), y_2(g), x_1(g), x_2(g), x_3(g))$ as the coordinates of $g$ on $Y^{(3)}_{\mathcal{K}^{(3)},\infty}$.
Set $F (g):=
\begin{pmatrix}
 y_1(g)y_2(g) & y_1(g)x_2(g) & x_3(g) \\ 0 & y_1(g) & x_1(g) \\ 0 & 0 & 1
\end{pmatrix}$ and let $\mathrm{d}F$ denote the differential of $F$ at the identity matrix $1_3$. 

\begin{lem} \label{lem:differential_eta}
{\itshape 
 Let $\{X_j\}_{-2\leq j\leq 2}$ denote the basis of $\mathcal{P}_{3,\mathbf{C}}$ introduced in Section~$\ref{sec:ES}$. Then we have
\begin{multline*}
  \mathrm{d}F
\begin{pmatrix}
 X_2 & X_1 & X_0 & X_{-1} & X_{-2}
\end{pmatrix}  \\ =
\begin{pmatrix}
 \left(\frac{\partial}{\partial y_1}\right)_{1_3} &  \left(\frac{\partial}{\partial y_2}\right)_{1_3} &  \left(\frac{\partial}{\partial x_1}\right)_{1_3} &  \left(\frac{\partial}{\partial x_2}\right)_{1_3} &  \left(\frac{\partial}{\partial x_3}\right)_{1_3}
\end{pmatrix}
\begin{pmatrix}
 1 & 0 & 1 & 0 & 1 \\
-2 & 0 & 0 & 0 & -2 \\
0 & \sqrt{-1} & 0 & \sqrt{-1} & 0 \\
-2\sqrt{-1} & 0 & 0 & 0 & 2\sqrt{-1} \\
0 & 1 & 0 & -1 & 0
\end{pmatrix}.
\end{multline*}
}
\end{lem}

Taking dual of Lemma~\ref{lem:differential_eta}, we obtain 
\begin{multline}
 (F^*)^{-1}\begin{pmatrix}
 \mathrm{d}X_2 & \mathrm{d}X_1 & \mathrm{d}X_0 & \mathrm{d}X_{-1} & \mathrm{d}X_{-2}
\end{pmatrix} \\
=\begin{pmatrix}
 (\mathrm{d}y_1)_{1_3} & (\mathrm{d}y_2)_{1_3} & (\mathrm{d}x_1)_{1_3} & (\mathrm{d}x_2)_{1_3} & (\mathrm{d}x_3)_{1_3}
\end{pmatrix}  \,\,\raisebox{2.5em}{$^{\mathrm{t}}$}\!\! \begin{pmatrix}
 1 & 0 & 1 & 0 & 1 \\
-2 & 0 & 0 & 0 & -2 \\
0 & \sqrt{-1} & 0 & \sqrt{-1} & 0 \\
-2\sqrt{-1} & 0 & 0 & 0 & 2\sqrt{-1} \\
0 & 1 & 0 & -1 & 0
\end{pmatrix}^{-1}.  \label{eq:pullback_1}
\end{multline}

Next we extend (\ref{eq:pullback_1}) to the identity between global differential forms. 

\begin{lem} \label{lem:jacobian}
{\itshape 
 Let $g\in \mathrm{GL}_3(\mathbf{R})$. Then we have 
\begin{multline*}
 \begin{pmatrix}
  (\mathrm{d}y_1)_g & (\mathrm{d}y_2)_g & (\mathrm{d}x_1)_g &(\mathrm{d}x_2)_g & (\mathrm{d}x_3)_g
 \end{pmatrix}=\begin{pmatrix}
  (\mathrm{d}y_1)_{1_3} & (\mathrm{d}y_2)_{1_3} & (\mathrm{d}x_1)_{1_3} &(\mathrm{d}x_2)_{1_3} & (\mathrm{d}x_3)_{1_3}
 \end{pmatrix}\circ L_{g}^{-1} \\
=  \begin{pmatrix}
(\mathrm{d}y_1)_{1_3} & (\mathrm{d}y_2)_{1_3} & (\mathrm{d}x_1)_{1_3} & (\mathrm{d}x_2)_{1_3} & (\mathrm{d}x_3)_{1_3}
 \end{pmatrix} \dfrac{\partial (y_1(g^{-1}),y_2(g^{-1}),x_1(g^{-1}),x_2(g^{-1}),x_3(g^{-1}))}{\partial (y_1,y_2,x_1,x_2,x_3)} \\
=\begin{pmatrix}
(\mathrm{d}y_1)_{1_3} & (\mathrm{d}y_2)_{1_3} & (\mathrm{d}x_1)_{1_3} & (\mathrm{d}x_2)_{1_3} & (\mathrm{d}x_3)_{1_3}
 \end{pmatrix}
\begin{pmatrix}
 \frac{1}{y_1} & 0 & 0 & 0 & 0 \\ 
0 & \frac{1}{y_2} & 0 & 0 & 0 \\ 
0 & 0 & \frac{1}{y_1} & 0 & -\frac{x_2}{y_1y_2} \\
0 & 0 & 0 & \frac{1}{y_2} & 0 \\ 0 & 0 & 0 & 0 & \frac{1}{y_1y_2}
\end{pmatrix}
\end{multline*}
where $L_g$ denotes the left translation.
}%
\end{lem}

Combining (\ref{eq:pullback_1}) with Lemma~\ref{lem:jacobian}, we have 
\begin{multline}
 (F^*)^{-1}\left(\begin{pmatrix}
 \mathrm{d}X_2 & \mathrm{d}X_1 & \mathrm{d}X_0 & \mathrm{d}X_{-1} & \mathrm{d}X_{-2}
\end{pmatrix} \circ L_g^{-1}\right) \\
=\begin{pmatrix}
 \mathrm{d}y_1 & \mathrm{d}y_2 & \mathrm{d}x_1 & \mathrm{d}x_2 & \mathrm{d}x_3
\end{pmatrix} \begin{pmatrix}
 0 & 0 & \frac{1}{y_1} & 0 & 0 \\ 
-\frac{1}{4y_2} & 0 & \frac{1}{2y_2} & 0 & -\frac{1}{4y_2} \\
0 & -\frac{x_2+\sqrt{-1}y_2}{2y_1y_2} & 0 & \frac{x_2-\sqrt{-1}y_2}{2y_1y_2} & 0 \\
	       \frac{\sqrt{-1}}{4y_2} & 0 & 0 & 0 & -\frac{\sqrt{-1}}{4y_2} \\
0 & \frac{1}{2y_1y_2} & 0 & -\frac{1}{2y_1y_2} & 0
\end{pmatrix}. \label{eq:pullback_2}
\end{multline}

By taking wedge products of (\ref{eq:pullback_2}) and considering appropriate base change given in Section~\ref{sec:wedge_p}, we can explicitly describe the differential forms $\{\boldsymbol{\omega}^i_j\}_{-3\leq j\leq 3}$ introduced in Section~\ref{sec:wedge_p} with respect to the coordinates $(y_1,y_2,x_1,x_2,x_3)$ of $Y^{(3)}_{\mathcal{K}_3,\infty}$. We summarize the results in the following proposition.

\begin{prop}\label{prop:bfomega}
{\itshape 
 Set $\varsigma_1=\mathrm{d}y_1$, $\varsigma_2=\mathrm{d}y_2$, $\varsigma_3=\mathrm{d}x_1$, $\varsigma_4=\mathrm{d}x_2$ and $\varsigma_5=\mathrm{d}x_3$. For $i=2$ and $3$, define $\boldsymbol{\varsigma}^i$ as
\begin{align*}
  \boldsymbol{\varsigma}^{i}= 
\begin{cases}
 \begin{pmatrix}
  \varsigma_{2,1} & \varsigma_{2,0} & \varsigma_{2,-1} & \varsigma_{2,-2} & \varsigma_{1,0} & \varsigma_{1,-1} & \varsigma_{1,-2} & \varsigma_{0,-1} & \varsigma_{0,-2} & \varsigma_{-1,-2}
 \end{pmatrix} & \text{for $i=2$}, \\
\begin{pmatrix}
  \varsigma_{2,1,0} & \varsigma_{2,1,-1} & \varsigma_{2,1,-2} & \varsigma_{2,0,-1} & \varsigma_{2,0,-2} & \varsigma_{2,-1,-2} & \varsigma_{1,0,-1} & \varsigma_{1,0,-2} & \varsigma_{1,-1,-2} & \varsigma_{0,-1,-2}
 \end{pmatrix} & \text{for $i=3$}
\end{cases}
\end{align*}
where $\varsigma_{j,j'}$ and $\varsigma_{j,j',j''}$ abbreviate $\varsigma_j\wedge \varsigma_{j'}$ and $\varsigma_{j}\wedge \varsigma_{j'} \wedge \varsigma_{j''}$
respectively. Then the differential forms $
\begin{pmatrix}
 \boldsymbol{\omega}^i_3 & \dotsc & \boldsymbol{\omega}_{-3}^i
\end{pmatrix}$ is described on $Y^{(3)}_{\mathcal{K}_3,\infty}$ as $\boldsymbol{\varsigma}^i Q^i$ for 
\begin{align*}
 Q^2 &=
\begin{pmatrix}
 0 & \frac{1}{2y_1y_2} & 0 & 0 & 0 & -\frac{1}{2y_1y_2} & 0 \\
0 & 0 & \frac{x_2+\sqrt{-1}y_2}{2y_1^2y_2} & 0 & \frac{x_2-\sqrt{-1}y_2}{2y_1^2y_2} & 0 & 0 \\
0 & -\frac{\sqrt{-1}}{2y_1y_2} & 0 & 0 & 0 & -\frac{\sqrt{-1}}{2y_1y_2} & 0 \\
0 & 0 & -\frac{1}{2y_1^2y_2} & 0 & -\frac{1}{2y_1^2y_2} & 0 & 0 \\
\frac{x_2+\sqrt{-1}y_2}{8y_1y_2^2} & 0 & -\frac{x_2-5\sqrt{-1}y_2}{8y_1y_2^2} &0 &  -\frac{x_2+5\sqrt{-1}y_2}{8y_1y_2^2} & 0 & \frac{x_2-\sqrt{-1}y_2}{8y_1y_2^2} \\
0 & -\frac{\sqrt{-1}}{4y_2^2} & 0 & \frac{\sqrt{-1}}{2y_2^2} & 0 & -\frac{\sqrt{-1}}{4y_2^2} & 0 \\
-\frac{1}{8y_1y_2^2} & 0 & \frac{1}{8y_1y_2^2} & 0 & \frac{1}{8y_1y_2^2} & 0 & -\frac{1}{8y_1y_2^2} \\
\frac{\sqrt{-1}x_2-y_2}{8y_1y_2^2} & 0 & -\frac{3(\sqrt{-1}x_2+y_2)}{8y_1y_2^2} & 0 & \frac{3(\sqrt{-1}x_2-y_2)}{8y_1y_2^2} & 0 & -\frac{\sqrt{-1}x_2+y_2}{8y_1y_2^2} \\
0 & 0 & 0 & \frac{\sqrt{-1}}{y_1^2y_2} & 0 & 0 & 0 \\ 
\frac{\sqrt{-1}}{8y_1y_2^2} & 0 & -\frac{3\sqrt{-1}}{8y_1y_2^2} & 0 & \frac{3\sqrt{-1}}{8y_1y_2^2} & 0 & -\frac{\sqrt{-1}}{8y_1y_2^2}
\end{pmatrix}, \\
Q^3 &=
\begin{pmatrix}
\frac{x_2+\sqrt{-1}y_2}{24y_1^2y_2^2} & 0 & \frac{x_2-\sqrt{-1}y_2}{8y_1^2y_2^2} & 0 & \frac{x_2+\sqrt{-1}y_2}{8y_1^2y_2^2} & 0 & \frac{x_2-\sqrt{-1}y_2}{24y_1^2y_2^2} \\
0 & 0 & 0 & -\frac{\sqrt{-1}}{3y_1y_2^2} & 0 & 0 & 0 \\ 
-\frac{1}{24y_1^2y_2^2} & 0 & -\frac{1}{8y_1^2y_2^2} & 0 & -\frac{1}{8y_1^2y_2^2} & 0 & -\frac{1}{24y_1^2y_2^2} \\
\frac{\sqrt{-1}x_2-y_2}{24y_1^2y_2^2} & 0 & \frac{\sqrt{-1}x_2+y_2}{8y_1^2y_2^2} & 0 & -\frac{\sqrt{-1}x_2-y_2}{8y_1^2y_2^2} & 0 & -\frac{\sqrt{-1}x_2+y_2}{24y_1^2y_2^2} \\
0 & 0 & 0 & -\frac{\sqrt{-1}}{6y_1^3y_2} & 0 & 0 & 0 \\
\frac{\sqrt{-1}}{24y_1^2y_2^2} & 0 & \frac{\sqrt{-1}}{8y_1^2y_2^2} & 0 & -\frac{\sqrt{-1}}{8y_1^2y_2^2} & 0 & -\frac{\sqrt{-1}}{24y_1^2y_2^2} \\
\frac{\sqrt{-1}x_2-y_2}{48y_1y_2^3} & 0 & -\frac{\sqrt{-1}x_2-3y_2}{16y_1y_2^3} & 0 & \frac{\sqrt{-1}x_2+3y_2}{16y_1y_2^3} & 0 & -\frac{\sqrt{-1}x_2+y_2}{48y_1y_2^3} \\
 0 & -\frac{\sqrt{-1}}{8y_1^2y_2^2} & 0 & -\frac{\sqrt{-1}}{12y_1^2y_2^2} & 0 & -\frac{\sqrt{-1}}{8y_1^2y_2^2} & 0 \\
 \frac{\sqrt{-1}}{48y_1y_2^3} & 0 & -\frac{\sqrt{-1}}{16y_1y_2^3} & 0 & \frac{\sqrt{-1}}{16y_1y_2^3} & 0 &  -\frac{\sqrt{-1}}{48y_1y_2^3} \\
0 & \frac{1}{8y_1^2y_2^2} & 0 & 0 & 0 & -\frac{1}{8y_1^2y_2^2} & 0
\end{pmatrix}.
\end{align*}
}
\end{prop}

\section{Appendix: Formulas for ${\mathcal P}(X, Y, Z, A, B, C)$}\label{sec:AppB}

We retain the notations in Sections \ref{sec:branch} and \ref{sec:ES}.   
In Appendix \ref{sec:AppB},    
   we summarize several formulas concerning the matrix introduced in (\ref{eq:matP}) or Lemma~\ref{lem:construction_delta} (i):
\begin{align*}   
      {\mathcal P}(X, Y, Z, A, B, C)  
   = \left(   
            {\mathcal P}_{\alpha, \beta} (X, Y, Z, A, B, C) 
       \right)_{\substack{\lambda_3\geq \alpha\geq -\lambda_3 \\ 3\geq \beta\geq -3}} 
      \in {\rm M}_{2\lambda_3+1, 7} ( {\mathbf C}[X,Y,Z ; A,B,C]_{w_{\boldsymbol{\lambda}}}  ).
\end{align*}

\begin{lem}\label{lem:Pab}
{\itshape For each $-\lambda_3 \leq \alpha \leq \lambda_3$ and  $-3\leq \beta \leq 3$,   we have 
\begin{align*}
   {\mathcal P}_{ \alpha,  \beta} (X,Y, Z, A,B,C)
   &=      2^{ \alpha-\beta }
               \sum_{ \substack{  - \alpha  -a+2j+k = w_{\boldsymbol{\lambda}},  \\  a + \beta+2b+c = w_{\boldsymbol{\lambda}} }  }
         (-2^{-2})^{ j+b } 
                         {}_{w_{\boldsymbol{\lambda}}} H_{j-a,j-\alpha,k}
                         \cdot {}_{w_{\boldsymbol{\lambda}}} H_{b+a,b+\beta,c}       \\
  &\qquad      \times   \left(                            
                     (-X+\sqrt{-1}Y)^{\alpha -a} 
                     (X^2+Y^2)^{j-\alpha}
                     Z^{k}         \right)     
       \otimes 
               \left(  (-A+\sqrt{-1}B)^{a-\beta } 
                        (A^2+B^2)^{b+\beta}
                          C^{c}    
                       \right),  
\end{align*}
where ${}_eH_{i,j,k} \in {\mathbf Z}$ is defined as in \rm{(\ref{eq:defH})}.  
}
\end{lem}
\begin{proof}
Using the identities
\begin{align*}
 z_1X+z_2Y+z_3Z&=-\dfrac{1}{2}(z_1+\sqrt{-1}z_2)(-X+\sqrt{-1}Y)-\dfrac{1}{2}(-z_1+\sqrt{-1}z_2)(X+\sqrt{-1}Y)+z_3Z,\\
 z_1A+z_2B+z_3C&=-\dfrac{1}{2}(z_1+\sqrt{-1}z_2)(-A+\sqrt{-1}B)-\dfrac{1}{2}(-z_1+\sqrt{-1}z_2)(A+\sqrt{-1}B)+z_3C,
\end{align*}
we expand $(z_1X+z_2Y+z_3Z)^{w_{\boldsymbol{\lambda}}}\otimes (z_1A+z_2B+z_3C)^{w_{\boldsymbol{\lambda}}}$ as  
\begin{multline*}
        (z_1 X + z_2 Y + z_3 Z)^{w_{\boldsymbol{\lambda}}}  \otimes    (z_1 A + z_2 B + z_3 C)^{w_{\boldsymbol{\lambda}}}  \\ 
\begin{aligned}
\qquad    =  \sum_{ \substack{ i+j+k = w_{\boldsymbol{\lambda}}  \\  a+b+c = w_{\boldsymbol{\lambda}} }  }   
          {}_{w_{\boldsymbol{\lambda}}} H_{i,j,k} 
          \cdot {}_{w_{\boldsymbol{\lambda}}} H_{a,b,c} 
     &     \left(   (-\frac{1}{2})^{i+j}  
                     (z_1 + \sqrt{-1} z_2)^i   (-X+\sqrt{-1}Y)^i 
                      (-z_1 + \sqrt{-1} z_2)^j   (X+\sqrt{-1}Y)^j    
                      z^k_3  Z^k         \right)     \\
      &   \otimes 
               \left( (-\frac{1}{2})^{a+b}  
                     (z_1 + \sqrt{-1} z_2)^a   (-A+\sqrt{-1}B)^a 
                      (-z_1 + \sqrt{-1} z_2)^b   (A+\sqrt{-1}B)^b   
                      z^c_3  C^c    
                       \right).  
\end{aligned}
\end{multline*}
Changing the variables $i$ to $i+j$ and $a$ to $a+b$, 
the expanded polynomial above is calculated as
\begin{multline*}
\begin{aligned}
    \sum_{ \substack{ i+2j+k = w_{\boldsymbol{\lambda}}  \\  a+2b+c = w_{\boldsymbol{\lambda}} }  }
         (-1)^{i+j+a+b} 
           2^{-(i+2j + a + 2b)}
                  {}_{w_{\boldsymbol{\lambda}}} H_{i+j,j,k}
                  \cdot {}_{w_{\boldsymbol{\lambda}}} H_{a+b,b,c} &
           \left(     
                      (z_1 + \sqrt{-1} z_2)^{i}   
                     (-X+\sqrt{-1}Y)^{i} 
                     (X^2+Y^2)^j
                      z^{2j+k}_3  Z^{k}         \right)     \\
      &  \otimes 
               \left(  (z_1 + \sqrt{-1} z_2)^{a}  
                        (-A+\sqrt{-1}B)^{a} 
                        (A^2+B^2)^b
                      z^{2b+c}_3  C^{c}    
                       \right)                 
\end{aligned}        \\
 \begin{aligned}
  =   \sum_{ \substack{ i+2j+k = w_{\boldsymbol{\lambda}}  \\  a+2b+c = w_{\boldsymbol{\lambda}} }  }
         (-1)^{i+j+a+b} 
           2^{-(i+2j + a + 2b)}  &
                  {}_{w_{\boldsymbol{\lambda}}} H_{i+j,j,k} 
                  \cdot {}_{w_{\boldsymbol{\lambda}}} H_{a+b,b,c}  
             (z_1 + \sqrt{-1} z_2)^{i+a} z^{2 w_{\boldsymbol{\lambda}} -i-a}_3    \\
        &  \times  \left(                            
                     (-X+\sqrt{-1}Y)^{i} 
                     (X^2+Y^2)^j
                     Z^{k}         \right)     
        \otimes 
               \left(  (-A+\sqrt{-1}B)^{a} 
                        (A^2+B^2)^b
                          C^{c}    
                       \right)  
 \end{aligned}  \\
\begin{aligned}
 \substack{ i \mapsto i-a \\  =}
    \sum_{ \substack{ i-a+2j+k = w_{\boldsymbol{\lambda}}  \\  a+2b+c = w_{\boldsymbol{\lambda}} }  }
         (-1)^{i+j+b} 
           2^{-(i+2j  + 2b)}  
   &               {}_{w_{\boldsymbol{\lambda}}} H_{i+j-a,j,k}
                  \cdot {}_{w_{\boldsymbol{\lambda}}} H_{a+b,b,c}  
                  v^{(2w_{\boldsymbol{\lambda}},0)}_i   \\
    &   \times   \left(                            
                     (-X+\sqrt{-1}Y)^{i-a} 
                     (X^2+Y^2)^j
                     Z^{k}         \right)     
       \otimes 
               \left(  (-A+\sqrt{-1}B)^{a} 
                        (A^2+B^2)^b
                          C^{c}     \right).
\end{aligned}                      
\end{multline*}
Here we use the definition $    v^{(\lambda,\delta)}_{\pm \mu}
    =  ( \pm z_1+\sqrt{-1} z_2 )^{\mu}  z^{\lambda - \mu}_3  \in V^{(3)}_{\boldsymbol{\lambda}}$ for $0\leq \mu\leq \lambda$ (see Section~\ref{sec:repSO3} for details). 
Hence, for each $-3\leq \beta \leq 3$, we compute
\begin{multline*}
      (z_1 X + z_2 Y + z_3 Z)^{w_{\boldsymbol{\lambda}}}  \otimes    (z_1 A + z_2 B + z_3 C)^{w_{\boldsymbol{\lambda}}}  v^{(3,\delta)}_{  \beta }\\
\begin{aligned}
=  \sum_{ \substack{ i-a+2j+k = w_{\boldsymbol{\lambda}}  \\  a+2b+c = w_{\boldsymbol{\lambda}} }  }
         (-1)^{i+j+b} 
           2^{-(i+2j  + 2b)} & 
                             {}_{w_{\boldsymbol{\lambda}}} H_{i+j-a,j,k}  
                             \cdot {}_{w_{\boldsymbol{\lambda}}} H_{a+b,b,c} 
               v^{(2w_{\boldsymbol{\lambda}}+3,\delta)}_{i+\beta}       \\        
  &       \times  \left(                            
                     (-X+\sqrt{-1}Y)^{i-a} 
                     (X^2+Y^2)^j
                     Z^{k}         \right)     
       \otimes 
               \left(  (-A+\sqrt{-1}B)^{a} 
                        (A^2+B^2)^b
                          C^{c}    
                       \right) \end{aligned}  \\                       
\begin{aligned}
 \stackrel{ \substack{  i\mapsto \alpha-\beta   \\  
                                  a \mapsto a-\beta   } }{=} 
\sum_\alpha   
v^{\boldsymbol{\lambda}}_\alpha  
  \sum_{ \substack{ \alpha-a+2j+k = w_{\boldsymbol{\lambda}}  \\  a-\beta+2b+c = w_{\boldsymbol{\lambda}} }  } 
         (-1)^{\alpha-\beta+j+b} &
           2^{-(\alpha-\beta+2j  + 2b)}  
                             {}_{w_{\boldsymbol{\lambda}}} H_{\alpha + j-a, j, k}  
                             \cdot {}_{w_{\boldsymbol{\lambda}}} H_{a-\beta+b,b,c}      \\        
         &  \times \left(                            
                     (-X+\sqrt{-1}Y)^{\alpha-a} 
                     (X^2+Y^2)^j
                     Z^{k}         \right)     
       \otimes 
               \left(  (-A+\sqrt{-1}B)^{a-\beta} 
                        (A^2+B^2)^b
                          C^{c}    
                       \right)  
\end{aligned}
   \\
\begin{aligned}
 \stackrel{ \substack{  j \mapsto j - \alpha   \\  
                                  b \mapsto b + \beta   } }{=} 
\sum_\alpha   
v^{\boldsymbol{\lambda}}_\alpha  
  \sum_{ \substack{ -\alpha-a+2j+k = w_{\boldsymbol{\lambda}}  \\  a+\beta+2b+c = w_{\boldsymbol{\lambda}} }  }&
         (-1)^{ j+b } 
           2^{-(-\alpha + \beta+2j  + 2b)}  
                             {}_{w_{\boldsymbol{\lambda}}} H_{ j-a, j-\alpha, k}  
                             \cdot {}_{w_{\boldsymbol{\lambda}}} H_{b+a, b+\beta,c}      \\        
     &   \times   \left(                            
                     (-X+\sqrt{-1}Y)^{\alpha-a} 
                     (X^2+Y^2)^{j-\alpha}
                     Z^{k}         \right)     
       \otimes 
               \left(  (-A+\sqrt{-1}B)^{a-\beta} 
                        (A^2+B^2)^{b + \beta}
                          C^{c}    
                       \right)
\end{aligned}
\end{multline*}
by using multiplicativity (\ref{eq:vmult}). Comparing with the definition of $\mathcal{P}(X,Y,Z,A,B,C)$, we find the statement.
\end{proof}

The image of ${\mathcal P}_{\alpha, \beta}$ under the differential operator $\nabla^{\boldsymbol{n}}$ in (\ref{eq:nabkl}) is computed as follows:

\begin{lem}\label{lem:AppB2}
{\itshape
Let $S=-X+\sqrt{-1}Y$ and $T= X+\sqrt{-1}Y$.  
Then we have 
\begin{multline}\label{eq:sumST1}
        \nabla^{\boldsymbol{n}} {\mathcal P}_{\alpha, \beta} (X, Y)
  = 2^{-n_1}(-1)^{w_{\boldsymbol{\lambda}}+\alpha}\sqrt{-1}^{n_2+w_{\boldsymbol{\lambda}}}
           S^{\frac{n_1+\alpha -\beta}{2}} T^{ \frac{n_1- \alpha + \beta}{2}}  \\     \times   \sum_{ b\in \mathbf{Z}} (-1)^b
                         {}_{w_{\boldsymbol{\lambda}}} H_{b+\frac{n_1+\alpha+\beta}{2}+n_2-w_{\boldsymbol{\lambda}},-b+\frac{n_1-\alpha-\beta}{2},2w_{\boldsymbol{\lambda}}-n_1-n_2}
                         \cdot  {}_{w_{\boldsymbol{\lambda}}} H_{-(b+\beta)+w_{\boldsymbol{\lambda}}-n_2, b+\beta, n_2}.
\end{multline}

}
\end{lem}
\begin{proof}
Let $\nabla_{k,l}$ be the differential operator which is introduced in (\ref{eq:defnab}).
By using Lemma \ref{lem:Pab}, we find that 
\begin{align*}
   \left(   2^{ \alpha-\beta } \right)^{-1}  &
  \nabla_{k, l} {\mathcal P}_{\alpha, \beta}(X,Y,Z,A,B,C)    \\
  & =  \left(   2^{ \alpha-\beta } \right)^{-1} 
   \frac{1}{k! \, l!}  \frac{\partial^{k} }{ \partial Z^k}  \otimes \frac{\partial^{l} } {\partial C^l  }
     {\mathcal P}_{\alpha, \beta}(X,Y,Z,A,B,C) |_{ A=-Y,  B=X,  Z=C=0  }  \\
&=       
\sum_{ \substack{  - \alpha  -a+2j = w_{\boldsymbol{\lambda}}-k  \\  a + \beta+2b = w_{\boldsymbol{\lambda}} -l}  }
         (-2^{-2})^{ j+b } 
                         {}_{w_{\boldsymbol{\lambda}}} H_{j-a,j-\alpha,k} 
                         \cdot {}_{w_{\boldsymbol{\lambda}}} H_{b+a,b+\beta, l}       \\
  &  \hspace*{6em}    \times   \left(                            
                     (-X+\sqrt{-1}Y)^{\alpha -a} 
                     (X^2+Y^2)^{j-\alpha}
                            \right)     
       \otimes 
               \left(  (-A+\sqrt{-1}B)^{a-\beta } 
                        (A^2+B^2)^{b+\beta}
                       \right)|_{ A=-Y,  B=X   }  \\
&=     
\sum_{ \substack{  - \alpha  -a+2j = w_{\boldsymbol{\lambda}}-k  \\  a + \beta+2b = w_{\boldsymbol{\lambda}} -l }  }
         (-2^{-2})^{ j+b } 
                         {}_{w_{\boldsymbol{\lambda}}} H_{j-a,j-\alpha,k}
                         \cdot  {}_{w_{\boldsymbol{\lambda}}} H_{b+a,b+\beta, l}       \\
  &    \hspace*{15em}     \times        (-X+\sqrt{-1}Y)^{\alpha -a} 
                     (X^2+Y^2)^{j-\alpha}      
               ( Y +\sqrt{-1}X)^{a-\beta } 
                        (X^2+Y^2)^{b+\beta}.
 \end{align*}
 Put $S=-X+\sqrt{-1}Y$ and $T=X+\sqrt{-1}Y$. Then the above equation is calculated as 
 \begin{align*} 
 \sum_{ \substack{  - \alpha  -a+2j = w_{\boldsymbol{\lambda}}-k  \\  a + \beta+2b = w_{\boldsymbol{\lambda}} -l }  }
         (2^{-2})^{ j+b }  (-1)^{a-\alpha} \sqrt{-1}^{a-\beta}
                         {}_{w_{\boldsymbol{\lambda}}} H_{j-a,j-\alpha,k}
                         \cdot  {}_{w_{\boldsymbol{\lambda}}} H_{b+a,b+\beta, l}     S^{j+b} T^{j+b-\alpha+\beta}.
 \end{align*}
Now substitute $(k,l)=(2w_{\boldsymbol{\lambda}}-n_1-n_2,n_2)$ as in Section~\ref{sec:branch}. 
Since 
 \begin{align*}
 2j+2b-\alpha+\beta = 2w_{\boldsymbol{\lambda}}-k-l = 2w_{\boldsymbol{\lambda}}- (2w_{\boldsymbol{\lambda}}-n_1-n_2) - n_2 = n_1, 
\end{align*}
we see that  $j+b= \frac{n_1+\alpha -\beta}{2}$ holds. Moreover we have  
\begin{align*}
   \begin{cases} - \alpha  -a+2j = w_{\boldsymbol{\lambda}}-k=n_1+n_2-w_{\boldsymbol{\lambda}},   \\
                        a + \beta+2b = w_{\boldsymbol{\lambda}}-l=w_{\boldsymbol{\lambda}}-n_2,  \end{cases} 
   \Longleftrightarrow \quad 
   \begin{cases} a = -2b-\beta - n_2 +w_{\boldsymbol{\lambda}}   \\
                        j = -b + \frac{n_1+\alpha-\beta}{2}. \end{cases} 
\end{align*} 
Substituing these, we obtain the statement.  
\end{proof}

The following lemma is a special case that $\alpha - \beta = \pm n_1$ in Lemma \ref{lem:AppB2}.

\begin{lem}\label{lem:AppB3}
{\itshape 
We find that   
\begin{align*}
        \nabla^{\boldsymbol{n}} {\mathcal P}_{\alpha, \beta} (X, Y)
  &=  (-2^{-1})^{ n_1 } 
       \sqrt{-1}^{  n_2 + w_{\boldsymbol{\lambda}} } 
         \binom{w_{\boldsymbol{\lambda}}}{ n_1+n_2-w_{\boldsymbol{\lambda}} }  
        \binom{w_{\boldsymbol{\lambda}}}{ w_{\boldsymbol{\lambda}}-n_2 }   
      \times \begin{cases}  (-1)^{n_2} T^{n_1}  &   \text{when } \alpha - \beta =- n_1,  \\                                                  
                              (-1)^{w_{\boldsymbol{\lambda}}}             S^{n_1}  &    \text{when }\alpha - \beta = n_1.  
                                          \end{cases}          
\end{align*}
}
\end{lem}
\begin{proof}
Suppose that $\alpha - \beta = -n_1$. Substituting this, we find that (\ref{eq:sumST1})  becomes 
 \begin{align*}
     2^{n_1}& \sqrt{-1}^{-(n_2+w_{\boldsymbol{\lambda}})}   \nabla^{\boldsymbol{n}} {\mathcal P}_{\alpha, \beta} (X, Y)  \\
&= (-1)^{w_{\boldsymbol{\lambda}}}  T^{n_1}  \sum_{b\in \mathbf{Z}}  (-1)^{b+\alpha}
                         {}_{w_{\boldsymbol{\lambda}}} H_{b+\frac{n_1+\alpha+\beta}{2}+n_2-w_{\boldsymbol{\lambda}}, -b+\frac{n_1-\alpha-\beta}{2},2w_{\boldsymbol{\lambda}}-n_1-n_2}
                         \cdot  {}_{w_{\boldsymbol{\lambda}}} H_{-(b+\beta)+w_{\boldsymbol{\lambda}}-n_2,b+\beta, n_2}   \\
&\stackrel{ b\mapsto b-\alpha }{=}
 (-1)^{w_{\boldsymbol{\lambda}}} T^{n_1} 
       \sum_{b\in \mathbf{Z}}  (-1)^{b}   {}_{w_{\boldsymbol{\lambda}}} H_{b+\frac{n_1-(\alpha-\beta)}{2}+n_2-w_{\boldsymbol{\lambda}}, -b+\frac{n_1+(\alpha-\beta)}{2},2w_{\boldsymbol{\lambda}}-n_1-n_2}
                         \cdot  {}_{w_{\boldsymbol{\lambda}}} H_{-b+(\alpha-\beta)+w_{\boldsymbol{\lambda}}-n_2,b-(\alpha-\beta), n_2}   \\
&=(-1)^{w_{\boldsymbol{\lambda}}} T^{n_1} 
       \sum_{b\in \mathbf{Z}}  (-1)^{b}   {}_{w_{\boldsymbol{\lambda}}} H_{b+n_1+n_2-w_{\boldsymbol{\lambda}}, -b,2w_{\boldsymbol{\lambda}}-n_1-n_2}
                         \cdot  {}_{w_{\boldsymbol{\lambda}}} H_{-(b+n_1+n_2-w_{\boldsymbol{\lambda}}),b+n_1, n_2}.
 \end{align*}
This summation does not vanish if and only if $b = w_{\boldsymbol{\lambda}}-n_1-n_2$ holds. 
In this case the above summation becomes 
\begin{align*}
    (-1)^{w_{\boldsymbol{\lambda}}}& T^{n_1} 
       \sum_{b\in \mathbf{Z}}  (-1)^{b}   {}_{w_{\boldsymbol{\lambda}}} H_{b+n_1+n_2-w_{\boldsymbol{\lambda}}, -b,2w_{\boldsymbol{\lambda}}-(n_1+n_2)}
                         \cdot  {}_{w_{\boldsymbol{\lambda}}} H_{-(b+n_1+n_2-w_{\boldsymbol{\lambda}}),b+n_1, n_2} \\
& =    (-1)^{w_{\boldsymbol{\lambda}}}  T^{n_1} 
           (-1)^{  w_{\boldsymbol{\lambda}}-n_1-n_2 } 
                         {}_{w_{\boldsymbol{\lambda}}} H_{0, n_1+n_2-w_{\boldsymbol{\lambda}}, 2w_{\boldsymbol{\lambda}}-n_1-n_2}
                         \cdot  {}_{w_{\boldsymbol{\lambda}}} H_{0, w_{\boldsymbol{\lambda}}-n_2, n_2}     \\
& =   (-1)^{n_1+n_2}   T^{n_1} 
        \binom{w_{\boldsymbol{\lambda}}}{ n_1+n_2-w_{\boldsymbol{\lambda}} }  
        \binom{w_{\boldsymbol{\lambda}}}{ w_{\boldsymbol{\lambda}}-n_2 }.
\end{align*}
This proves the statement in the case where $\alpha - \beta = -n_1$ holds.  The case where $\alpha-\beta=n_1$ holds is verified similarly, but in this case, we should substitute $b=n_1$ in the last step, which causes the difference of the signature.
\end{proof}

\section*{Acknowledgements} 

The authors are grateful to Miki Hirano, Taku Ishii and Tadashi Miyazaki for kindly sending the authors their preprints  \cite{him} and \cite{im}, which motivate this research project and play crucial roles in the present article. 
This work was partially supported by the Research Institute for Mathematical Sciences, a Joint Usage/Research Center located in Kyoto University (RIMS satellite meeting ``Special values of automorphic $L$-functions and their $p$-adic properties'' held at Miyama-cho, Kyoto in 2016).  
The first (resp.\ second) author  is supported 
       by JSPS Grant-in-Aid for Young Scientists Grant Number JP18K13395  (resp.\ for Young Scientists (B)  Grant Number JP17K14174).


\begin{thebibliography}{9999999}


\bibitem[BW80]{bw80}
 Armand Borel and Nolan R.~Wallach, 
 {\itshape Continuous cohomology, discrete subgroups and representations of reductive groups}, 
 Ann.\ of Math.\ Stud., {\bf 94}, Princeton Univ.\ Press, Princeton, 1980.

\bibitem[Clo90]{clo90}
Laurent Clozel, 
{\itshape Motifs et formes automorphes$:$ applications du principe, de fonctorialit$\acute{ e}$}, 
in: {\em Automorphic forms, Shimura varieties and $L$-functions} (L.~Clozel et J.S.~Milne edit.), vol.~1, 77--159, Acad.\ Press (1990).

\bibitem[Coa89]{coa89}
John Coates, 
{\itshape On $p$-adic $L$-functions attached to motives over ${\mathbf Q}$}: II,  
Bol.\ Soc.\ Bras.\ Mat., {\bf 20} (1989), 101--112.
 

\bibitem[CP89]{cp89}
John Coates and Bernadette Perrin-Riou, 
{\itshape On $p$-adic $L$-functions attached to motives over ${\mathbf Q}$}, in: {\em Algebraic Number Theory ---in honor of K.~Iwasawa} (J.~Coates, R.~Greenberg, B.~Mazur and I.~Satake eds.), Adv.\ Stud.\ Pure Math., {\bf 17} (1989), 23--54.


\bibitem[Del79]{del79}
Pierre Deligne, 
{\itshape Valeurs de fonctions $L$ et p\'eriodes d'int\'egrales}, 
in: {\itshape Automorphic Forms,  Representations, and $L$-functions} (Proc.\ Sympos.\ Pure Math., Oregon State Univ., Corvallis, Ore., 1977), 
Proc.\ Symp.\ Pure Math., {\bf 33} Part II, Amer.\ Math.\ Soc., Providence, (1979), 247--289.

\bibitem[FH13]{fh13}
William Fulton and Joe Harris, 
{\itshape Representation Theory$:$ A First Course}, 
Graduate Texts in Mathematics, {\bf 129}, Springer--Verlag (2013). 


\bibitem[GJ72]{gj72} Roger Godement and Herv\'e Jacquet, {\em Zeta functions of simple algebras}, Lecture Note in Math., \textbf{260}, Springer--Verlag (1972), ix+188 p.p. 

\bibitem[GW99]{gw99}
Roe Goodman and Nolan R.~Wallach, 
{\em Representations and Invariants of the classical groups}, 
Encyclopedia of Mathematics and its Applications, Cambridge University Press (1999). 


\bibitem[GR14]{gr14}
Harald Grobner and Anantharam Raghuram, 
{\itshape On some arithmetic properties of automorphic forms of ${\rm GL}_m$ over a division algebra}, 
Int.\ J.\ Number Theory, {\bf 10} (2014), 963--1013.

\bibitem[Hid94]{hid94}
{Haruzo Hida}, 
{\itshape On the critical values of $L$-functions of ${\rm GL}(2)$ and ${\rm GL}(2)  \times {\rm GL}(2)$}, 
Duke Math.\ J., {\bf 74} (1994), 431--529.

\bibitem[HIM16]{him16}
Miki Hirano, Taku Ishii and Tadashi Miyazaki, 
{\itshape The archimedean zeta integrals for ${\rm GL}(3) \times {\rm GL}(2)$}, 
Proc.\ Japan Acad.\ Ser.~A, {\bf 92} (2016), 27--32.

\bibitem[HIM]{him}
Miki Hirano, Taku Ishii and Tadashi Miyazaki, 
{\itshape Archimedean zeta integrals for ${\rm GL}(3) \times {\rm GL}(2)$}, 
to appear in Memoirs Amer.\ Math.\ Soc.

\bibitem[IM]{im}
Taku Ishii and Tadashi Miyazaki, 
{\itshape Calculus of archimedean Rankin--Selberg integrals with recurrence relations}, 
preprint available at arXiv:2006.04095.

\bibitem[Jan19]{jan19}
Fabian Januszewski,  
{\itshape On period relations for automorphic $L$-functions I}, 
Trans.\ Amer.\ Math.\ Soc., {\bf 371} (2019), 6547--6580.

\bibitem[Jan]{jan}
Fabian Januszewski,  
   {\itshape Non-abelian $p$-adic Rankin--Selberg $L$-functions and non-vanishing of central $L$-values}, 
   preprint available at arXiv:1708.02616.


\bibitem[JPSS81]{jpss81}
Herve Jacquet, Ilya I.~Piatetski-Shapiro and Joseph Shalika,  
{\itshape Conducteur des repr\'{e}sentations du groupe lin\'{e}aire}, 
Math.\ Ann., {\bf 256} (1981), 199--214.

\bibitem[JPSS83]{jpss83}
Herve Jacquet, Ilya I.~Piatetski-Shapiro and Joseph Shalika,  
{\itshape Rankin--Selberg convolutions}, 
Amer.\ J.\ Math., {\bf 105} (1983), 367--464.

\bibitem[KS13]{ks13}
Hendrik Kasten and Claus-G\"unter Schmidt,  
{\itshape The critical values of Rankin--Selberg convolutions}, 
Int.\ J.\ Number Theory, {\bf 9} (2013), 205--256.

\bibitem[KMS00]{kms00}
David Kazhdan, Barry Mazur and Claus-G\"unter Schmidt,  
{\itshape Relative modular symbols and Rankin--Selberg convolutions},  
J.\ Reine Angew.\ Math., {\bf 519} (2000), 97--141.


\bibitem[Kna94]{kna94}
Anthony W.~Knapp, 
{\itshape Local Langlands correspondence$:$ the archimedean case},
Proceedings of symposia in pure mathematics, {\bf 55}, part 2, (1994), 393--410. 

\bibitem[Mah05]{mah05} 
Joachim Mahnkopf, 
{\itshape Cohomology of arithmetic groups, parabolic subgroups and the special values of $L$-functions on ${\rm GL}_n$},  
J.\ Inst.\ Math.\  Jussieu, {\bf 4} (2005), 553--637.

\bibitem[Man72]{man72} 
Juri~I.~Manin, {\itshape Parabolic points and zeta-functions of modular curves}, Math USSR IZV., {\bf 6} (1), (1972), 19--64.

\bibitem[RS08]{rs08}
A. Raghuram and Freydoon Shahidi, 
{\itshape On certain period relations for cusp forms on ${\rm GL}_n$},   
Int.\ Math.\ Res.\ Not., {\bf 2008} (2008), 1--23.

\bibitem[Rag16]{rag16}
A. Raghuram, 
{\itshape Critical values of Rankin--Selberg L-functions for ${\rm GL}_n \times {\rm GL}_{n-1}$ and the symmetric cube L-functions for ${\rm GL}_2$},   
Forum.\ Math., {\bf 28} (2016), 457--489.

\bibitem[Shim76]{shi76} 
Goro Shimura, {\itshape The special values of the zeta functions associated with cusp forms}, Comm.\ Pure Appl.\ Math., {\bf 29} (1976), 783--804.

\bibitem[Shin76]{shin76} 
Takuro Shintani, {\itshape On an explicit formula for class-$1$ {\rm ``}Whittaker functions{\rm ''} on $\mathrm{GL}_n$ over $\mathfrak{P}$-adic fields}, Proc.\ Japan Acad., \textbf{52} (1976), 180--182. 

\bibitem[Sun17]{sun17}
Binyong Sun, 
{\itshape The nonvanishing hypothesis at infinity for Rankin--Selberg convolutions},   
 J.\ Amer.\ Math.\ Soc.  {\bf 30} (2017), 1--25.

\bibitem[Vat99]{vat99}
Vinayak Vatsal, 
{\itshape Canonical periods and congruence formulae},   
Duke Math.\ J.,  {\bf 98} (1999), 397--419.

\bibitem[Yos01]{yos01}
Hiroyuki Yoshida, 
{\itshape Motives and Siegel modular forms},     
Amer.\ J.\ Math.,  {\bf 123} (2001), 1171--1197.

\end{thebibliography}
\end{document}